\documentclass{amsart}
\usepackage{amssymb}
\usepackage{mathrsfs}
\usepackage[all]{xy}
\usepackage{hyperref}

\usepackage{graphicx}
\usepackage{caption}
\usepackage{subcaption}

\newcommand{\C}{\mathbb{C}}
\newcommand{\Ql}{\mathbb{Q}_\ell}
\newcommand{\Zl}{\mathbb{Z}_\ell}
\newcommand{\Fl}{\mathbb{F}_\ell}
\newcommand{\Qlb}{\overline{\mathbb{Q}}_\ell}
\newcommand{\Z}{\mathbb{Z}}
\newcommand{\N}{\mathbb{N}}

\newcommand{\K}{\mathbb{K}}
\newcommand{\F}{\mathbb{F}}
\renewcommand{\O}{\mathbb{O}}

\newcommand{\bk}{\Bbbk}

\newcommand{\fg}{\mathfrak{g}}
\newcommand{\fz}{\mathfrak{z}}
\newcommand{\fp}{\mathfrak{p}}
\newcommand{\fl}{\mathfrak{l}}
\newcommand{\fu}{\mathfrak{u}}

\newcommand{\cN}{\mathscr{N}}
\newcommand{\cO}{\mathscr{O}}
\newcommand{\cC}{\mathscr{C}}
\newcommand{\fL}{\mathfrak{L}}
\newcommand{\fN}{\mathfrak{N}}

\newcommand{\Db}{D^{\mathrm{b}}}
\newcommand{\Perv}{\mathsf{Perv}}

\newcommand{\Rep}{\mathsf{Rep}}
\newcommand{\Loc}{\mathsf{Loc}}

\newcommand{\bT}{\mathbb{T}}
\newcommand{\bS}{\mathbb{S}}
\newcommand{\Res}{\mathbf{R}}
\newcommand{\Ind}{\mathbf{I}}

\newcommand{\uInd}{\underline{\Ind}}

\newcommand{\cL}{\mathcal{L}}
\newcommand{\cM}{\mathcal{M}}
\newcommand{\IC}{\mathcal{IC}}
\newcommand{\cF}{\mathcal{F}}
\newcommand{\cG}{\mathcal{G}}

\newcommand{\ubk}{\underline{\bk}}

\newcommand{\cE}{\mathcal{E}}
\newcommand{\cS}{\mathcal{S}}
\newcommand{\cD}{\mathcal{D}}

\newcommand{\simto}{\xrightarrow{\sim}}
\DeclareMathOperator{\End}{End}
\DeclareMathOperator{\Hom}{Hom}

\DeclareMathOperator{\Irr}{Irr}

\newcommand{\cusp}{\mathrm{cusp}}

\makeatletter
\def\lotimes{\@ifnextchar_{\@lotimessub}{\@lotimesnosub}}
\def\@lotimessub_#1{\mathchoice{\mathbin{\mathop{\otimes}^L}_{#1}}%
  {\otimes^L_{#1}}{\otimes^L_{#1}}{\otimes^L_{#1}}}
\def\@lotimesnosub{\mathbin{\mathop{\otimes}^L}}
\makeatother

\newcommand{\GL}{\mathrm{GL}}

\newcommand{\SL}{\mathrm{SL}}
\newcommand{\Sp}{\mathrm{Sp}}
\newcommand{\Spin}{\mathrm{Spin}}
\newcommand{\SO}{\mathrm{SO}}
\newcommand{\G}{{\mathrm{G}}}
\newcommand{\fS}{\mathfrak{S}}

\newcommand{\Psico}{\Psi^{\mathrm{co}}}
\newcommand{\psico}{\psi^{\mathrm{co}}}
\newcommand{\Xico}{\Xi^{\mathrm{co}}}
\newcommand{\xico}{\xi^{\mathrm{co}}}
\newcommand{\Omegaco}{\Omega^{\mathrm{co}}}
\newcommand{\omegaco}{\omega^{\mathrm{co}}}
\newcommand{\omegacop}{\omega^{\mathrm{co}\prime}}
\newcommand{\Part}{\mathrm{Part}}
\newcommand{\Bipart}{\mathrm{Bipart}}
\newcommand{\Comp}{\N^\infty}
\newcommand{\sa}{\mathsf{a}}
\newcommand{\sfb}{\mathsf{b}}
\newcommand{\sm}{\mathsf{m}}
\newcommand{\uPart}{\underline{\Part}}
\newcommand{\uBipart}{\underline{\Bipart}}

\newcommand{\blambda}{{\boldsymbol{\lambda}}}
\newcommand{\tr}{\mathrm{t}}
\newcommand{\ve}{\mathrm{ve}}

\newtheorem*{thm*}{Theorem}
\numberwithin{equation}{section}
\newtheorem{thm}{Theorem}[section]
\newtheorem{lem}[thm]{Lemma}
\newtheorem{prop}[thm]{Proposition}
\newtheorem{cor}[thm]{Corollary}
\newtheorem{stmt}[thm]{Statement}

\theoremstyle{definition}

\theoremstyle{remark}
\newtheorem{rmk}[thm]{Remark}

\DeclareFontFamily{U}{mathx}{\hyphenchar\font45}
\DeclareFontShape{U}{mathx}{m}{n}{
      <5> <6> <7> <8> <9> <10>
      <10.95> <12> <14.4> <17.28> <20.74> <24.88>
      mathx10
      }{}
\DeclareSymbolFont{mathx}{U}{mathx}{m}{n}
\DeclareFontSubstitution{U}{mathx}{m}{n}
\DeclareMathAccent{\widebar}{0}{mathx}{"73}

\title[Modular generalized Springer correspondence II]{Modular generalized Springer correspondence II: classical groups}

\author{Pramod N. Achar}
\address{Department of Mathematics\\
  Louisiana State University\\
  Baton Rouge, LA 70803\\
  U.S.A.}
\email{pramod@math.lsu.edu}

\author{Anthony Henderson}
\address{School of Mathematics and Statistics\\
  University of Sydney, NSW 2006\\
  Australia}
\email{anthony.henderson@sydney.edu.au}

\author{Daniel Juteau}
\address{Laboratoire de Math\'ematiques Nicolas Oresme\\
  Universit\'e de Caen, BP 5186\\
  14032 Caen Cedex\\ 
  France}
\email{daniel.juteau@unicaen.fr}

\author{Simon Riche}
\address{Universit{\'e} Blaise Pascal - Clermont-Ferrand II, Laboratoire de Math{\'e}matiques, CNRS, UMR 6620, Campus universitaire des C{\'e}zeaux, F-63177 Aubi{\`e}re Cedex, France
}
\email{simon.riche@math.univ-bpclermont.fr}


\thanks{P.A. was supported by NSF Grant No.~DMS-1001594.  A.H. was supported by ARC Future Fellowship Grant No.~FT110100504. D.J. and S.R. were supported by ANR Grants No.~ANR-09-JCJC-0102-01 and ANR-13-BS01-0001-01. S.R. was supported by ANR Grant No.~ANR-2010-BLAN-110-02.}

\begin{document}

\begin{abstract}
We construct a modular generalized Springer correspondence for any classical group, by generalizing to the modular setting various results of Lusztig in the case of characteristic-$0$ coefficients. We determine the cuspidal pairs in all classical types, and compute the correspondence explicitly for $\SL(n)$ with coefficients of arbitrary characteristic and for $\SO(n)$ and $\Sp(2n)$ with characteristic-$2$ coefficients.
\end{abstract}

\maketitle

\section{Introduction}
\label{sec:intro}

\subsection{Summary}
This paper continues a series, commenced in~\cite{genspring1}, which aims to construct and describe a modular generalized Springer correspondence for connected reductive groups: in other words, to prove analogues, for sheaves with modular coefficients, of the fundamental results of Lusztig~\cite{lusztig,lusztig-fourier,lusztig-cusp2} on the generalized Springer correspondence for $\Qlb$-sheaves. In~\cite{genspring1} we accomplished this for the group $\GL(n)$. Here we establish more of the foundational results, and use them to construct the correspondence for classical groups more generally.

\subsection{Statement of the main result}

Recall the set-up from~\cite{genspring1}: $G$ denotes a connected reductive group over $\C$, and we consider $G$-equivariant perverse sheaves on the nilpotent cone $\cN_G$ with coefficients in a field $\bk$ of positive characteristic $\ell$. The simple perverse sheaves are indexed by the set $\fN_{G,\bk}$ of pairs $(\cO,\cE)$ where $\cO$ is a nilpotent orbit and $\cE$ is an irreducible $G$-equivariant $\bk$-local system on $\cO$. As we recall in~\S\ref{sec:generalities}, there is a subset $\fN_{G,\bk}^\cusp\subset\fN_{G,\bk}$ of \emph{cuspidal pairs}. Let $\mathfrak{L}$ be a set of representatives of $G$-conjugacy classes of Levi subgroups of $G$. For any $L\in\fL$ and $(\cO_L,\cE_L)\in\fN_{L,\bk}^\cusp$, we have a corresponding \emph{induction series} $\fN_{G,\bk}^{(L,\cO_L,\cE_L)}\subset\fN_{G,\bk}$.

\begin{thm}
\label{thm:main}
Assume that $G$ is classical, and that $\bk$ is big enough \textup{(}see below for the precise conditions\textup{)}. Then we have a disjoint union
\begin{equation} \label{eqn:disjointness}
\fN_{G,\bk} = \bigsqcup_{L \in \mathfrak{L}} \bigsqcup_{(\cO_L,\cE_L)\in\fN_{L,\bk}^\cusp}\fN_{G,\bk}^{(L,\cO_L,\cE_L)},
\end{equation}
and for any $L\in\fL$ and $(\cO_L,\cE_L)\in\fN_{L,\bk}^\cusp$ we have a canonical bijection
\begin{equation} \label{eqn:bijection}
\fN_{G,\bk}^{(L,\cO_L,\cE_L)}\longleftrightarrow\Irr(\bk[N_G(L)/L]).
\end{equation}
Hence we obtain a bijection
\begin{equation} \label{eqn:mgsc}
\fN_{G,\bk} \longleftrightarrow
\bigsqcup_{L \in \mathfrak{L}} \bigsqcup_{(\cO_L,\cE_L)\in\fN_{L,\bk}^\cusp}\Irr(\bk[N_G(L)/L]), 
\end{equation}
which we call the \emph{modular generalized Springer correspondence} for $G$. 
\end{thm}

Here, $\Irr(\bk[N_G(L)/L])$ denotes the set of isomorphism classes of irreducible $\bk$-representations of the relative Weyl group $N_G(L)/L$; we say that $G$ is \emph{classical} if its root system has irreducible components only of types $\mathbf{A}$, $\mathbf{B}$, $\mathbf{C}$ or $\mathbf{D}$; and we say that $\bk$ is \emph{big enough} if it satisfies the following conditions:
\begin{enumerate}
\item if the root system of $G$ contains a component of type $\mathbf{A}_{n-1}$, then $\bk$ contains all $n$-th roots of unity of its algebraic closure;
\item if the root system of $G$ contains a component of type $\mathbf{B}$ or $\mathbf{D}$, then $\bk$ contains all fourth roots of unity of its algebraic closure.
\end{enumerate}
The second condition is, of course, vacuous if $\bk$ happens to have characteristic $2$. As we will see, for particular groups these conditions on $\bk$ can be weakened (for instance, in \cite{genspring1} we proved Theorem~\ref{thm:main} for $G=\GL(n)$ and $\bk$ arbitrary); we imposed them uniformly in order to have a concise statement.

\subsection{Overview of the proof of Theorem~\ref{thm:main}}

The main content of Theorem~\ref{thm:main} can be divided into two rather different results on induction series: the
disjointness (i.e.~the fact that the unions on the right-hand side of~\eqref{eqn:disjointness} are disjoint), and the parametrization (i.e.~the canonical bijection~\eqref{eqn:bijection}).
Recall that in Lusztig's paper~\cite{lusztig}, 
these results had uniform proofs, whereas case-by-case arguments were needed for the explicit descriptions of the cuspidal pairs and the generalized Springer correspondence, the latter being completed subsequently in~\cite{lus-spalt,spaltenstein}. 

Theorem~\ref{thm:cocycle} of this paper provides a uniform construction of the parametrization~\eqref{eqn:bijection}, which requires only mild assumptions on $(\cO_L,\cE_L)$ (weaker than the assumptions of Theorem~\ref{thm:main}). In particular, this statement does not require $G$ to be classical. As in~\cite{lusztig}, the word ``canonical'' refers to the fact that this bijection does not depend on any choice: it is characterized by a geometric condition, relating to the restriction of a certain perverse sheaf constructed from $(L,\cO_L,\cE_L)$ to the induced orbit $\mathrm{Ind}_L^G(\cO_L')$, where $\cO_L' \subset \cN_L$ is a nilpotent orbit determined by $(\cO_L,\cE_L)$ (which frequently coincides with $\cO_L$).

On the other hand, our proof of disjointness (or rather of that part of its content which goes beyond the general result Corollary~\ref{cor:partial-disjointness}) relies on the classification of cuspidal pairs, and hence requires case-by-case arguments. More precisely, we use general results to reduce the proof of this disjointness to two key statements about cuspidal pairs (see Theorem~\ref{thm:conditional}). Then we use induction on the rank within each classical type to classify cuspidal pairs, and simultaneously check these statements:
for $\SL(n)$ in Theorem~\ref{thm:sln-cuspidal}, for $\Sp(2n)$ when $\ell=2$ in Theorem~\ref{thm:sp2n2-cuspidal}, for $\Sp(2n)$ when $\ell\neq 2$ in Theorem~\ref{thm:sp2n-cuspidal}, for $\Spin(n)$ when $\ell=2$ in Theorems~\ref{thm:spinodd2-cuspidal} and~\ref{thm:spineven2-cuspidal}, and for $\Spin(n)$ when $\ell\neq 2$ in Theorems~\ref{thm:sonot2-cuspidal} and~\ref{thm:spinnot2-cuspidal}. (Easy arguments, explained in~\S\ref{ss:reductions}, reduce the classification of modular cuspidal pairs to the case where $G$ is simply connected and quasi-simple.)

This approach is similar to the one used for $\GL(n)$ in~\cite{genspring1}. The main new complication in this paper is the appearance of non-constant local systems. (In the case where the local systems are constant, the bijection~\eqref{eqn:bijection} is easy to see; hence Theorem~\ref{thm:cocycle} was not needed in~\cite{genspring1}.)

\subsection{Remarks on cuspidal pairs}

In the classification of cuspidal pairs we use two general results, which provide an upper bound and a lower bound for the number of cuspidal pairs. Namely, as observed in~\cite{genspring1}, all pairs obtained by modular reduction from a cuspidal pair in characteristic $0$ are cuspidal, providing the lower bound. On the other hand, an easy generalization of a result of Lusztig (Proposition~\ref{prop:distinguished}) says that if $(\cO,\cE)$ is a cuspidal pair then $\cO$ is a distinguished orbit, providing the upper bound. We will show that, for quasi-simple classical groups, the description of cuspidal pairs is always given by one of these two extremes. In type $\mathbf{A}$ for all $\ell$ and in types $\mathbf{B}$, $\mathbf{C}$, $\mathbf{D}$ when $\ell \neq 2$, the only cuspidal pairs are those obtained by modular reduction; this is reminiscent of the result of Geck--Hiss--Malle~\cite[Theorem 4.11]{ghm} on cuspidal unipotent Brauer characters of finite classical groups. On the other hand, in types $\mathbf{B}$, $\mathbf{C}$, $\mathbf{D}$ when $\ell=2$, all pairs supported on a distinguished orbit are cuspidal; that is, cuspidal pairs are as plentiful as possible.

Among the exceptional groups, which will be treated in the third paper of this series, there are cases where the number of cuspidal pairs is strictly between these upper and lower bounds. As a result, there are cases where we cannot easily verify Statement~\ref{stmt:dist-char} of the present paper, which asserts the distinctness of central characters of cuspidal pairs supported on the same orbit. Therefore the third paper will include a different proof of the disjointness of induction series, based on a Mackey formula for our induction and restriction functors.

\subsection{Explicit determination of the correspondence}

Theorem~\ref{thm:main} raises the problem of determining the modular generalized Springer correspondence~\eqref{eqn:mgsc} explicitly in terms of the usual combinatorial parametrizations of both sides. In Theorem~\ref{thm:sln-det} we solve this problem for $G=\SL(n)$ by a similar method to that used in~\cite{genspring1} for the $\GL(n)$ case, and in Theorems~\ref{thm:sp2n2-det} and~\ref{thm:so4n2-det} we solve it for $G=\SO(n)$ and $G=\Sp(2n)$ when $\ell=2$. In particular, the latter results determine for the first time the (un-generalized) modular Springer correspondence for these groups when $\ell=2$, complementing the results of~\cite{jls} on the $\ell\neq 2$ case; see Corollary~\ref{cor:ordinary-springer}.

These determinations require further general results, proved in Section~\ref{sect:transitivity}, which play the same role in our theory that Lusztig's restriction theorem~\cite[Theorem 8.3]{lusztig} did in the determination of his generalized Springer correspondence. 

\subsection{Acknowledgements}

We are grateful to C\'edric Bonnaf\'e for helpful conversations.

\subsection{Organization of the paper} 

The remainder of the paper falls into three parts. In Sections~\ref{sec:generalities}--\ref{sect:strategy}, $G$ is a general connected reductive group, and we prove a number of general results underlying Theorem~\ref{thm:main}. In Sections~\ref{sec:sln}--\ref{sec:son} we take $G$ to be a simply connected quasi-simple classical group, considering the various types in turn; these sections complete the proof of Theorem~\ref{thm:main}, along the lines set out in Section~\ref{sect:strategy}. In Section~\ref{sec:det2} we compute the modular generalized Springer correspondence in the cases mentioned above.

\tableofcontents

\section{Generalities}
\label{sec:generalities}

In this section we continue the study, begun in~\cite[Section 2]{genspring1}, of the basic definitions and constructions required to formulate the modular generalized Springer corrrespondence for an arbitrary connected reductive group over $\C$.

\subsection{Some notation}

Our notation follows~\cite{genspring1}. In particular, $\bk$ denotes a field of characteristic $\ell > 0$. We consider sheaves of $\bk$-vector spaces on varieties over $\C$. For a complex algebraic group $H$ acting on a variety $X$, we denote by $\Db_H(X,\bk)$ the constructible $H$-equivariant derived category and by $\Perv_H(X,\bk)$ its subcategory of $H$-equivariant perverse $\bk$-sheaves on $X$. We denote by $\Loc(X,\bk)$ the category of $\bk$-local systems on $X$, and by $\Loc^H(X,\bk)$ the category of $H$-equivariant local systems.

Throughout, $G$ denotes a connected reductive complex algebraic group, $\fg$ its Lie algebra and $\cN_G\subset\fg$ its nilpotent cone. Recall that $G$ has finitely many orbits in $\cN_G$, and that every simple object in $\Perv_G(\cN_G,\bk)$ is of the form $\IC(\cO,\cE)$ where $\cO\subset\cN_G$ is a $G$-orbit and $\cE$ is an irreducible $G$-equivariant $\bk$-local system on $\cO$. Let $\fN_{G,\bk}$ denote the set of such pairs $(\cO,\cE)$, where the local systems $\cE$ on a given orbit $\cO$ are taken up to isomorphism. Thus $\fN_{G,\bk}$ is finite and parametrizes the isomorphism classes of simple objects of $\Perv_G(\cN_G,\bk)$.

Let $P \subset G$ be a parabolic subgroup, and let $L \subset P$ be a Levi factor. Then $L$ is also a connected reductive group, with Lie algebra $\fl$ and nilpotent cone $\cN_L$. We will denote by $U_P$ the unipotent radical of $P$, by $\fp$ the Lie algebra of $P$, and by $\fu_P$ the Lie algebra of $U_P$. 

As explained in~\cite[\S2.1]{genspring1}, we have two restriction functors
\[
\Res^G_{L \subset P}, {}'\Res^G_{L \subset P} : \Perv_G(\cN_G,\bk) \to \Perv_{L}(\cN_{L},\bk),
\] 
which are exchanged by Verdier duality, and an induction functor 
\[
\Ind_{L \subset P}^G : \Perv_{L}(\cN_{L},\bk) \to \Perv_G(\cN_G,\bk),
\]
which commutes with Verdier duality.
All these functors are exact, and we have adjunctions ${}'\Res^G_{L \subset P}\dashv \Ind_{L \subset P}^G\dashv \Res^G_{L \subset P}$.

For simplicity we will say that $L \subset G$ is a \emph{Levi subgroup} if it is a Levi factor of a parabolic subgroup of $G$. Given a Levi subgroup $L$, we write $\fz_L$ for the centre of its Lie algebra $\fl$, and $\fz_L^\circ$ for the open subset $\{z\in\fz_L\,|\,G_z^\circ=L\}$, where $G_z$ denotes the stabilizer of $z$ in $G$ and $G_z^\circ$ its identity component. The \emph{Lusztig stratification} of $\fg$, defined in~\cite[\S 6]{lusztig-cusp2}, expresses $\fg$ as the disjoint union of the smooth $G$-stable irreducible locally closed subvarieties
\[
Y_{(L,\cO_L)}:= G\cdot(\cO_L+\fz_L^\circ),
\]
where $L$ runs over the Levi subgroups of $G$ and $\cO_L$ over the nilpotent orbits for $L$, and $Y_{(L,\cO_L)}=Y_{(M,\cO_M)}$ if and only if the pairs $(L,\cO_L)$ and $(M,\cO_M)$ are $G$-conjugate. As in~\cite{genspring1}, we define
\[
X_{(L,\cO)} := \overline{Y_{(L,\cO)}}
\qquad\text{and}\qquad
\widetilde{Y}_{(L,\cO_L)}:=G\times^L(\cO_L+\fz_L^\circ),
\]
and we let $\varpi_{(L,\cO_L)}:\widetilde{Y}_{(L,\cO_L)}\to Y_{(L,\cO_L)}$ denote the natural $G$-equivariant morphism. By~\cite[proof of Lemma 5.1.28]{letellier}, $\varpi_{(L,\cO_L)}$ is a Galois covering with Galois group $N_G(L,\cO_L)/L$, where $N_G(L,\cO_L)$ denotes the subgroup of the normalizer $N_G(L)$ that preserves the orbit $\cO_L$. Here $n\in N_G(L,\cO_L)$ acts on $\widetilde{Y}_{(L,\cO_L)}$ by
\[
n\cdot(g*(x+z))=gn^{-1}*(n\cdot (x+z)),\text{ for }g\in G,\, x\in\cO_L,\, z\in\fz_L^\circ.
\]
If $\cE_L$ denotes an $L$-equivariant local system on $\cO_L$, then we write $\widetilde{\cE_L}$ for the unique $G$-equivariant local system on $\widetilde{Y}_{(L,\cO_L)}$ whose pull-back to $G\times(\cO_L+\fz_L^\circ)$ is $\ubk_G\boxtimes(\cE_L\boxtimes\ubk_{\fz_L^\circ})$.

\subsection{Induction series}
\label{ss:series}

Recall that a simple object $\cF$ in the abelian category $\Perv_G(\cN_G,\bk)$ is called \emph{cuspidal} if for any proper parabolic $P\subsetneq G$ and Levi factor $L \subset P$ we have $\Res^G_{L \subset P}(\cF)=0$. By~\cite[Proposition 2.1]{genspring1}, the definition is unchanged if we instead require ${}'\Res^G_{L \subset P}(\cF)=0$. A pair $(\cO,\cE)\in\fN_{G,\bk}$ is called a \emph{cuspidal pair} if $\IC(\cO,\cE)$ is cuspidal. Let $\fN_{G,\bk}^\cusp\subset\fN_{G,\bk}$ denote the set of cuspidal pairs.

Recall from~\cite[Corollary 2.7]{genspring1} that every simple object of $\Perv_G(\cN_G,\bk)$ appears as a quotient of $\Ind_{L \subset P}^G(\IC(\cO_L,\cE_L))$ for some $L\subset P\subset G$ as above and $(\cO_L,\cE_L)\in\fN_{L,\bk}^\cusp$. The cuspidal objects of $\Perv_G(\cN_G,\bk)$ occur here in the case $L=P=G$.

As in~\cite{genspring1}, we make essential use of the Fourier--Sato transform $\bT_{\fg}$, an autoequivalence of the category of conic $G$-equivariant perverse $\bk$-sheaves on $\fg$ (for its definition and basic properties, see~\cite{ahjr}).  By~\cite[Corollary~2.12]{genspring1}, for any $(\cO_L,\cE_L) \in \fN_{L,\bk}^\cusp$, there is a unique pair $(\cO_L',\cE_L')\in\fN_{L,\bk}^\cusp$ such that
\begin{equation}
\label{eqn:l-fourier}
\bT_{\fl}(\IC(\cO_L,\cE_L))\cong\IC(\cO_L'+\fz_L,\cE_L'\boxtimes\ubk_{\fz_L}).
\end{equation}
As mentioned in~\cite[Remark~2.13]{genspring1}, it is possible that, as in the characteristic-zero case, we have $(\cO_L',\cE_L')=(\cO_L,\cE_L)$ always. (We will see in Corollary~\ref{cor:sln-fourier-invariance} below that this holds when $G = \SL(n)$.)
Next,~\cite[Corollary~2.18]{genspring1} tells us that there is a canonical isomorphism
\begin{equation} 
\label{eqn:fourier}
\bT_{\fg}(\Ind_{L \subset P}^G(\IC(\cO_L,\cE_L)))\cong\IC(Y_{(L,\cO_L')},(\varpi_{(L,\cO_L')})_*\widetilde{\cE_L'}).
\end{equation}
It follows from~\eqref{eqn:fourier} that $\Ind_{L \subset P}^G(\IC(\cO_L,\cE_L))$ does not depend on $P$, up to canonical isomorphism. We refer to the set of isomorphism classes of simple quotients of $\Ind_{L \subset P}^G(\IC(\cO_L,\cE_L))$, or to the corresponding subset $\fN_{G,\bk}^{(L,\cO_L,\cE_L)}$ of $\fN_{G,\bk}$, as the \emph{induction series} attached to $(L,\cO_L,\cE_L)$. Clearly, this induction series is unchanged if the triple $(L,\cO_L,\cE_L)$ is conjugated by an element of $G$.

\begin{lem}
\label{lem:fourier}
Let $(\cO,\cE)\in\fN_{G,\bk}$. Then $(\cO,\cE)\in\fN_{G,\bk}^{(L,\cO_L,\cE_L)}$ if and only if 
\[ \bT_{\fg}(\IC(\cO,\cE))\cong\IC(Y_{(L,\cO_L')},\cD) \] 
for some simple local system $\cD$ on $Y_{(L,\cO_L')}$ that is a quotient of $(\varpi_{(L,\cO_L')})_*\widetilde{\cE_L'}$, where $(\cO_L',\cE_L')\in\fN_{L,\bk}^\cusp$ is as in \eqref{eqn:l-fourier}. In particular, $\fN_{G,\bk}^{(L,\cO_L,\cE_L)}$ is canonically in bijection with the set of isomorphism classes of simple quotients of $(\varpi_{(L,\cO_L')})_*\widetilde{\cE_L'}$.
\end{lem}

\begin{proof}
This follows immediately from~\eqref{eqn:fourier}, using the fact that $\bT_{\fg}$ is exact and fully faithful and $\IC$ preserves heads; see \cite[Proposition 2.28]{juteau-aif}.
\end{proof}

\begin{cor}
\label{cor:partial-disjointness}
If $(L,\cO_L,\cE_L)$ and $(M,\cO_M,\cE_M)$ are two triples as above where $(L,\cO_L')$ and $(M,\cO_M')$ are not $G$-conjugate, then $\fN_{G,\bk}^{(L,\cO_L,\cE_L)}\cap\fN_{G,\bk}^{(M,\cO_M,\cE_M)}=\emptyset$. In particular, this holds if $L$ and $M$ are not $G$-conjugate.
\end{cor}

\begin{proof}
Since $(L,\cO_L')$ and $(M,\cO_M')$ are not $G$-conjugate, $Y_{(L,\cO_L')}$ and $Y_{(M,\cO_M')}$ are different pieces of the Lusztig stratification. The result follows from Lemma~\ref{lem:fourier}.
\end{proof}


In defining induction series, we focused on quotients rather than subobjects; however, we have an analogue of `Property A' for Harish--Chandra induction of cuspidal modular representations~\cite[Section 2.2]{gh}:

\begin{lem} \label{lem:head-socle}
The induction series attached to $(L,\cO_L,\cE_L)$ equals the set of isomorphism classes of simple subobjects of $\Ind_{L \subset P}^G(\IC(\cO_L,\cE_L))$. 
\end{lem}

\begin{proof}
In view of Lemma~\ref{lem:fourier} and its analogue for subobjects, it suffices to prove that for any simple local system $\cL$ on $Y_{(L,\cO_L')}$ we have 
\[
\dim\Hom_{\Loc(Y_{(L,\cO_L')},\bk)}\bigl((\varpi_{(L,\cO_L')})_*\widetilde{\cE_L'},\cL\bigr)=\dim\Hom_{\Loc(Y_{(L,\cO_L')},\bk)}\bigl(\cL,(\varpi_{(L,\cO_L')})_*\widetilde{\cE_L'}\bigr).
\] 
Because $\varpi_{(L,\cO_L')}$ is a Galois covering, the functors $(\varpi_{(L,\cO_L')})_*$ and $(\varpi_{(L,\cO_L')})^*$ are biadjoint; so it is equivalent to show that 
\[
\dim\Hom_{\Loc(\widetilde{Y}_{(L,\cO_L')},\bk)}\bigl(\widetilde{\cE_L'},(\varpi_{(L,\cO_L')})^*\cL\bigr)=\dim\Hom_{\Loc(\widetilde{Y}_{(L,\cO_L')},\bk)}\bigl((\varpi_{(L,\cO_L')})^*\cL,\widetilde{\cE_L'}\bigr).
\]
Since the local system $\widetilde{\cE_L'}$ is simple, it sufficies to prove that $(\varpi_{(L,\cO_L')})^*\cL$ is semisimple. However, since this local system is $N_G(L,\cO_L')/L$-equivariant, its socle is also $N_G(L,\cO_L')/L$-equivariant; and since $(\varpi_{(L,\cO_L')})^*\cL$ is simple as an $N_G(L,\cO_L')/L$-equivariant local system (see equivalence~\eqref{eqn:equivalence-equivariant} below), this socle must be equal to
$(\varpi_{(L,\cO_L')})^*\cL$.
\end{proof}

In the following corollary, we denote by $\cE^\vee$ the local system dual to $\cE$. In the statement we use the fact that if $(\cO_L, \cE_L)$ is a cuspidal pair for $L$, then $(\cO_L, \cE_L^\vee)$ is also a cuspidal pair (see~\cite[Remark~2.3]{genspring1}).

\begin{cor}
If $(\cO, \cE) \in \fN_{G,\bk}^{(L,\cO_L,\cE_L)}$, then $(\cO, \cE^\vee) \in \fN_{G,\bk}^{(L,\cO_L,\cE_L^\vee)}$.
\end{cor}

\begin{proof}
Since Verdier duality commutes with $\Ind_{L \subset P}^G$, this follows from Lemma~\ref{lem:head-socle}. 
\end{proof}

The following easy result is sometimes useful in determining induction series.

\begin{lem} \label{lem:bounds}
Let $(\cO,\cE)\in\fN_{G,\bk}$, and suppose that $\IC(\cO,\cE)$ is a quotient of $\Ind_{L \subset P}^G(\IC(\cO_L,\cE_L))$ for some $L\subset P\subset G$ as above and $(\cO_L,\cE_L)\in\fN_{L,\bk}$. Then
\[
G\cdot\cO_L\subset \overline{\cO}\subset G\cdot(\overline{\cO_L}+\fu_P).
\]
\end{lem}

\begin{proof}
The support of $\Ind_{L \subset P}^G(\IC(\cO_L,\cE_L))$ is contained in $G\cdot(\overline{\cO_L}+\fu_P)$ by~\cite[Corollary 2.15(1)]{genspring1}, so $\overline{\cO}\subset G\cdot(\overline{\cO_L}+\fu_P)$ (the latter being closed). By adjunction, 
\[ 
\Hom\bigl(\IC(\cO_L,\cE_L),\Res_{L\subset P}^G(\IC(\cO,\cE))\bigr)=\Hom\bigl(\Ind_{L \subset P}^G(\IC(\cO_L,\cE_L)),\IC(\cO,\cE)\bigr)\neq 0,
\] 
implying that $\cO_L$ is contained in the support of $\Res_{L\subset P}^G(\IC(\cO,\cE))$. By definition of $\Res_{L\subset P}^G$, the latter support is contained in $\overline{\cO}$, so $G\cdot\cO_L\subset\overline{\cO}$.  
\end{proof}

\subsection{Cuspidal pairs and distinguished orbits}

Recall that a $G$-orbit $\cO\subset\cN_G$ is said to be \emph{distinguished} if it does not meet $\cN_L$ for any proper Levi subgroup $L$ of $G$. We then have the following analogue of~\cite[Proposition 2.8]{lusztig}.

\begin{prop}
\label{prop:distinguished}
If $(\cO,\cE)\in\fN_{G,\bk}^\cusp$, then $\cO$ is distinguished.
\end{prop}

This follows immediately from:

\begin{prop}
\label{prop:distinguished-restrict}
Let $(\cO,\cE)\in\fN_{G,\bk}$. If $\cO$ meets $\cN_L$ where $L\subset P\subset G$ are as above, then ${}'\Res^G_{L \subset P}(\IC(\cO,\cE))\neq 0$.
\end{prop} 

\begin{proof}
Let $x\in\cO\cap\cN_L$ and let $\cC$ be the $L$-orbit of $x$. It suffices to show that
\begin{equation*}
\mathsf{H}^{-\dim(\mathscr{C})} \Bigl( {}'\Res^G_{L \subset P} \bigl( \IC(\cO,\mathcal{E}) \bigr)_x \Bigr)\neq 0, 
\end{equation*}
which by definition of ${}'\Res^G_{L \subset P}$ is equivalent to 
\begin{equation} \label{eqn:non-vanishing}
\mathsf{H}_c^{-\dim(\mathscr{C})} \bigl( (x+\fu_P) \cap \overline{\cO}, \IC(\cO,\mathcal{E}) \bigr)\neq 0.
\end{equation}
But for any $y\in x+\fu_P$, we have $x\in\overline{G\cdot y}$, since there is a $1$-parameter subgroup of $G$ that fixes $\fl$ and contracts $\fu_P$ to zero. Hence $(x+\fu_P) \cap \overline{\cO}=(x+\fu_P) \cap \cO$, and~\eqref{eqn:non-vanishing} becomes
\begin{equation} \label{eqn:non-vanishing2}
\mathsf{H}_c^{\dim(\cO)-\dim(\mathscr{C})} \bigl( (x+\fu_P) \cap \cO, \mathcal{E} \bigr)\neq 0.
\end{equation}
The proof of~\eqref{eqn:non-vanishing2} is analogous to that of the corresponding statement in~\cite{lusztig}. Namely, by \cite[Proposition~5.1.15(1)]{letellier} we have $\dim \bigl( (x+\fu_P) \cap \cO \bigr) \leq \frac{1}{2}(\dim(\cO)-\dim(\mathscr{C}))$. One sees in exactly the same way as in~\cite[Lemma 2.9(a)]{lusztig} that $U_P\cdot x$ is an irreducible component of $(x+\fu_P)\cap\cO$ of dimension equal to $\frac{1}{2}(\dim(\cO)-\dim(\mathscr{C}))$, and that the restriction of $\mathcal{E}$ to $U_P\cdot x$ is a constant sheaf because $(U_P)_x$ is connected. It follows that $\mathsf{H}_c^{\dim(\cO)-\dim(\mathscr{C})} \bigl( U_P \cdot x, \mathcal{E} \bigr)\neq 0$, which implies~\eqref{eqn:non-vanishing2}.
\end{proof}

\begin{rmk}
When $G=\GL(n)$, only the regular orbit $\cO_{(n)}$ is distinguished. So Proposition~\ref{prop:distinguished} is consistent with~\cite[Theorem 3.1]{genspring1}.
\end{rmk}

We conclude this subsection with a useful observation about distinguished orbits. It is probably known to experts, but we could not find a reference. 

\begin{lem} \label{lem:dist-norm}
If $L$ is a Levi subgroup of $G$ and $\cO_L$ is a distinguished nilpotent orbit of $L$, then $N_G(L,\cO_L)=N_G(L)$; that is, the normalizer $N_G(L)$ preserves $\cO_L$.
\end{lem}

\begin{proof}
Since the assumptions and conclusion are unchanged if one replaces $G$ by a central quotient, we can assume that $G$ is a product of simple groups; thus it suffices to consider the case where $G$ is itself simple. 
The root system of $L$ is then a sum of irreducible root systems, at most one of which is of type different from $\mathbf{A}$. Correspondingly, $\cN_L$ is a product of nilpotent cones for simple groups, at most one of which is of type different from $\mathbf{A}$, and the orbit $\cO_L$ is a product of distinguished orbits in these nilpotent cones. The action of the normalizer $N_G(L)$ on $\cN_L$ preserves the product of the factors of type $\mathbf{A}$, and in each of these factors the distinguished orbit must be the regular orbit. So it suffices to consider the case where $\cN_L$ does have a factor of type different from $\mathbf{A}$, and to show that the action of $N_G(L)$ preserves each distinguished nilpotent orbit in that non-type-$\mathbf{A}$ factor.

If $G$ is of classical type, then in fact $N_G(L)$ preserves every nilpotent orbit in the non-type-$\mathbf{A}$ factor of $\cN_L$, since the action of each element of $N_G(L)$ on the non-type-$\mathbf{A}$ factor of $\cN_L$ is the same as that of some element of $L$. To see this, one can assume that $G=\Sp(V)$ or $\SO(V)$ where $V$ is a vector space with a non-degenerate skew-symmetric or symmetric bilinear form. If for example $G=\Sp(V)$, then there is some orthogonal direct sum decomposition $V=U\oplus U^\perp$ such that $L=\Sp(U)\times H$ and $N_G(L)=\Sp(U)\times H'$ where $H$ and $H'$ are subgroups of $\GL(U^\perp)$; thus, the action of each element of $N_G(L)$ on the non-type-$\mathbf{A}$ factor $\cN_{\Sp(U)}$ of $\cN_L$ is that of some element of $\Sp(U)$.

If $G$ is of exceptional type, then the distinguished nilpotent orbits in the non-type-$\mathbf{A}$ factor of $\cN_L$ have different dimensions (this can be checked case-by-case, for instance using the tables in~\cite[Section 8.4]{cm}), so the claim is again obvious.  
\end{proof}

\begin{rmk}
Since the distinguished nilpotent orbits in $L$ are exactly the Richardson orbits of distinguished parabolic subalgebras of $\fl$ (see~\cite[p.~123]{cm}), the lemma is equivalent to the statement that if $\fp$, $\fp'$ are distinguished parabolic subalgebras of $\fl$, then the pairs $(\fl, \fp)$ and $(\fl, \fp')$ are $G$-conjugate iff $\fp$ and $\fp'$ are $L$-conjugate, or in other words that in the Bala--Carter classification of nilpotent orbits as stated in~\cite[Theorem~8.2.12]{cm}, the parabolic subalgebra $\fp_{\fl}$ is defined uniquely up to $L$-conjugacy. In the case of classical groups, this property is observed in~\cite[p.~126]{cm}. If $G$ is simple of exceptional type, it follows from the fact that, for each proper Levi subgroup $L$ of $G$, the Levi subalgebras of non-$L$-conjugate distinguished parabolic subalgebras of $L$ have different semisimple ranks (since $L$ is not of type $\mathbf{E}_8$).
\end{rmk}

\section{Study of an induced local system}
\label{sec:study}

In this section we continue to let $G$ be an arbitrary connected reductive group over $\C$. We fix a Levi subgroup $L$ and an orbit $\cO \subset \cN_L$. (To reduce clutter, we drop the subscript $L$ from the notation $\cO_L$ of the previous section.) Sometimes it will be convenient to choose a parabolic subgroup $P \subset G$ with Levi factor $L$, but our main constructions will not depend on this choice. 

We let $\mathrm{Ind}_L^G(\cO)$ denote the nilpotent orbit in $\cN_G$ induced by $\cO$. The most familiar definition of $\mathrm{Ind}_L^G(\cO)$ is as the unique nilpotent orbit that intersects $\cO+\fu_P$ in a dense set (see~\cite[Theorem 7.1.1]{cm}), but it can also be defined in a way that is independent of $P$: it is the unique nilpotent orbit that is dense in $\cN_G\cap X_{(L,\cO)}$ (see~\cite[Theorem 7.1.3 and its proof]{cm}).

\subsection{Overview}
\label{ss:overview}

Our aim in this section is to study the local system $(\varpi_{(L,\cO)})_*\widetilde{\cE}$ on $Y_{(L,\cO)}$, where $\cE$ is an irreducible $L$-equivariant $\bk$-local system on $\cO$ and $\widetilde{\cE}$ is the resulting local system on $\widetilde{Y}_{(L,\cO)}$. The specific statement we will prove is the following; its proof will be completed in~\S\ref{ss:simple-quotients}. In this statement we use the fact that since $\varpi_{(L,\cO)} : \widetilde{Y}_{(L,\cO)} \to Y_{(L,\cO)}$ is a Galois covering with Galois group $N_G(L,\cO)/L$, there is a natural functor
\[
\Rep(N_G(L,\cO)/L, \bk) \to \Loc(Y_{(L,\cO)}, \bk): \ U \mapsto \cL_U;
\]
see \S\ref{ss:Loc-Galois} below (and in particular~\eqref{eqn:local-systems-Galois}) for the details.

\begin{thm}
\label{thm:cocycle}
Assume that the following conditions hold:
\begin{align}
& \negthickspace\!\begin{array}{l}
\text{there is a parabolic subgroup $P\subset G$ with Levi factor $L$}\\
\text{such that for any $y\in\mathrm{Ind}_L^G(\cO)\cap(\cO+\fu_P)$, $G_y\subset P$;}
\end{array} 
\label{eqn:distinguished-substitute} \\
& \text{$\cE$ is absolutely irreducible;}  \label{eqn:abs-irr} \\
& \text{the isomorphism class of $\cE$ is fixed by the action of $N_G(L,\cO)$.}\label{eqn:fixed}
\end{align}
Then there is a unique \textup{(}up to isomorphism\textup{)} indecomposable direct summand of $(\varpi_{(L,\cO)})_*\widetilde{\cE}$ whose $\IC$-extension has a nonzero restriction to $\mathrm{Ind}_L^G(\cO)$. This direct summand appears with multiplicity one, and its head $\overline{\cE}$ is absolutely irreducible. Moreover, the following properties hold:
\begin{enumerate}
\item $\dim \Hom \bigl( (\varpi_{(L,\cO)})_*\widetilde{\cE}, \overline{\cE} \bigr) = 1$;
\item the morphism $\widetilde{\cE} \to (\varpi_{(L,\cO)})^* \overline{\cE}$ obtained by adjunction from the unique \textup{(}up to scalar\textup{)} morphism $ (\varpi_{(L,\cO)})_*\widetilde{\cE} \to \overline{\cE}$ is an isomorphism;
\item the local system $(\varpi_{(L,\cO)})_*\widetilde{\cE}$ is isomorphic to $\overline{\cE} \otimes \cL_{\bk[N_G(L,\cO)/L]}$, and its endomorphism algebra $\End((\varpi_{(L,\cO)})_*\widetilde{\cE})$ is isomorphic to $\bk[N_G(L,\cO)/L]$;
\item the assignment $U \mapsto \overline{\cE} \otimes \cL_U$ induces a bijection
\begin{equation} 
\label{eqn:cocycle}
\Irr(\bk[N_G(L,\cO)/L])
\longleftrightarrow
\left\{\begin{array}{c}
\text{isomorphism classes of simple} \\
\text{quotients of $(\varpi_{(L,\cO)})_*\widetilde{\cE}$}
\end{array}\right\}.
\end{equation}
\end{enumerate}
\end{thm}

Our motivation is the occurrence of the right-hand side of~\eqref{eqn:cocycle} in Lemma~\ref{lem:fourier}, and after Section~\ref{sect:transitivity} we will need Theorem~\ref{thm:cocycle} only in the case where $(\cO,\cE)$ is a cuspidal pair for $L$. (In that case, condition~\eqref{eqn:distinguished-substitute} follows from Proposition~\ref{prop:distinguished} and Lemma~\ref{lem:cgu}; condition~\eqref{eqn:abs-irr} will be guaranteed by our assumptions on $\bk$; and condition~\eqref{eqn:fixed} will be implied by a certain property of cuspidal pairs. See the proof of Theorem~\ref{thm:conditional}.) However, it seemed worthwhile to explain the proof of Theorem~\ref{thm:cocycle} in as general a setting as possible. We will introduce our assumptions~\eqref{eqn:distinguished-substitute}--\eqref{eqn:fixed} on $(\cO,\cE)$ as they are needed in the arguments.

\begin{rmk} \label{rmk:trivial-case}
Notice that if $\cE=\ubk_{\cO}$ (and hence $\widetilde{\cE}=\ubk_{\widetilde{Y}_{(L,\cO)}}$), then $\ubk_{Y_{(L,\cO)}}$ satisfies the properties (1)--(4) of the local system $\overline{\cE}$ in the theorem, by the well-known generalities on Galois coverings recalled in \S\ref{ss:Loc-Galois}. 
We will see in Lemma~\ref{lem:trivial-case} below that, if $\cE=\ubk_{\cO}$, we do indeed have $\overline{\cE}=\ubk_{Y_{(L,\cO)}}$. Hence, in the case of $\GL(n)$, the bijection~\eqref{eqn:cocycle} coincides with the one used in~\cite[proof of Lemma 3.8]{genspring1}.
\end{rmk}

\begin{rmk} \label{rmk:cocycle}
Although this is not the point of view we will emphasize in the proof, the reader may find it enlightening to interpret Theorem~\ref{thm:cocycle} in terms of modular representations of finite groups. Namely,
the irreducible $L$-equivariant local system $\cE$ on $\cO$ corresponds to an irreducible $\bk$-representation $V$ of the group $A_L(x):=L_x/L_x^\circ$, where $x$ is a chosen element of $\cO$. 
There is a Galois covering of $Y_{(L,\cO)}$ with group $A_{N_G(L)}(x)$, denoted $\widetilde{Y}_{(L,\cO)}'$ in \S\ref{ss:geometry} below, from which $\widetilde{Y}_{(L,\cO)}$ is obtained by factoring out the action of the normal subgroup $A_L(x)$; note that $A_{N_G(L)}(x)/A_L(x)\cong N_G(L)_x/L_x\cong N_G(L,\cO)/L$. Hence each representation of $A_{N_G(L)}(x)$ determines a local system on $Y_{(L,\cO)}$, and one can check that $(\varpi_{(L,\cO)})_*\widetilde{\cE}$ corresponds in this way to the induced representation $\mathrm{Ind}_{A_L(x)}^{A_{N_G(L)}(x)}(V)$. The local system $\overline{\cE}$ produced by Theorem~\ref{thm:cocycle} corresponds to an irreducible representation $\overline{V}$ of $A_{N_G(L)}(x)$ whose restriction to $A_L(x)$ is isomorphic to $V$, by property (2). A necessary condition for such a representation $\overline{V}$ to exist is that the isomorphism class of $V$ be fixed by the conjugation action of $A_{N_G(L)}(x)$, which is exactly what assumption~\eqref{eqn:fixed} implies. 
Given this and the fact that $V$ is absolutely irreducible (by assumption~\eqref{eqn:abs-irr}), the obstruction to the existence of $\overline{V}$ is the cohomology class in $\mathsf{H}^2(N_G(L,\cO)/L,\bk^\times)$ determined by $V$ as in~\cite[Theorem 11.7]{isaacs}.
Thus, part of Theorem~\ref{thm:cocycle} is the purely algebraic statement that this cohomology class is trivial. However, an equally important part of Theorem~\ref{thm:cocycle} is that the construction of $\overline{\cE}$ is canonical: in other words, the geometric condition about $\IC$-extension singles out one particular choice of $\overline{V}$, and hence one particular bijection~\eqref{eqn:cocycle}.
\end{rmk}

\begin{rmk} \label{rmk:history}
The prototype for Theorem~\ref{thm:cocycle} is Lusztig's result~\cite[Theorem 9.2]{lusztig}, a main ingredient of his generalized Springer correspondence for $\Qlb$-sheaves. That result required $(\cO,\cE)$ to be a cuspidal pair for $L$, and was in the setting of the group $G$ rather than the Lie algebra $\fg$; since he worked over an algebraically closed field of characteristic $0$, Lusztig had no need to refer to absolute irreducibility or to heads of indecomposable summands. His original proof does not carry over to the modular setting. He subsequently gave a different proof in the Lie algebra context in~\cite[Sections 6-7]{lusztig-cusp2}, on which our arguments are essentially based. However, significant modifications are required, primarily because the property Lusztig stated as~\cite[Lemma 6.8(c)]{lusztig-cusp2} does not hold for modular cuspidal pairs (let alone the more general pairs allowed by Theorem~\ref{thm:cocycle}): that is, the parabolic subgroups $P$ of $G$ having $L$ as a Levi factor do not necessarily form a single orbit under the conjugation action of $N_G(L)$. For this reason, we cannot directly use the constructions in~\cite[Section 7]{lusztig-cusp2} involving the variety of all $G$-conjugates of $P$. Our alternative constructions were inspired by the treatment of Bonnaf\'e~\cite{bonnafe1}; he worked with $G$ rather than $\fg$, but many of his proofs require only minor adaptations to our case. 
\end{rmk}

\begin{rmk}
In~\cite[Proposition~9.5]{lusztig}, in the setting of $\Qlb$-sheaves, Lusztig characterized the image of the sign representation of $N_G(L,\cO)/L$ under his version of the bijection~\eqref{eqn:cocycle}. A general analogue of such a result in our context seems unlikely, since $N_G(L,\cO)/L$ need not be a Coxeter group even if $(\cO,\cE)$ is a cuspidal pair for $L$. For example, if $G$ is of type $\mathbf{E}_6$, a Levi subgroup $L$ of type $\mathbf{A}_2$ can have cuspidal pairs in characteristic $3$ by~\cite[Theorem 3.1]{genspring1}, and in that case $N_G(L,\cO)/L=N_G(L)/L$ is isomorphic to $(\fS_3\times\fS_3)\rtimes\Z/2\Z$ (see~\cite{howlett}).  
\end{rmk}

\subsection{Local systems and Galois coverings}
\label{ss:Loc-Galois}

In this subsection we collect some generalities on Galois coverings.

Let $A$ be a finite group, and let $\pi : X \to Y$ be a Galois covering with Galois group $A$. 
Since $A$ is finite, the datum of an $A$-equivariant local system on $X$ is equivalent to the datum of a local system $\cM$ on $X$, together with isomorphisms $\varphi_a : \cM \simto a^* \cM$ for all $a \in A$, which satisfy $\varphi_{ba} = a^*(\varphi_b) \circ \varphi_a$ for any $a,b \in A$. (Here, by abuse of notation we still denote by $a : X \simto X$ the action of $a \in A$.) Any object of the form $\pi^* \cL$, with $\cL$ in $\Loc(Y,\bk)$, has a canonical $A$-equivariant structure in which $\varphi_a$ is the isomorphism $\pi^*\cL=(\pi\circ a)^*\cL\simto a^*\pi^*\cL$. It is well known that the functor $\pi^*$ induces an equivalence of categories
\begin{equation}
\label{eqn:equivalence-equivariant}
\Loc(Y,\bk) \simto \Loc^A(X,\bk).
\end{equation}

Now we define a canonical fully faithful functor 
\begin{equation}
\label{eqn:functor-equivariant}
\Rep(A,\bk) \to \Loc^A(X,\bk)
\end{equation}
as follows:
to any finite dimensional $\bk$-representation $V$ of $A$ we associate the constant local system $\underline{V}_X$, with $\varphi_a : \underline{V}_X \simto a^*\underline{V}_X$ defined so that the composition 
\[ \underline{V}_X \overset{\varphi_a}{\longrightarrow} a^* \underline{V}_X \cong \underline{V}_X \] 
(where the second isomorphism is the canonical one arising from the fact that $\underline{V}_X$ is constant) is the automorphism of $\underline{V}_X$ induced by the action of $a$ on $V$. The essential image of~\eqref{eqn:functor-equivariant} is the subcategory consisting of $A$-equivariant local systems whose underlying local system is constant. 

Composing~\eqref{eqn:functor-equivariant} with an inverse of the equivalence~\eqref{eqn:equivalence-equivariant} (uniquely defined up to isomorphism of functors) we obtain a fully faithful functor
\begin{equation}
\label{eqn:local-systems-Galois}
\Rep(A,\bk) \to \Loc(Y,\bk):V\mapsto\cL_V,
\end{equation}
whose essential image is the subcategory whose objects are the local systems $\cL$ such that $\pi^* \cL$ is constant. By definition, for any representation $V$ of $A$ we have a canonical isomorphism $\pi^*\cL_V\cong\underline{V}_X$ of $A$-equivariant local systems on $X$; moreover, if $\cL$ is a local system on $Y$, then any isomorphism $\pi^*\cL\simto\underline{V}_X$ of $A$-equivariant local systems is induced by a unique isomorphism $\cL\simto\cL_V$.

\begin{lem}
\label{lem:reg-rep}
We have a canonical isomorphism $\pi_*\ubk_{X}\cong\cL_{\bk[A]}$, where $\bk[A]$ denotes the left regular representation of $A$.
\end{lem}
\begin{proof}
From the definitions one can easily write down a canonical isomorphism $\pi^*\pi_*\ubk_{X}\cong\underline{\bk[A]}_X$ of $A$-equivariant local systems, which implies the claim.
\end{proof}

For any $A$-equivariant local system $\cM$ on $X$, the group $A$ has a natural action (by isomorphisms in $\Loc(Y,\bk)$) on the direct image $\pi_*\cM$, in which $a\in A$ acts as the composition $\pi_*\cM\simto\pi_*a^*\cM\simto \pi_*\cM$ where the first isomorphism is $\pi_*(\varphi_a)$ and the second is base change. In particular, $A$ acts on $\pi_*\ubk_{X}$; under the isomorphism of Lemma~\ref{lem:reg-rep}, this corresponds to the action of $A$ on $\bk[A]$ by right multiplication, and hence it induces an algebra isomorphism
\begin{equation} \label{eqn:end-alg}
\bk[A]\simto\End(\pi_*\ubk_X).
\end{equation}
For any local system $\cL$ on $Y$, the projection formula gives an isomorphism $\pi_*\pi^*\cL\cong \cL\otimes\pi_*\ubk_X$, and the $A$-action on $\pi_*\pi^*\cL$ is the one induced by the $A$-action on $\pi_*\ubk_X$. 

Assume now that we are given a group automorphism $\theta$ of $A$, and automorphisms $\dot{\vartheta}$ of $X$ and $\vartheta$ of $Y$, which satisfy
\[
\pi \circ \dot{\vartheta} = \vartheta \circ \pi, \quad \text{and} \quad \dot{\vartheta}(a \cdot x) = \theta(a) \cdot \dot{\vartheta}(x)
\]
for all $a \in A$ and $x \in X$. Then for any $V \in \Rep(A,\bk)$ one can define a local system $\vartheta^* \cL_V$ on $Y$, and a representation $V^\theta$ of $A$ (which is isomorphic to $V$ as a vector space, with the action of $a\in A$ corresponding to the action of $\theta(a)$ on $V$).

\begin{lem}
\label{lem:theta^*-local-system}
We have a canonical isomorphism
$\vartheta^* \cL_V \cong \cL_{V^\theta}$.
\end{lem}

\begin{proof}
Consider the composition of natural isomorphisms of local systems
\[
\pi^* \vartheta^* \cL_V \cong (\vartheta \circ \pi)^* \cL_V = (\pi \circ \dot{\vartheta})^* \cL_V \cong\dot{\vartheta}^* \pi^* \cL_V \cong \dot{\vartheta}^* \underline{V}_X \cong \underline{V}_X.
\]
Using the fact that $\dot{\vartheta} \circ a = \theta(a) \circ \dot{\vartheta}$,
it is easily checked that $\varphi_a:\pi^* \vartheta^* \cL_V \simto a^* \pi^* \vartheta^* \cL_V$ corresponds, under this composition, to the isomorphism $\underline{V}_X \simto a^* \underline{V}_X$ whose composition with the canonical isomorphism $a^*\underline{V}_X\cong \underline{V}_X$ is the automorphism of $\underline{V}_X$ induced by the action of $\theta(a)$ on $V$. Hence we have defined a canonical isomorphism of $A$-equivariant local systems $\pi^* \vartheta^* \cL_V \cong \underline{(V^\theta)}_X$, and the lemma follows.
\end{proof}

Finally, we will need the following easy result.

\begin{lem}
\label{lem:local-systems-abs-irr}
Let $\cM$ be a local system on $Y$ such that $\pi^* \cM$ is absolutely irreducible. Then the functor $V \mapsto \cM \otimes \cL_V$ induces an equivalence of categories between $\Rep(A,\bk)$ and
the full subcategory of $\Loc(Y,\bk)$ whose objects are the local systems $\cL$ such that $\pi^* \cL$ is isomorphic to a direct sum of copies of $\pi^* \cM$.
\end{lem}

\begin{proof}
One can construct a functor in the reverse direction as follows: if $\cL$ is an object of $\Loc(Y,\bk)$ such that $\pi^* \cL$ is isomorphic to a direct sum of copies of $\pi^* \cM$ then one can consider the vector space $\Hom(\pi^* \cM, \pi^* \cL)$, endowed with the natural $A$-action induced by the $A$-equivariant structures on $\pi^* \cL$ and $\pi^* \cM$. Using the fact that $\Hom(\pi^* \cM, \pi^* \cM)=\bk$ (since $\pi^* \cM$ is absolutely irreducible), one can easily check that this provides an inverse to the functor of the lemma.
\end{proof}

\subsection{Preliminary results}
\label{ss:preliminary}

We will need the following analogues in the Lie algebra setting of results stated in the group setting in~\cite{lusztig} or~\cite{bonnafe-cgu}.

\begin{lem}
\label{lem:dim-lusztig}
Let $P \subset G$ be a parabolic subgroup with Levi factor $L$.
Let $x \in \fl$ and $y \in \fg$. Then
\[
\dim \{ gP \mid g^{-1} \cdot y \in (L \cdot x) + \fu_P \} \leq \frac{1}{2} \bigl( \dim G_y - \dim L_x \bigr).
\]
\end{lem}

\begin{proof}
The proof is entirely analogous to that of~\cite[(1.3.1)]{lusztig}. 
\end{proof}

\begin{lem}
\label{lem:component-groups}
Let $P \subset G$ be a parabolic subgroup with Levi factor $L$. Let $x \in \cO$ and $y \in \fu_P$ be such that $x+y \in \mathrm{Ind}_L^G(\cO)$. Then:
\begin{enumerate}
\item $G_{x+y}^\circ\subset P$, and hence the natural morphism $A_P(x+y)\to A_G(x+y)$ is injective;
\item the natural morphism $A_P(x+y)\to A_L(x)$ is surjective.
\end{enumerate}
\end{lem}

\begin{proof}
This proof is analogous to that of~\cite[Corollary 7.3(d)]{lusztig}. Since $P$ acts transitively on $\mathrm{Ind}_L^G (\cO) \cap (\cO + \fu_P)$ and $\dim\mathrm{Ind}_L^G (\cO)=\dim\cO+2\dim\fu_P$ (see~\cite[Theorem~7.1.1]{cm}), we have
\[
\dim P_{x+y}=\dim P - \dim(\cO + \fu_P)=\dim G-\dim\mathrm{Ind}_L^G (\cO)=\dim G_{x+y},
\]
proving (1). Now $L_x U_P$ acts transitively on $\mathrm{Ind}_L^G (\cO) \cap (x + \fu_P)$, which is irreducible (being dense in $x + \fu_P$), implying that $L_x^\circ U_P$ also acts transitively on it. Hence if $g \in L_x$, there exist $h \in L_x^\circ$ and $u \in U_P$ such that $g \cdot (x+y) = hu \cdot (x+y)$. Then $u^{-1} h^{-1} g \in P_{x+y}$, and its image in $A_L(x)$ is $gL_x^\circ$, which proves (2).
\end{proof}

The following lemma will in fact not be used until later sections, but we place it here to highlight the connection with Lemma~\ref{lem:component-groups}(1).

\begin{lem}
\label{lem:cgu}
Let $P \subset G$ be a parabolic subgroup with Levi factor $L$.  Let $x \in \cO$ and $y \in \fu_P$ be such that $x+y \in \mathrm{Ind}_L^G(\cO)$. If $\cO$ is a distinguished orbit in $\cN_L$, then $G_{x+y} \subset P$, and hence $A_P(x+y)=A_G(x+y)$.
\end{lem}

\begin{proof}
The proof is analogous to that of~\cite[Theorem (3)]{bonnafe-cgu}.
By~\cite[Theorem~8.2.6 and Corollary~7.1.7]{cm} there exists a distinguished parabolic subgroup $Q \subset L$, with Levi factor $M$, such that $x\in\fu_Q$ and $\cO$ is the Richardson orbit of $Q$, i.e.~$\cO=\mathrm{Ind}_{M}^L(\{0\})$. Then $Q':=QU_P$ is a parabolic subgroup of $G$, which is distinguished by~\cite[Proposition~2.3]{bonnafe-cgu}, and $x+y\in\fu_{Q'}$ belongs to the Richardson orbit of $Q'$ since $\mathrm{Ind}_L^G(\mathrm{Ind}_{M}^L(\{0\}))=\mathrm{Ind}_{M}^G(\{0\})$ (see~\cite[Proposition~7.1.4(ii)]{cm}). By~\cite[Theorem 2.2(e)]{bonnafe-cgu} we deduce that $G_{x+y} \subset Q' \subset P$.
\end{proof}

\subsection{Geometry}
\label{ss:geometry}

Following~\cite[6.11]{lusztig-cusp2}, we define
\[
T_{(L,\cO)}:=\bigcup_{L'\supset L} Y_{(L',\mathrm{Ind}_{L}^{L'}(\cO))},
\]
where the union is over Levi subgroups $L'$ of $G$ containing $L$, and $\mathrm{Ind}_{L}^{L'}(\cO)$ denotes the nilpotent orbit in $\cN_{L'}$ induced by $\cO$.

\begin{lem}
$T_{(L,\cO)}$ is an open subvariety of $X_{(L,\cO)}$.
\end{lem}

\begin{proof}
This is proved in~\cite[Proposition~6.12]{lusztig-cusp2} under the assumption that $\cO$ supports a cuspidal pair in characteristic $0$, but in fact the only step of the proof that uses that assumption can be easily seen to hold in general. (Namely, to show the equality labelled (c), instead of invoking~\cite[Lemma 6.10]{lusztig-cusp2}, simply observe that both sides equal $\mathrm{codim}_{\mathrm{Lie}(L')}\, C'$.)
\end{proof}

Clearly $T_{(L,\cO)}$ is $G$-stable. It contains $Y_{(L,\cO)}$ as an open subset and contains $Y_{(G,\mathrm{Ind}_L^G(\cO))}=\fz_G+\mathrm{Ind}_L^G(\cO)$ as a closed subset (closed by~\cite[Proposition 6.5]{lusztig-cusp2}). Note that for every element $y$ of $T_{(L,\cO)}$ we have
\begin{equation} \label{eqn:stabilizer-dim}
\dim G_y=\dim L-\dim \cO,
\end{equation}
since if $y\in x+\fz_{L'}^\circ$ for $x\in\mathrm{Ind}_L^{L'}(\cO)$, then $G_y^\circ=(L'_x)^\circ$.

\begin{rmk}
In the setting of the group $G$ rather than the Lie algebra $\fg$, the variety analogous to $T_{(L,\cO)}$ is the one denoted $X_{\mathrm{min}}$ by Bonnaf\'e (see~\cite[Remark 2.4]{bonnafe1} for the description of $X_{\mathrm{min}}$ analogous to the above definition of $T_{(L,\cO)}$). 
\end{rmk}

Now we fix a parabolic subgroup $P \subset G$ having $L$ as a Levi factor.  Recall that
\[
X_{(L,\cO)} = G  \cdot (\overline{\cO} + \fz_L + \fu_P).
\]
As in~\cite[\S2.6]{genspring1}, we also consider the variety
\[
\widetilde{X}_{(L,\cO)} := G \times^P (\overline{\cO} + \fz_L + \fu_P).
\]
Recall from~\cite[(2.12)]{genspring1} that we have a cartesian square
\begin{equation} \label{eqn:basic-square}
\vcenter{\xymatrix{
\widetilde{Y}_{(L,\cO)} \ar@{^{(}->}[r] \ar[d]_-{\varpi_{(L,\cO)}} & \widetilde{X}_{(L,\cO)} \ar[d]^-{\pi_{(L,\cO)}} \\
Y_{(L,\cO)} \ar@{^{(}->}[r] & X_{(L,\cO)}
}}
\end{equation}
where the horizontal morphisms are open embeddings.  The image of the top embedding is $G\times^P (\cO + \fz_L^\circ + \fu_P)$; see~\cite[Lemma 5.1.27]{letellier}. We set
\[
\widetilde{T}_{(L,\cO)} := \pi_{(L,\cO)}^{-1}(T_{(L,\cO)}),
\]
and denote by $\tau_{(L,\cO)} : \widetilde{T}_{(L,\cO)} \to T_{(L,\cO)}$ the restriction of $\pi_{(L,\cO)}$.
The following result is adapted from~\cite[Theorem~2.3(a)]{bonnafe1}.

\begin{prop}
\label{prop:relative-normalization}
The variety $\widetilde{T}_{(L,\cO)}$ is the normalization of $T_{(L,\cO)}$ relative to the Galois covering $\varpi_{(L,\cO)}$. In particular, this variety is independent of the choice of $P$ up to canonical isomorphism, and it is endowed with a natural action of $N_G(L,\cO)/L$ commuting with the $G$-action, such that $\tau_{(L,\cO)}$ is $G \times N_G(L,\cO)/L$-equivariant \textup{(}where $N_G(L,\cO)/L$ acts trivially on $T_{(L,\cO)}$\textup{)}.
\end{prop}

\begin{proof}
Consider the smooth open subvariety
\[
\widetilde{X}^\circ_{(L,\cO)} := G \times^P (\cO + \fz_L + \fu_P) \quad \subset \widetilde{X}_{(L,\cO)}.
\]
First, we claim that $\widetilde{T}_{(L,\cO)}$ is included in $\widetilde{X}^\circ_{(L,\cO)}$ (and hence is smooth). Indeed, let $x \in \overline{\cO}$, $z \in \fz_L$, and $y \in \fu_P$, and assume that $x+z+y \in Y_{(L',\mathrm{Ind}_{L}^{L'}(\cO))}$ for some $L' \supset L$. Then $\dim G_{x+z+y}= \dim L-\dim \cO$ by~\eqref{eqn:stabilizer-dim}. But we have $G_{x+z+y} \supset P_{x+z+y}$, and $P \cdot (x+z+y) \subset (L \cdot x)+z + \fu_P$, which implies that
\[
\dim G_{x+z+y} \geq \dim P - \dim\bigl( P \cdot (x+z+y) \bigr) \geq \dim L -\dim(L \cdot x).
\]
Hence we obtain finally that $\dim(L \cdot x) \geq \dim \cO$, which implies that $x \in \cO$ and finishes the proof of the claim.

Now we claim that $\tau_{(L,\cO)}$ is finite. Since this morphism is projective, it suffices to prove that it is quasi-finite. For $y\in T_{(L,\cO)}$, we have
\[
\tau_{(L,\cO)}^{-1}(y) \cong \{gP \in G/P \mid g^{-1}\cdot y \in \cO + \fz_L + \fu_P \}
\]
by the first claim, and it suffices to prove that the set on the right-hand side is finite. But 
the $\fz_L$ component of any element of $(G\cdot y)\cap(\cO+\fz_L+\fu_P)$ is the semisimple part of its projection to $\fl \cong \fp/\fu_P$, and hence is conjugate to the semisimple part $y_s$ of $y$ (by the Lie algebra analogue of~\cite[Lemma~1.6]{bonnafe1}). Since $(G \cdot y_s) \cap \fz_L$ is finite,
it suffices to show that
\[
\{gP \in G/P \mid g^{-1}\cdot y \in (L\cdot x) + \fu_P \}
\]
is finite for any $x\in\cO+\fz_L$, and this follows immediately from Lemma~\ref{lem:dim-lusztig} and~\eqref{eqn:stabilizer-dim}.

Finally, since we have a diagram
\begin{equation}
\label{eqn:zyt-diagram}
\vcenter{\xymatrix{
\widetilde{Y}_{(L,\cO)} \ar@{^{(}->}[r] \ar[d]_-{\varpi_{(L,\cO)}} & \widetilde{T}_{(L,\cO)} \ar[d]^-{\tau_{(L,\cO)}}  \\
Y_{(L,\cO)} \ar@{^{(}->}[r] & T_{(L,\cO)}
}}
\end{equation}
where the horizontal arrows are open embeddings, $\widetilde{T}_{(L,\cO)}$ is smooth, and $\tau_{(L,\cO)}$ is finite, we obtain that $\widetilde{T}_{(L,\cO)}$ is the normalization of $T_{(L,\cO)}$ relative to $\varpi_{(L,\cO)}$. The rest of the statement follows from the functoriality of relative normalization.
\end{proof}

The next statement is the analogue of~\cite[7.5(f)]{lusztig-cusp2}. The proof follows that of~\cite[Lemma~2.5 and Corollary~2.6]{bonnafe1}.

\begin{lem} \label{lem:pinch}
Assume that~\eqref{eqn:distinguished-substitute} holds.
Then $\tau_{(L,\cO)}$ restricts to an isomorphism $\tau_{(L,\cO)}^{-1}(\mathrm{Ind}_L^G(\cO))\simto\mathrm{Ind}_L^G(\cO)$. In particular, the action of $N_G(L,\cO)/L$ on $\tau_{(L,\cO)}^{-1}(\mathrm{Ind}_L^G(\cO))$ is trivial.
\end{lem}

\begin{proof}
Since the statement is independent of $P$, we can assume that $P$ satisfies the condition in~\eqref{eqn:distinguished-substitute}.
It suffices to prove that for any $y \in \mathrm{Ind}_L^G(\cO) \cap (\cO + \fu_P)$, the fibre $\tau_{(L,\cO)}^{-1}(y)$ is a single point. As in the proof of Proposition~\ref{prop:relative-normalization}, we have
\[
\tau_{(L,\cO)}^{-1}(y) \cong \{gP \in G/P \mid g^{-1}\cdot y \in \cO + \fu_P \}.
\]
Since $P$ acts transitively on $\mathrm{Ind}_L^G(\cO) \cap (\cO + \fu_P)$ (see~\cite[Theorem~7.1.1]{cm}), the right-hand side is $G_y P/P$, which is a single point by our assumption.
\end{proof}

Now choose $x\in\cO$ and define $\widetilde{Y}_{(L,\cO)}'=G\times^L(L/L_x^\circ\times\fz_L^\circ)$. Let $\nu_{(L,\cO)}:\widetilde{Y}_{(L,\cO)}'\to \widetilde{Y}_{(L,\cO)}$ be the map induced by the covering $L/L_x^\circ\to L/L_x\cong\cO$. Then $\nu_{(L,\cO)}$ is a Galois covering with group $A_L(x)$, and $\varpi_{(L,\cO)}\circ\nu_{(L,\cO)}$ is a Galois covering with group $A_{N_G(L)}(x)=N_G(L)_x/L_x^\circ$. Here $n\in N_G(L)_x$ acts on $\widetilde{Y}_{(L,\cO)}'$ by
\[
n \cdot (g * (mL_x^\circ,z))=gn^{-1}*(nmn^{-1}L_x^\circ,n\cdot z),\text{ for }g\in G,\, m\in L,\,z\in\fz_L^\circ,
\] 
and $\nu_{(L,\cO)}$ is the quotient map for the action of the normal subgroup $L_x$ (with the smaller normal subgroup $L_x^\circ$ acting trivially).

Recall the parabolic subgroup $P$ with Levi factor $L$. Following~\cite[Remark~3.3]{bonnafe1} we consider the $P$-action on $L/L_x^\circ\times \fz_L \times \fu_P$ in which $lu\in LU_P=P$ acts by
\[
(lu) \cdot (mL_x^\circ, z, v) := \Bigl( lmL_x^\circ, z, l \cdot \bigl( (um \cdot x - m \cdot x ) + (u \cdot z - z) + u \cdot v \bigr) \Bigr),
\]
and the smooth variety
\[
\widetilde{X}^{\circ \prime}_{(L,\cO)}:= G \times^P (L/L_x^\circ\times \fz_L \times \fu_P).
\]
Then we have a Galois covering $\mu_{(L,\cO)}$ with Galois group $A_L(x)$ and a diagram
\[
\xymatrix{
\widetilde{Y}_{(L,\cO)}' \ar[d]_-{\nu_{(L,\cO)}} \ar@{^{(}->}[rr] & & \widetilde{X}^{\circ \prime}_{(L,\cO)} \ar[d]^-{\mu_{(L,\cO)}} \\
\widetilde{Y}_{(L,\cO)} \ar@{^{(}->}[r] & \widetilde{T}_{(L,\cO)} \ar@{^{(}->}[r] & \widetilde{X}^\circ_{(L,\cO)}.
}
\]
 
We denote by $\widetilde{T}_{(L,\cO)}'$ the inverse image of $T_{(L,\cO)}$ in $\widetilde{X}^{\circ \prime}_{(L,\cO)}$, and by~$\sigma_{(L,\cO)}:\widetilde{T}_{(L,\cO)}'\to\widetilde{T}_{(L,\cO)}$ the restriction of $\mu_{(L,\cO)}$. By construction, this morphism is a Galois covering with Galois group $A_L(x)$.

The following result follows from the same arguments as for Proposition~\ref{prop:relative-normalization}.

\begin{lem} \label{lem:normalization2}
The variety $\widetilde{T}_{(L,\cO)}'$ is the normalization of $T_{(L,\cO)}$ relative to the Galois covering $\varpi_{(L,\cO)}\circ\nu_{(L,\cO)}$. In particular, this variety is independent of the choice of $P$ up to canonical isomorphism, and is endowed with a natural action of $A_{N_G(L)}(x)$ extending the action of $A_L(x)$ and commuting with the action of $G$, and such that $\sigma_{(L,\cO)}$ is $G \times A_{N_G(L)}(x)$-equivariant. \textup{(}Here $A_{N_G(L)}(x)$ acts on $\widetilde{T}_{(L,\cO)}$ via its quotient $A_{N_G(L)}(x)/A_L(x)\cong N_G(L,\cO)/L$.\textup{)}
\end{lem}

We summarize our constructions in the following diagram:
\begin{equation*}
\label{eqn:diagram-geometry}
\vcenter{
\xymatrix{
\widetilde{Y}_{(L,\cO)}' \ar[d]_-{\nu_{(L,\cO)}} \ar@{^{(}->}[r] & \widetilde{T}_{(L,\cO)}' \ar[d]_-{\sigma_{(L,\cO)}} \ar@{^{(}->}[r] & \widetilde{X}^{\circ \prime}_{(L,\cO)} \ar[d]^-{\mu_{(L,\cO)}} & \\
\widetilde{Y}_{(L,\cO)} \ar@{^{(}->}[r] \ar[d]_-{\varpi_{(L,\cO)}} & \widetilde{T}_{(L,\cO)} \ar[d]_-{\tau_{(L,\cO)}} \ar@{^{(}->}[r] & \widetilde{X}^\circ_{(L,\cO)} \ar@{^{(}->}[r] & \widetilde{X}_{(L,\cO)} \ar[d]^-{\pi_{(L,\cO)}} \\
Y_{(L,\cO)} \ar@{^{(}->}[r] & T_{(L,\cO)} \ar@{^{(}->}[rr] & & X_{(L,\cO)}.
}
}
\end{equation*}
Note once again that the first two columns of this diagram do not depend on the choice of the parabolic subgroup $P$ (but the varieties $\widetilde{X}^{\circ \prime}_{(L,\cO)}$, $\widetilde{X}^{\circ}_{(L,\cO)}$ and $\widetilde{X}_{(L,\cO)}$ \emph{do} depend on the choice of $P$).

\subsection{Local systems}
\label{subsect:extension-action}

We denote by  $\cE$ an irreducible $L$-equivariant local system on $\cO$. Choosing $x\in\cO$ as above, we obtain a corresponding irreducible representation $V=\cE_x$ of $A_L(x)$. Note that, up to isomorphism, $\cE$ is the local system on $\cO$ associated to $V$ by the functor~\eqref{eqn:local-systems-Galois} for the Galois covering $L/L_x^\circ\to L/L_x\cong\cO$; in other words, the $A_L(x)$-equivariant local system obtained by pulling $\cE$ back to $L/L_x^\circ$ is isomorphic to the constant sheaf $\underline{V}_{L/L_x^\circ}$. 
 
Recall that $\widetilde{\cE}$ is defined to be the unique $G$-equivariant local system on $\widetilde{Y}_{(L,\cO)}$ whose pull-back to $G\times(\cO+\fz_L^\circ)$ is $\ubk_G\boxtimes(\cE\boxtimes\ubk_{\fz_L^\circ})$. Alternatively, $\widetilde{\cE}$ is the local system associated to $V$ by the functor~\eqref{eqn:local-systems-Galois} for the Galois covering $\nu_{(L,\cO)}:\widetilde{Y}_{(L,\cO)}'\to \widetilde{Y}_{(L,\cO)}$. (To see this, observe that $(\nu_{(L,\cO)})^*\widetilde{\cE}$ is a $G\times A_L(x)$-equivariant local system on $\widetilde{Y}_{(L,\cO)}'$ whose pull-back to $G\times(L/L_x^\circ\times\fz_L^\circ)$ is $\ubk_G\boxtimes(\underline{V}_{L/L_x^\circ}\boxtimes\ubk_{\fz_L^\circ})$, hence $(\nu_{(L,\cO)})^*\widetilde{\cE}\cong\underline{V}_{\widetilde{Y}_{(L,\cO)}'}$.)

Given a parabolic subgroup $P\subset G$ with $L$ as Levi factor, we similarly define $\dot{\cE}$ to be the unique $G$-equivariant local system on $\widetilde{X}_{(L,\cO)}^\circ$ whose pull-back to $G\times(\cO+\fz_L+\fu_P)$ is $\ubk_G\boxtimes(\cE\boxtimes\ubk_{\fz_L}\boxtimes\ubk_{\fu_P})$. By similar reasoning, $\dot{\cE}$ is the local system associated to $V$ by the functor~\eqref{eqn:local-systems-Galois} for the Galois covering $\mu_{(L,\cO)}:\widetilde{X}^{\circ\prime}_{(L,\cO)}\to \widetilde{X}_{(L,\cO)}^\circ$. Clearly $\dot{\cE}$ is an extension of $\widetilde{\cE}$ under the open embedding $\widetilde{Y}_{(L,\cO)}\hookrightarrow \widetilde{X}_{(L,\cO)}^\circ$.

Let $\widehat{\cE}$ denote the restriction of $\dot{\cE}$ to $\widetilde{T}_{(L,\cO)}$, which is therefore also an extension of $\widetilde{\cE}$. Alternatively, $\widehat{\cE}$ is the local system associated to $V$ by the functor~\eqref{eqn:local-systems-Galois} for the Galois covering $\sigma_{(L,\cO)}:\widetilde{T}_{(L,\cO)}'\to \widetilde{T}_{(L,\cO)}$. Since this covering is independent of $P$ (see Proposition~\ref{prop:relative-normalization} and Lemma~\ref{lem:normalization2}), the local system $\widehat{\cE}$ is independent of $P$.

Now let $\cF$ denote the $G$-equivariant local system on $\tau_{(L,\cO)}^{-1}(\mathrm{Ind}_L^G(\cO))$ obtained by restricting $\widehat{\cE}$. We have the following analogue of part of~\cite[Lemma 7.10(a)]{lusztig-cusp2}:

\begin{lem} \label{lem:crucial}
Assume that~\eqref{eqn:distinguished-substitute} and~\eqref{eqn:abs-irr} hold. Then
the local system $\cF$ is absolutely irreducible. 
\end{lem}

\begin{proof}
Since the statement is independent of $P$, we can assume that $P$ satisfies the condition in~\eqref{eqn:distinguished-substitute}. In addition to choosing $x\in\cO$ as above, let $y\in\fu_P$ be such that $x+y\in\mathrm{Ind}_L^G(\cO)$. Then $G_{x+y}\subset P$. Recall from Lemma~\ref{lem:pinch} that $\tau_{(L,\cO)}^{-1}(\mathrm{Ind}_L^G(\cO))$ is the $G$-orbit of $1*(x+y)$, whose stabilizer in $G$ is $G_{x+y}$. Since $\cF$ is $G$-equivariant, it corresponds to some representation $V'$ of $A_G(x+y)=A_P(x+y)$. We must show that the representation $V'$ is absolutely irreducible. However, by construction, the representation $V'$ is obtained by pulling back the representation $V$ through the natural homomorphism $A_P(x+y)\to A_L(x)$ (compare~\cite[Corollary 7.4]{lusztig}; we have no induction of representations here, since $A_P(x+y)=A_G(x+y)$). This homomorphism is surjective by Lemma~\ref{lem:component-groups}(2), and $V$ is absolutely irreducible by assumption~\eqref{eqn:abs-irr}, so the claim follows.
\end{proof}

Recall that the $(N_G(L,\cO)/L)$-action on $\widetilde{Y}_{(L,\cO)}$ extends to $\widetilde{T}_{(L,\cO)}$ (see Proposition~\ref{prop:relative-normalization}). Hence for any $n\in N_G(L,\cO)/L$ we have a local system $n^*\widehat{\cE}$ on $\widetilde{T}_{(L,\cO)}$. If~\eqref{eqn:distinguished-substitute} holds, from Lemma~\ref{lem:pinch} we know that $N_G(L,\cO)/L$ acts trivially on $\tau_{(L,\cO)}^{-1}(\mathrm{Ind}_L^G(\cO))$, so the restriction of this local system to $\tau_{(L,\cO)}^{-1}(\mathrm{Ind}_L^G(\cO))$ is $\cF$. We then have the following analogue of~\cite[Lemma 7.10(b)(c)]{lusztig-cusp2}.

\begin{prop}
\label{prop:equivariance-hatE}
Assume that~\eqref{eqn:distinguished-substitute}, \eqref{eqn:abs-irr} and~\eqref{eqn:fixed} hold. Then for any $n\in N_G(L,\cO)/L$ there is a unique isomorphism
\[
\varphi_n:\widehat{\cE}\simto n^*\widehat{\cE}
\]
such that the induced automorphism of $\cF$ is the identity. Moreover, we have
\[
\varphi_{mn}=n^*(\varphi_{m})\circ\varphi_n
\]
for all $m,n \in N_G(L,\cO)/L$; in other words, the isomorphisms $\varphi_n$ constitute an $(N_G(L,\cO)/L)$-equivariant structure on $\widehat{\cE}$.
\end{prop}

\begin{proof}
Choose some $\dot{n}$ in $A_{N_G(L)}(x)$ whose image in $N_G(L,\cO)/L$ is $n$. Then we can apply Lemma~\ref{lem:theta^*-local-system} to the automorphism of $\widetilde{T}'_{(L,\cO)}$, resp.~$\widetilde{T}_{(L,\cO)}$, resp.~$A_L(x)$, induced by the action of $\dot{n}$, resp.~the action of $n$, resp.~the conjugation by $\dot{n}$, to obtain an isomorphism between $n^* \widehat{\cE}$ and the local system associated to the twist $V^{\dot{n}}$ of $V$ by the conjugation action of $\dot{n}$ on $A_L(x)$.
From assumption~\eqref{eqn:fixed} we deduce that $V^{\dot{n}}$ is isomorphic to $V$, so there exists an isomorphism $\varphi_n:\widehat{\cE}\simto n^*\widehat{\cE}$. 

By assumption~\eqref{eqn:abs-irr}, the representation $V$ is absolutely irreducible, and hence so is $\widehat{\cE}$. So for each $n\in N_G(L,\cO)/L$, the isomorphism $\varphi_n$ is unique up to scalar. In particular, $n^*(\varphi_{m})\circ\varphi_n$ must be a scalar multiple of $\varphi_{mn}$ for all $m,n\in N_G(L,\cO)/L$. 
By the remark before the statement of the proposition, $\varphi_n$ induces an automorphism of $\cF$; by Lemma~\ref{lem:crucial}, this induced automorphism must be a scalar multiplication, so we can uniquely normalize $\varphi_n$ to make it the identity. With this normalization, the equation $\varphi_{mn}=n^*(\varphi_{m})\circ\varphi_n$ is clear. 
\end{proof}

For the remainder of this section, we continue to assume that~\eqref{eqn:distinguished-substitute}, \eqref{eqn:abs-irr} and~\eqref{eqn:fixed} all hold. Restricting the isomorphisms $\varphi_n$ of Proposition~\ref{prop:equivariance-hatE} to $\widetilde{Y}_{(L,\cO)}$, we obtain a collection of isomorphisms
\[
\psi_n:\widetilde{\cE}\simto n^*\widetilde{\cE}
\]
for all $n\in N_G(L,\cO)/L$, satisfying the rule
\[
\psi_{mn}=n^*(\psi_{m})\circ\psi_n.
\]
In other words, by considering the extension $\widehat{\cE}$ to the larger variety $\widetilde{T}_{(L,\cO)}$, we have defined a \emph{canonical} $(N_G(L,\cO)/L)$-equivariant structure on $\widetilde{\cE}$. 

By the equivalence~\eqref{eqn:equivalence-equivariant}, this $(N_G(L,\cO)/L)$-equivariant structure on $\widetilde{\cE}$ gives rise to a local system $\overline{\cE}$ on $Y_{(L,\cO)}$ equipped with a canonical isomorphism
\begin{equation}
\label{eqn:isom-barE}
(\varpi_{(L,\cO)})^*\overline{\cE}\cong\widetilde{\cE}.
\end{equation}
Since $\widetilde{\cE}$ is absolutely irreducible by~\eqref{eqn:abs-irr}, we can apply Lemma~\ref{lem:local-systems-abs-irr} with $\cM=\overline{\cE}$.

Using the projection formula, from~\eqref{eqn:isom-barE} we deduce a canonical isomorphism
\begin{equation}
\label{eqn:isom-direct-image-tildeE}
(\varpi_{(L,\cO)})_* \widetilde{\cE} \cong \overline{\cE}\otimes(\varpi_{(L,\cO)})_*\ubk_{\widetilde{Y}_{(L,\cO)}}.
\end{equation}
The canonical $(N_G(L,\cO)/L)$-equivariant structure on $\widetilde{\cE}$ defines a canonical action of $N_G(L,\cO)/L$ on $(\varpi_{(L,\cO)})_* \widetilde{\cE}$. Under isomorphism~\eqref{eqn:isom-direct-image-tildeE}, this action is induced by the canonical action on $(\varpi_{(L,\cO)})_*\ubk_{\widetilde{Y}_{(L,\cO)}}$
(see Lemma~\ref{lem:reg-rep} and the subsequent comments). From~\eqref{eqn:end-alg} and Lemma~\ref{lem:local-systems-abs-irr} we deduce the algebra isomorphism
\begin{equation} \label{eqn:end-alg2}
\bk[N_G(L,\cO)/L]\simto\End\bigl((\varpi_{(L,\cO)})_* \widetilde{\cE}\bigr).
\end{equation}
This is the analogue of~\cite[Proposition 7.14]{lusztig-cusp2}. (Note that, in isolation,~\eqref{eqn:end-alg2} is a less conclusive statement in the modular case, since $(\varpi_{(L,\cO)})_* \widetilde{\cE}$ is not semisimple.)

\begin{lem} \label{lem:trivial-case}
If $\cE=\ubk_{\cO}$, then $\overline{\cE}=\ubk_{Y_{(L,\cO)}}$.
\end{lem}

\begin{proof}
In this case $\widehat{\cE}=\ubk_{\widetilde{T}_{(L,\cO)}}$ and the isomorphism $\varphi_n$ defined in Proposition~\ref{prop:equivariance-hatE} is clearly the canonical one. Hence $\widetilde{\cE}=\ubk_{\widetilde{Y}_{(L,\cO)}}$ as an $(N_G(L,\cO)/L)$-equivariant local system, so $\overline{\cE}=\ubk_{Y_{(L,\cO)}}$.
\end{proof}

\begin{rmk} \label{rmk:indeterminacy}
When $\cE$ is nontrivial, it is a challenging problem to describe concretely the local system $\overline{\cE}$ on $Y_{(L,\cO)}$ (for example, by specifying explicitly the corresponding representation $\overline{V}$ of $A_{N_G(L)}(x)$ -- see Remark~\ref{rmk:cocycle}). Even in Lusztig's setting where $(\cO,\cE)$ is a characteristic-$0$ cuspidal pair, this problem is unsolved in general, although some cases were settled by Bonnaf\'e in~\cite{bonnafe1,bonnafe2}.  
\end{rmk}

\subsection{Proof of Theorem~\ref{thm:cocycle}}
\label{ss:simple-quotients}

Finally we are in a position to prove Theorem~\ref{thm:cocycle}. 
We continue to assume~\eqref{eqn:distinguished-substitute}, \eqref{eqn:abs-irr} and~\eqref{eqn:fixed}.

Using~\eqref{eqn:isom-direct-image-tildeE}, from Lemma~\ref{lem:local-systems-abs-irr} we deduce properties $(1)$ and $(4)$ of Theorem~\ref{thm:cocycle} for the above local system $\overline{\cE}$. Properties $(2)$ and $(3)$ already emerged from the above discussion (see~\eqref{eqn:isom-barE},~\eqref{eqn:isom-direct-image-tildeE} and~\eqref{eqn:end-alg2}). Hence what remains to be proved is that $(\varpi_{(L,\cO)})_*\widetilde{\cE}$ has a unique direct summand whose $\IC$-extension has non-zero restriction to $\mathrm{Ind}_L^G(\cO)$, appearing with multiplicity one, and that $\overline{\cE}$ is the head of this direct summand.

First, from $(3)$ and Lemma~\ref{lem:local-systems-abs-irr} we deduce that the indecomposable direct summands of $(\varpi_{(L,\cO)})_* \widetilde{\cE}$ are of the form $\overline{\cE}\otimes \cL_Q$ for $Q$ an indecomposable direct summand of $\bk[N_G(L,\cO)/L]$, i.e.~an indecomposable projective $\bk[N_G(L,\cO)/L]$-module. 
Now, let $j : Y_{(L,\cO)} \to T_{(L,\cO)}$ be the inclusion, and let $d:=\dim(Y_{(L,\cO)})$. Then the restriction of $\IC(Y_{(L,\cO)}, (\varpi_{(L,\cO)})_* \widetilde{\cE})$ to the induced orbit $\mathrm{Ind}_L^G(\cO)$ is the same as that of $j_{!*} \bigl( (\varpi_{(L,\cO)})_* \widetilde{\cE} [d]\bigr)$. Moreover, since $\tau_{(L,\cO)}$ is finite, $\widetilde{T}_{(L,\cO)}$ is smooth, and $\widehat{\cE}$ extends $\widetilde{\cE}$, there is a canonical isomorphism of perverse sheaves
\begin{equation*}
j_{!*} \bigl( (\varpi_{(L,\cO)})_* \widetilde{\cE} [d] \bigr) \cong (\tau_{(L,\cO)})_* \widehat{\cE} [d].
\end{equation*}
By definition, the restriction of the right-hand side to $\mathrm{Ind}_L^G(\cO)$ is $\cF[d]$, which is indecomposable by Lemma~\ref{lem:crucial}. (Here we abuse notation by denoting also by $\cF$ the local system on $\mathrm{Ind}_L^G(\cO)$ corresponding to the previous $\cF$ on the isomorphic variety $\tau_{(L,\cO)}^{-1}(\mathrm{Ind}_L^G(\cO))$.) This implies that there exists a unique direct summand $\cG$ of $(\varpi_{(L,\cO)})_* \widetilde{\cE}$ whose $\IC$-extension has a nonzero restriction to $\mathrm{Ind}_L^G(\cO)$ (appearing with multiplicity one), and that the restriction of $\cG$ to $\mathrm{Ind}_L^G(\cO)$ is $\cF[d]$. 

Let $Q$ be the corresponding indecomposable projective $\bk[N_G(L,\cO)/L]$-module, so that $\cG \cong \overline{\cE} \otimes \cL_Q$, and consider the vector space
$\Hom \bigl( (\varpi_{(L,\cO)})_* \widetilde{\cE}, \cG \bigr)$.
This vector space has a canonical action of $N_G(L,\cO)/L$ induced by the action on $(\varpi_{(L,\cO)})_* \widetilde{\cE}$, and using Lemma~\ref{lem:local-systems-abs-irr} it is easily checked that this $\bk[N_G(L,\cO)/L]$-module is isomorphic to $Q$. Now consider the following morphism:
\[
\Hom \bigl( (\varpi_{(L,\cO)})_* \widetilde{\cE}, \cG \bigr) \xrightarrow{j_{!*}(\cdot [d])} \Hom\bigl( (\tau_{(L,\cO)})_* \widehat{\cE}[d], j_{!*}(\cG[d]) \bigr) \to \End(\cF[d]) \cong \bk,
\]
where the second arrow is induced by restriction to $\mathrm{Ind}_L^G(\cO)$. Then by construction this morphism is nonzero, and $(N_G(L,\cO)/L)$-equivariant if $N_G(L,\cO)/L$ acts trivially on the right-hand side. It follows that $Q$ is the projective cover of the trivial representation, which concludes the proof.

\section{Induction and restriction}
\label{sect:transitivity}

Consider a chain of Levi subgroups $L \subset M \subset G$.  Let $P \subset Q$ be parabolic subgroups of $G$ whose Levi factors are $L$ and $M$, respectively.  Also, let $R = P \cap M$.  

Fix a pair $(\cO,\cE) \in \fN_{L,\bk}$ satisfying the following conditions:
\begin{enumerate}
\item for any $y \in \mathrm{Ind}_L^M(\cO) \cap (\cO + \fu_R)$, $M_y \subset R$;
\item for any $y \in \mathrm{Ind}_L^G(\cO) \cap (\cO + \fu_P)$, $G_y \subset P$;
\item $\cE$ is absolutely irreducible;
\item the isomorphism class of $\cE$ is fixed by $N_G(L,\cO)$.
\end{enumerate}
The last condition implies, of course, that the isomorphism class of $\cE$ is also fixed by the smaller group $N_M(L,\cO)$.  These conditions say that $(L,\cO)$ satisfies \eqref{eqn:distinguished-substitute}--\eqref{eqn:fixed} with respect to both $M$ and $G$, so we can invoke Theorem~\ref{thm:cocycle} in both settings. Notice that if $\cO$ is a distinguished nilpotent orbit for $L$, then conditions $(1)$ and $(2)$ follow from Lemma~\ref{lem:cgu}.

In this section we prove that the resulting objects (local systems, equivariant structures, group actions) are compatible with induction from $M$ to $G$; see Theorems~\ref{thm:compatible-equivariance},~\ref{thm:uind} and~\ref{thm:uind-indec}. This culminates in Theorem~\ref{thm:restriction-compatibility}, a modular version of Lusztig's restriction theorem~\cite[Theorem 8.3]{lusztig}, which one would expect to need in order to determine the modular generalized Springer correspondence. In the present paper, the only application of these results is in Section~\ref{sec:det2}, where Theorem~\ref{thm:uind-indec} is used; in particular, they are not needed for the proof of Theorem~\ref{thm:main}. 

\subsection{Notation}

We will need notation for several versions of the varieties $Y_{(L,\cO)}$, $\widetilde{Y}_{(L,\cO)}$, etc. Define
\[
\fz_L^\bullet := \{z \in \fz_L \mid M_z^\circ = L \},
\]
an open subset of $\fz_L$ containing $\fz_L^\circ$.
Recalling that $Y_{(L,\cO)}=G \cdot(\cO + \fz_L^\circ)$, we define
\[
Y_{(L,\cO)}^M := M \cdot(\cO + \fz_L^\bullet)
\qquad\text{and}\qquad
Y_{(L,\cO)}^{M;G} := M \cdot(\cO + \fz_L^\circ).
\]
Likewise, $\widetilde{Y}_{(L,\cO)}^M$, $T_{(L,\cO)}^M$, etc., are defined in the same way as $\widetilde{Y}_{(L,\cO)}$, $T_{(L,\cO)}$, etc., but with the roles of $P$, $G$, and $\fz_L^\circ$ replaced by $R$, $M$, and $\fz_L^\bullet$.  The subset $\widetilde{Y}_{(L,\cO)}^{M;G} \subset \widetilde{Y}_{(L,\cO)}^M$ is defined similarly, except that we retain $\fz_L^\circ$.  Note that $\widetilde{Y}_{(L,\cO)}^{M;G}$ is a $(N_M(L,\cO)/L)$-stable dense open subset of $\widetilde{Y}_{(L,\cO)}^M$, and hence $Y_{(L,\cO)}^{M;G}$ is a dense open subset of $Y_{(L,\cO)}^M$, the image of $\widetilde{Y}_{(L,\cO)}^{M;G}$ under the Galois covering $\varpi_{(L,\cO)}^M:\widetilde{Y}_{(L,\cO)}^M\to Y_{(L,\cO)}^M$. Let $\varpi_{(L,\cO)}^{M;G}:\widetilde{Y}_{(L,\cO)}^{M;G}\to Y_{(L,\cO)}^{M;G}$ denote the restriction of $\varpi_{(L,\cO)}^M$.

Consider the varieties
\[
\begin{split}
\breve{Y}_{(L,\cO)} &:= G \times^M (M\cdot (\cO + \fz_L^\circ)) = G\times^M Y_{(L,\cO)}^{M;G},\\
\breve{X}_{(L,\cO)} &:= G \times^Q (Q\cdot (\overline{\cO} + \fz_L + \fu_P)) = G\times^Q (X_{(L,\cO)}^M+\fu_Q).
\end{split}
\]
Here, the last equality comes from the fact that $\fu_P=\fu_R+\fu_Q$, so
\[
Q\cdot (\overline{\cO} + \fz_L + \fu_P)=M\cdot(\overline{\cO} + \fz_L + \fu_R)+\fu_Q.
\]
We have a diagram of cartesian squares analogous to that in~\cite[\S 8.4]{lusztig}:
\begin{equation} \label{eqn:lusztig-diag}
\vcenter{\xymatrix@C=1.5cm{
\widetilde{Y}_{(L,\cO)} \ar[d]_\psi \ar@{^{(}->}[r] \ar@/_6ex/[dd]_{\varpi_{(L,\cO)}} &
  \widetilde{X}_{(L,\cO)} \ar[d]^\chi \ar@/^6ex/[dd]^{\pi_{(L,\cO)}} \\
\breve{Y}_{(L,\cO)} \ar[d]_{\varphi} \ar@{^{(}->}[r]  &
  \breve{X}_{(L,\cO)} \ar[d]^{\upsilon} \\
Y_{(L,\cO)} \ar@{^{(}->}[r] &
  X_{(L,\cO)}.
  }}
\end{equation}
The outer square is~\eqref{eqn:basic-square}. The map $\psi:\widetilde{Y}_{(L,\cO)}\to\breve{Y}_{(L,\cO)}$ is induced by $\varpi_{(L,\cO)}^{M;G}$, using the obvious identification $\widetilde{Y}_{(L,\cO)}=G\times^M \widetilde{Y}_{(L,\cO)}^{M;G}$. The maps $\varphi$, $\chi$ and $\upsilon$ are the natural ones, and the middle open embedding is provided by the following result. It is trivial that the diagram commutes, and the top and bottom squares are cartesian because the outer square~\eqref{eqn:basic-square} is cartesian and $\chi$ is surjective.

\begin{lem} \label{lem:generalized-letellier}
The natural map $G\times^M \mathfrak{m}\to G\times^Q \mathfrak{q}$ induces an isomorphism
\[
\breve{Y}_{(L,\cO)}=G\times^M Y_{(L,\cO)}^{M;G}\simto G\times^Q (Y_{(L,\cO)}^{M;G}+\fu_Q).
\]
\end{lem}

\begin{proof}
In the $M=L$ case this is~\cite[Lemma 5.1.27]{letellier}, and the proof in general is similar. One need only check that no nontrivial element of $U_Q$ belongs to the stabilizer of an element of $Y_{(L,\cO)}^{M;G}$; but the identity component of this stabilizer is contained in $M$ by the definition of $\fz_L^\circ$.
\end{proof}  

\subsection{Compatibility of actions and equivariant structures}

Recall that the map $\varpi_{(L,\cO)}$ is a Galois covering, the quotient of a free $(N_G(L,\cO)/L)$-action on $\widetilde{Y}_{(L,\cO)}$. Under the identification $\widetilde{Y}_{(L,\cO)}=G\times^M \widetilde{Y}_{(L,\cO)}^{M;G}$, the $(N_M(L,\cO)/L)$-action on $\widetilde{Y}_{(L,\cO)}$ obtained by restricting this $(N_G(L,\cO)/L)$-action corresponds to the $(N_M(L,\cO)/L)$-action on $G\times^M \widetilde{Y}_{(L,\cO)}^{M;G}$ induced by that on $\widetilde{Y}_{(L,\cO)}^{M;G}$. Hence the map $\psi$ is the quotient map for the $(N_M(L,\cO)/L)$-action on $\widetilde{Y}_{(L,\cO)}$, and it is a Galois covering with group $N_M(L,\cO)/L$. The map $\varphi$ is \'etale but not Galois in general, since $N_M(L,\cO)/L$ is not necessarily normal in $N_G(L,\cO)/L$.

Theorem~\ref{thm:cocycle} applied to $L\subset G$ gives us a canonical local system $\overline{\cE}$ on $Y_{(L,\cO)}$ and a corresponding canonical $(N_G(L,\cO)/L)$-equivariant structure on the local system $\widetilde{\cE}$ on $\widetilde{Y}_{(L,\cO)}$ (see the arguments following Proposition~\ref{prop:equivariance-hatE}). Applied to $L\subset M$, the same theorem gives us a local system $\overline{\cE}^M$ on $Y_{(L,\cO)}^M$, and a corresponding $(N_M(L,\cO)/L)$-equivariant structure on the local system $\widetilde{\cE}^M$ on $\widetilde{Y}_{(L,\cO)}^M$. Restricting to open subsets, we obtain a local system $\overline{\cE}^{M;G}$ on $Y_{(L,\cO)}^{M;G}$, and a corresponding $(N_M(L,\cO)/L)$-equivariant structure on the restriction $\widetilde{\cE}^{M;G}$ of $\widetilde{\cE}^M$ to $\widetilde{Y}_{(L,\cO)}^{M;G}$. Let $G \times^M \overline{\cE}^{M;G}$ denote the unique $G$-equivariant local system on $\breve{Y}_{(L,\cO)}$ whose pull-back to $G \times Y_{(L,\cO)}^{M;G}$ is $\ubk_G \boxtimes \overline{\cE}^{M;G}$. To this corresponds an $(N_M(L,\cO)/L)$-equivariant structure on the analogously-defined local system $G \times^M \widetilde{\cE}^{M;G}$ on $G\times^M \widetilde{Y}_{(L,\cO)}^{M;G}$. Under the identification $G\times^M \widetilde{Y}_{(L,\cO)}^{M;G}=\widetilde{Y}_{(L,\cO)}$, the local system $G \times^M \widetilde{\cE}^{M;G}$ is identified with $\widetilde{\cE}$, so we end up with an $(N_M(L,\cO)/L)$-equivariant structure on $\widetilde{\cE}$.

Our first goal is to prove:

\begin{thm} \label{thm:compatible-equivariance}
The $(N_M(L,\cO)/L)$-equivariant structure on $\widetilde{\cE}$ obtained as above, via the identification of $\widetilde{Y}_{(L,\cO)}$ with $G\times^M \widetilde{Y}_{(L,\cO)}^{M;G}$, is the restriction of the canonical $(N_G(L,\cO)/L)$-equivariant structure. Consequently, we have an isomorphism 
\[ G \times^M \overline{\cE}^{M;G}\cong\varphi^*\overline{\cE} \] 
of local systems on $\breve{Y}_{(L,\cO)}$.
\end{thm}

\subsection{Further geometry}

Since the canonical $(N_G(L,\cO)/L)$-equivariant structure on $\widetilde{\cE}$ is defined using the extension $\widehat{\cE}$ of $\widetilde{\cE}$ to $\widetilde{T}_{(L,\cO)}$ (see Proposition~\ref{prop:equivariance-hatE}), to prove Theorem~\ref{thm:compatible-equivariance} we need to relate $T_{(L,\cO)}$, $\widetilde{T}_{(L,\cO)}$ to the corresponding varieties $T_{(L,\cO)}^M$, $\widetilde{T}_{(L,\cO)}^M$. Hence we introduce
\[
\breve{T}_{(L,\cO)} := \upsilon^{-1}(T_{(L,\cO)})=G\times^Q \bigl( (X_{(L,\cO)}^{M}+\fu_Q)\cap T_{(L,\cO)} \bigr).
\]

\begin{lem}
We have an inclusion
\[ 
\breve{T}_{(L,\cO)}
\subset G\times^Q(T_{(L,\cO)}^M + \fu_Q)
\]
of open subsets of $\breve{X}_{(L,\cO)}$.
\end{lem}

\begin{proof}
The $M=L$ case is the statement that $\widetilde{T}_{(L,\cO)}\subset \widetilde{X}_{(L,\cO)}^\circ$, which was the first step of the proof of Proposition~\ref{prop:relative-normalization}. The general case is proved similarly.\end{proof}

Now the preimage $\chi^{-1}(\breve{T}_{(L,\cO)})$ of $\breve{T}_{(L,\cO)}$ in $\widetilde{X}_{(L,\cO)}$ is by definition $\widetilde{T}_{(L,\cO)}$. We need to describe the larger subset $\chi^{-1}(G\times^Q(T_{(L,\cO)}^M + \fu_Q))$. Note that we have obvious isomorphisms
\[
Q\times^P \mathfrak{p}\cong M\times^R \mathfrak{p}\cong (M\times^R \mathfrak{r})\times\fu_Q,
\]
coming respectively from $Q=MU_Q$, $P=RU_Q$ and from $\mathfrak{p}=\mathfrak{r}+\fu_Q$ (and the fact that the $R$-action on $\fu_Q$ extends to $M$). We use these isomorphisms to define a $Q$-action on $(M\times^R \mathfrak{r})\times\fu_Q$: the Levi factor $M\subset Q$ acts in the obvious diagonal way, and an element $u\in U_Q$ acts by the rule
\begin{equation} \label{eqn:uq-action}
u\cdot(m*x,y)=(m*x,u\cdot(m\cdot x) - m\cdot x + u\cdot y), \text{ for }m\in M,\, x\in\mathfrak{r},\, y\in\fu_Q. 
\end{equation}
In particular, we can identify $Q\times^P(\overline{\cO}+\fz_L+\fu_P)$ with $\widetilde{X}_{(L,\cO)}^M\times\fu_Q$ and hence identify $\widetilde{X}_{(L,\cO)}=G\times^P(\overline{\cO}+\fz_L+\fu_P)=G\times^Q \bigl( Q\times^P(\overline{\cO}+\fz_L+\fu_P) \bigr)$ with $G\times^Q(\widetilde{X}_{(L,\cO)}^M\times\fu_Q)$. Under this identification, the map $\chi$ becomes the map $G\times^Q(\widetilde{X}_{(L,\cO)}^M\times\fu_Q)\to G\times^Q(X_{(L,\cO)}^M+\fu_Q)$ induced by $\pi_{(L,\cO)}^M:\widetilde{X}_{(L,\cO)}^M\to X_{(L,\cO)}^M$. Hence the preimage $\chi^{-1}(G\times^Q(T_{(L,\cO)}^M + \fu_Q))$ is identified with $G\times^Q(\widetilde{T}_{(L,\cO)}^M \times \fu_Q)$, in such a way that the restriction $\sigma$ of $\chi$ to this preimage is the map induced by $\tau_{(L,\cO)}^M:\widetilde{T}_{(L,\cO)}^M\to T_{(L,\cO)}^M$.

To sum up, we have expanded~\eqref{eqn:lusztig-diag} into a diagram of cartesian squares in which all the horizontal maps are open embeddings:
\begin{equation} \label{eqn:expanded-diag}
\vcenter{
\xymatrix@C=0.7cm{
\widetilde{Y}_{(L,\cO)} \ar@{^{(}->}[r] \ar[d]_-{\psi} \ar@/_4ex/[dd]_{\varpi_{(L,\cO)}} & 
\widetilde{T}_{(L,\cO)} \ar@{^{(}->}[r] \ar[d]^-{\rho} \ar@/_4ex/[dd]|\hole_(.25){\tau_{(L,\cO)}} &
G\times^Q(\widetilde{T}_{(L,\cO)}^M \times \fu_Q)  \ar@{^{(}->}[r] \ar[d]^-{\sigma} &
\widetilde{X}_{(L,\cO)} \ar[d]^-{\chi} \ar@/^4ex/[dd]^{\pi_{(L,\cO)}} \\
\breve{Y}_{(L,\cO)} \ar@{^{(}->}[r] \ar[d]_-{\varphi} & 
\breve{T}_{(L,\cO)} \ar@{^{(}->}[r] \ar[d]^-{\mu} & 
G\times^Q(T_{(L,\cO)}^M + \fu_Q)  \ar@{^{(}->}[r] & 
\breve{X}_{(L,\cO)} \ar[d]^-{\upsilon} \\
Y_{(L,\cO)} \ar@{^{(}->}[r] &
T_{(L,\cO)} \ar@{^{(}->}[rr] &&
X_{(L,\cO)}.
}
}
\end{equation}
All the vertical maps here are finite except for those on the right-hand side (that is, except for $\chi$, $\upsilon$ and $\pi_{(L,\cO)}$). 

\subsection{Proof of Theorem~\ref{thm:compatible-equivariance}}

Since $U_Q$ acts trivially on the $M\times^R \mathfrak{r}$ component in the action~\eqref{eqn:uq-action}, the local system $\widehat{\cE}^M\boxtimes\ubk_{\fu_Q}$ on $\widetilde{T}_{(L,\cO)}^M\times\fu_Q$ is not just $M$-equivariant but $Q$-equivariant, and we have a well-defined local system $G\times^Q(\widehat{\cE}^M\boxtimes\ubk_{\fu_Q})$ on $G\times^Q(\widetilde{T}_{(L,\cO)}^M \times \fu_Q)$, whose restriction to $\widetilde{T}_{(L,\cO)}$ is easily seen to be $\widehat{\cE}$.

The $(N_M(L,\cO)/L)$-action on $\widetilde{T}_{(L,\cO)}^M$ defined by Proposition~\ref{prop:relative-normalization} commutes with the action of $M$ and preserves each fibre of the map $\tau_{(L,\cO)}^M$. So the induced $(N_M(L,\cO)/L)$-action on $\widetilde{T}_{(L,\cO)}^M\times\fu_Q$ (with trivial action on the second factor) commutes both with the $M$-action and with the $U_Q$-action defined by~\eqref{eqn:uq-action}; that is, it commutes with the whole $Q$-action. Hence it induces an $(N_M(L,\cO)/L)$-action on $G\times^Q(\widetilde{T}_{(L,\cO)}^M \times \fu_Q)$ which commutes with the $G$-action and preserves each fibre of the map $\sigma$. This in turn induces an $(N_M(L,\cO)/L)$-action on the subset $\widetilde{T}_{(L,\cO)}$ (which is a union of fibres of $\sigma$). By definition of the horizontal embeddings, this action extends the $(N_M(L,\cO)/L)$-action on $\widetilde{Y}_{(L,\cO)}$ viewed as $G\times^M \widetilde{Y}_{(L,\cO)}^{M;G}$. 
Therefore it coincides with the restriction of the $(N_G(L,\cO)/L)$-action on $\widetilde{T}_{(L,\cO)}$ defined by Proposition~\ref{prop:relative-normalization}, since the same compatibility holds for the actions on the dense subset $\widetilde{Y}_{(L,\cO)}$.

To prove Theorem~\ref{thm:compatible-equivariance} it suffices to prove that two $(N_M(L,\cO)/L)$-equivariant structures on $\widehat{\cE}$ are the same. The first is induced by the $(N_M(L,\cO)/L)$-equivariant structure on $G\times^Q(\widehat{\cE}^M\boxtimes\ubk_{\fu_Q})$, which in turn is induced by the $(N_M(L,\cO)/L)$-equivariant structure on $\widehat{\cE}^M$ defined by Proposition~\ref{prop:equivariance-hatE}. The second is the restriction of the $(N_G(L,\cO)/L)$-equivariant structure on $\widehat{\cE}$ defined by Proposition~\ref{prop:equivariance-hatE}. Hence it suffices to prove that the first $(N_M(L,\cO)/L)$-equivariant structure induces the trivial $(N_M(L,\cO)/L)$-equivariant structure on the restriction of $\widehat{\cE}$ to $\tau_{(L,\cO)}^{-1}(\mathrm{Ind}_L^G(\cO))$, and for this it suffices to prove the following statement:
\begin{equation*}
\begin{split}
&\text{the embedding }\widetilde{T}_{(L,\cO)}\hookrightarrow G\times^Q(\widetilde{T}_{(L,\cO)}^M\times\fu_Q)\\
&\text{maps }\tau_{(L,\cO)}^{-1}(\mathrm{Ind}_L^G(\cO))\text{ into } G\times^Q((\tau_{(L,\cO)}^M)^{-1}(\mathrm{Ind}_L^M(\cO))\times\fu_Q).
\end{split}
\end{equation*}
Then from~\eqref{eqn:expanded-diag} we see that it suffices to prove that
the embedding 
$\breve{T}_{(L,\cO)}\hookrightarrow G\times^Q(T_{(L,\cO)}^M+\fu_Q)$
maps 
$\mu^{-1}(\mathrm{Ind}_L^G(\cO))$
into 
$G\times^Q(\mathrm{Ind}_L^M(\cO)+\fu_Q)$,
or in other words that
\begin{equation} \label{eqn:inclusion}
(X_{(L,\cO)}^M+\fu_Q)\cap\mathrm{Ind}_L^G(\cO)\subset \mathrm{Ind}_L^M(\cO)+\fu_Q.
\end{equation}
But this is easy. The left-hand side of~\eqref{eqn:inclusion} is unchanged if $X_{(L,\cO)}^M$ is replaced by its intersection with $\cN_M$, namely $\overline{\mathrm{Ind}_L^M(\cO)}$. For any $M$-orbit $\cO'$ in $\overline{\mathrm{Ind}_L^M(\cO)}\setminus \mathrm{Ind}_L^M(\cO)$, $\cO'+\fu_Q$ is contained in the closure of the orbit $\mathrm{Ind}_M^G(\cO')$ whose codimension in $\cN_G$ is $\dim \cN_M-\dim\cO'$, strictly greater than the codimension $\dim \cN_M-\dim \mathrm{Ind}_L^M(\cO)$ of $\mathrm{Ind}_M^G(\mathrm{Ind}_L^M(\cO))=\mathrm{Ind}_L^G(\cO)$. Hence $\cO'+\fu_Q$ does not intersect $\mathrm{Ind}_L^G(\cO)$, proving~\eqref{eqn:inclusion}.

\subsection{Induction isomorphisms}

By~\cite[Proposition 2.17]{genspring1}, there is a canonical isomorphism
\begin{equation} \label{eqn:g-uind}
\uInd_{L\subset P}^G\bigl(\IC(\cO+\fz_L,\cE\boxtimes\ubk_{\fz_L})\bigr)\cong\IC\bigl(Y_{(L,\cO)},(\varpi_{(L,\cO)})_*\widetilde{\cE}\bigr).
\end{equation}
The same result applied to $M$ instead of $G$ gives us a canonical isomorphism
\begin{equation} \label{eqn:m-uind}
\uInd_{L\subset R}^M\bigl(\IC(\cO+\fz_L,\cE\boxtimes\ubk_{\fz_L})\bigr)\cong\IC\bigl(Y_{(L,\cO)}^M,(\varpi_{(L,\cO)}^M)_*\widetilde{\cE}^M\bigr).
\end{equation}
There is also a natural isomorphism of functors expressing the transitivity of induction:
\begin{equation} \label{eqn:trans-uind}
\uInd_{L\subset P}^G\cong\uInd_{M\subset Q}^G\circ\uInd_{L\subset R}^M:D_L^b(\mathfrak{l},\bk)\to D_G^b(\fg,\bk).
\end{equation}
This isomorphism is defined by a standard diagram, exactly analogous to~\cite[(7.6)]{ahr} but with the groups $T,C,B,L,P$ replaced respectively by $L,R,P,M,Q$, and nilpotent cones replaced by Lie algebras throughout. 

Combining~\eqref{eqn:g-uind},~\eqref{eqn:m-uind} and~\eqref{eqn:trans-uind}, we obtain a canonical isomorphism
\begin{equation} \label{eqn:uind}
\uInd_{M\subset Q}^G\bigl(\IC\bigl(Y_{(L,\cO)}^M,(\varpi_{(L,\cO)}^M)_*\widetilde{\cE}^M\bigr)\bigr)\cong 
\IC\bigl(Y_{(L,\cO)},(\varpi_{(L,\cO)})_*\widetilde{\cE}\bigr).
\end{equation}  
On the left-hand side of~\eqref{eqn:uind} we have an action of $N_M(L,\cO)/L$, induced functorially by the action on $(\varpi_{(L,\cO)}^M)_*\widetilde{\cE}^M$ derived from the canonical $(N_M(L,\cO)/L)$-equivariant structure on $\widetilde{\cE}^M$. On the right-hand side of~\eqref{eqn:uind}, likewise, we have an action of $N_G(L,\cO)/L$. The following crucial result says that these actions are compatible. (A special case of Theorem~\ref{thm:uind} was used in determining the modular generalized Springer correspondence for $\GL(n)$; see~\cite[proof of Lemma 3.11]{genspring1}. That case was relatively easy because the local system $\cE$ was trivial and the groups $N_M(L,\cO)/L$ and $N_G(L,\cO)/L$ were the same.) 

\begin{thm} \label{thm:uind}
The isomorphism~\eqref{eqn:uind} is $(N_M(L,\cO)/L)$-equivariant, in the sense that the $(N_M(L,\cO)/L)$-action on the left-hand side corresponds to the restriction of the $(N_G(L,\cO)/L)$-action on the right-hand side.
\end{thm}

\begin{proof}
The special case of this result where $L$ is a maximal torus $T$ (and thus, necessarily, $\cO=\{0\}$ and $\cE=\ubk_{\{0\}}$) was proved in~\cite[\S 7.6]{ahr}. The proof of the general case is similar, with the additional complication of nontrivial local systems.

Let $j:Y_{(L,\cO)}\hookrightarrow\fg$ be the inclusion. Since $j^*:\End(\IC(Y_{(L,\cO)},(\varpi_{(L,\cO)})_*\widetilde{\cE}))\to\End((\varpi_{(L,\cO)})_*\widetilde{\cE} [\dim Y_{(L,\cO)}])$ is an isomorphism, to prove the theorem it suffices to prove that the following isomorphism of shifted local systems on $Y_{(L,\cO)}$, induced by~\eqref{eqn:uind}, is $(N_M(L,\cO)/L)$-equivariant:
\begin{equation} \label{eqn:juind}
j^*\uInd_{M\subset Q}^G\bigl(\IC\bigl(Y_{(L,\cO)}^M,(\varpi_{(L,\cO)}^M)_*\widetilde{\cE}^M\bigr)\bigr)\cong(\varpi_{(L,\cO)})_*\widetilde{\cE}[\dim Y_{(L,\cO)}].
\end{equation}
Our aim now is to give another construction of the isomorphism~\eqref{eqn:juind}, one which can be seen to be $(N_M(L,\cO)/L)$-equivariant using Theorem~\ref{thm:compatible-equivariance}.

Using~\cite[Lemma 2.14]{genspring1}, the left-hand side of~\eqref{eqn:juind} can be rewritten as
\begin{equation} \label{eqn:form1}
j^*(\pi_{M\subset Q})_!\mathsf{Ind}_Q^G\bigl(\IC\bigl(Y_{(L,\cO)}^M+\fu_Q,(\varpi_{(L,\cO)}^M)_*\widetilde{\cE}^M\boxtimes\ubk_{\fu_Q}\bigr)\bigr)[\dim \fu_Q],
\end{equation}
where $\pi_{M\subset Q}:G\times^Q\mathfrak{q}\to\fg$ is the usual proper map and $\mathsf{Ind}_Q^G:D_Q^b(\mathfrak{q},\bk)\simto D_G^b(G\times^Q\mathfrak{q},\bk)$ is the standard equivalence of equivariant derived categories. Now the cartesianness of the bottom square in~\eqref{eqn:lusztig-diag}, together with Lemma~\ref{lem:generalized-letellier}, says that 
\begin{equation*}
(X_{(L,\cO)}^M+\fu_Q)\cap Y_{(L,\cO)}=Y_{(L,\cO)}^{M;G}+\fu_Q.
\end{equation*}
So the support of~\eqref{eqn:form1} is contained in $G \cdot \bigl(Y_{(L,\cO)}^{M;G}+\fu_Q \bigr)$, and this object can be rewritten as 
\begin{equation} \label{eqn:form2}
(\pi_{M\subset Q}')_!\mathsf{Ind}_Q^G\bigl((\varpi_{(L,\cO)}^{M;G})_*\widetilde{\cE}^{M;G}\boxtimes\ubk_{\fu_Q}\bigr)[\dim Y_{(L,\cO)}],
\end{equation}
where $\pi_{M\subset Q}':G\times^Q(Y_{(L,\cO)}^{M;G}+\fu_Q)\to Y_{(L,\cO)}$ is the restriction of $\pi_{M\subset Q}$. (In calculating the shift, we have used the formula $\dim Y_{(L,\cO)}=\dim G-\dim L+\dim\cO$ and its analogue for $Y_{(L,\cO)}^M$.) Under the isomorphism of Lemma~\ref{lem:generalized-letellier}, the map $\pi_{M\subset Q}'$ corresponds to the \'etale map $\varphi$, and thus~\eqref{eqn:form2} can be rewritten as
\begin{equation} \label{eqn:form3}
\varphi_*\mathsf{Ind}_M^G\bigl((\varpi_{(L,\cO)}^{M;G})_*\widetilde{\cE}^{M;G}\bigr)[\dim Y_{(L,\cO)}],
\end{equation} 
which in turn is isomorphic to
\begin{equation} \label{eqn:form4}
\varphi_*\psi_*\mathsf{Ind}_M^G(\widetilde{\cE}^{M;G})[\dim Y_{(L,\cO)}]\cong(\varpi_{(L,\cO)})_*\widetilde{\cE}[\dim Y_{(L,\cO)}],
\end{equation}
as required. 

The verification that the isomorphism~\eqref{eqn:juind} equals the isomorphism obtained by the preceding argument (namely the composition~\eqref{eqn:form1}$\cong$\eqref{eqn:form2}$\cong$\eqref{eqn:form3}$\cong$\eqref{eqn:form4}) will be omitted: it is routine diagram-chasing, no harder than the $L=T$ case proved in~\cite[Lemmas 7.8 and 7.9]{ahr}. The $(N_M(L,\cO)/L)$-equivariance of~\eqref{eqn:juind} now boils down to the $(N_M(L,\cO)/L)$-equivariance of~\eqref{eqn:form4}, which follows immediately from Theorem~\ref{thm:compatible-equivariance}. 
\end{proof}

Recall from Theorem~\ref{thm:cocycle} and its proof in \S\ref{ss:simple-quotients} that the indecomposable direct summands of $(\varpi_{(L,\cO)})_* \widetilde{\cE}$ are of the form $\overline{\cE}\otimes \cL_{P_V}$ where $P_V$ is the projective cover of an irreducible $\bk[N_G(L,\cO)/L]$-module $V$. Since $\IC$ is fully faithful and additive, it follows that the indecomposable direct summands of $\IC(Y_{(L,\cO)},(\varpi_{(L,\cO)})_*\widetilde{\cE})$ are of the form $\IC(Y_{(L,\cO)},\overline{\cE}\otimes \cL_{P_V})$. Then~\eqref{eqn:uind} has the following refinement:

\begin{thm} \label{thm:uind-indec}
For any irreducible $\bk[N_M(L,\cO)/L]$-module $U$, we have 
\[
\begin{split}
\uInd_{M\subset Q}^G\bigl(\IC(Y_{(L,\cO)}^M,\overline{\cE}^M\otimes \cL_{P_U}^M)\bigr)&\cong
\bigoplus_{V\in\Irr([\bk[N_G(L,\cO)/L])}
\IC(Y_{(L,\cO)},\overline{\cE}\otimes \cL_{P_V})^{\oplus m_{V,U}},
\end{split}
\]
where $m_{V,U}$ denotes the multiplicity of $P_V$ as a summand of the induced representation $\mathrm{Ind}_{N_M(L,\cO)/L}^{N_G(L,\cO)/L}(P_U)$.
\end{thm}

\begin{proof}
Since the left-hand side is a direct summand of the left-hand side of~\eqref{eqn:uind}, we know it has the form stated on the right-hand side with some multiplicities. We can determine these multiplicities by applying the functor $j^*$ of restriction to $Y_{(L,\cO)}$. By the same argument as in the proof of Theorem~\ref{thm:uind},
\begin{equation*}
j^*\uInd_{M\subset Q}^G\bigl(\IC(Y_{(L,\cO)}^M,\overline{\cE}^M\otimes \cL_{P_U}^M)\bigr)
\cong
\varphi_*\mathsf{Ind}_M^G(\overline{\cE}^{M;G}\otimes \cL_{P_U}^{M;G})[\dim Y_{(L,\cO)}],
\end{equation*}
where $\cL_{P_U}^{M;G}$ is the local system on $Y_{(L,\cO)}^{M;G}$ corresponding to the representation $P_U$ via the Galois covering $\varpi_{(L,\cO)}^{M;G}$. 
Applying Theorem~\ref{thm:compatible-equivariance}, this becomes
\begin{equation*}
\varphi_*(\varphi^*\overline{\cE}\otimes \breve{\cL}_{P_U})[\dim Y_{(L,\cO)}],
\end{equation*}
where $\breve{\cL}_{P_U}$ is the local system on $\breve{Y}_{(L,\cO)}$ corresponding to the representation $P_U$ via the Galois covering $\psi$. Applying the projection formula, this in turn becomes $\overline{\cE}\otimes\cL_I[\dim Y_{(L,\cO)}]$, where $I$ denotes the induced representation $\mathrm{Ind}_{N_M(L,\cO)/L}^{N_G(L,\cO)/L}(P_U)$. The result follows.
\end{proof}

\subsection{A restriction theorem}

For the remainder of this section, let $(\cO_L,\cE_L)$ denote a cuspidal pair in $\fN_{L,\bk}^\cusp$ satisfying the conditions introduced at the outset:
\begin{enumerate}
\setcounter{enumi}{2}
\item $\cE_L$ is absolutely irreducible;
\item the isomorphism class of $\cE_L$ is fixed by $N_G(L)$.
\end{enumerate}
Recall from Proposition~\ref{prop:distinguished} and Lemma~\ref{lem:dist-norm} that $N_G(L,\cO_L)=N_G(L)$ for cuspidal pairs, and from Lemma~\ref{lem:cgu} that conditions $(1)$ and $(2)$ are automatically satisfied. 

As mentioned in \S\ref{ss:series}, there is another pair $(\cO_L',\cE_L')\in\fN_{L,\bk}^\cusp$ (frequently, and possibly always, equal to $(\cO_L,\cE_L)$), such that 
\begin{equation*}
\bT_{\fl}(\IC(\cO_L,\cE_L))\cong\IC(\cO_L'+\fz_L,\cE_L'\boxtimes\ubk_{\fz_L}).
\end{equation*} 
Since Fourier transform is compatible with field extensions and with the adjoint action of $G$, conditions $(3)$ and $(4)$ hold for $(\cO_L',\cE_L')$ also. Hence we can apply the results of this section and the previous one with $(\cO,\cE)=(\cO_L',\cE_L')$.

Combining Lemma~\ref{lem:fourier} and Theorem~\ref{thm:cocycle}, we have a canonical bijection
\begin{equation} \label{eqn:bijection-g}
\Irr(\bk[N_G(L)/L])\longleftrightarrow\fN_{G,\bk}^{(L,\cO_L,\cE_L)},
\end{equation}
in which $V\in\Irr(\bk[N_G(L)/L])$ corresponds to the unique pair $(\cO_V,\cE_V)\in\fN_{G,\bk}$ such that
\begin{equation*} \label{eqn:fourier-definition-of-bijection}
\bT_{\fg}(\IC(\cO_V,\cE_V))\cong\IC(Y_{(L,\cO_L')},\overline{\cE_L'}\otimes\cL_{V}).
\end{equation*}
We now show that the bijection~\eqref{eqn:bijection-g} between simple objects is implemented by a functor between abelian categories. 

Recall the isomorphism~\eqref{eqn:fourier}. Since $\bT_{\fg}$ is an equivalence, the $(N_G(L)/L)$-action on $\IC(Y_{(L,\cO_L')},(\varpi_{(L,\cO_L')})_*\widetilde{\cE_L'})$ induces an $(N_G(L)/L)$-action on $\Ind_{L\subset P}^G(\IC(\cO_L,\cE_L))$. Hence we have a functor
\[
\bS_{G}^{(L,\cO_L,\cE_L)}=\Hom(\Ind_{L\subset P}^G(\IC(\cO_L,\cE_L)),-):\Perv_G(\cN_G,\bk)\to\Rep(N_G(L)/L,\bk).
\]
In the special case where $L=T$ is a maximal torus (and hence $(\cO_L,\cE_L)=(\{0\},\ubk)$), this is the Springer functor $\bS_G:\Perv_G(\cN_G,\bk)\to\Rep(N_G(T)/T,\bk)$, as defined in~\cite[Section 5]{ahjr}. So $\bS_{G}^{(L,\cO_L,\cE_L)}$ is a `generalized Springer functor'. 

\begin{prop}
Let $(\cO,\cE)\in\fN_{G,\bk}$.
\begin{enumerate}
\item If $(\cO,\cE)\in\fN_{G,\bk}^{(L,\cO_L,\cE_L)}$, then $\bS_{G}^{(L,\cO_L,\cE_L)}(\IC(\cO,\cE))$ is the irreducible representation of $N_G(L)/L$ corresponding to $(\cO,\cE)$ under~\eqref{eqn:bijection-g}.  
\item If $(\cO,\cE)\in\fN_{G,\bk}^{(L_1,\cO_{L_1},\cE_{L_1})}$ for some other triple $(L_1,\cO_{L_1},\cE_{L_1})$ such that $Y_{(L,\cO_L')}\not\subset X_{(L_1,\cO_{L_1}')}$, then $\bS_{G}^{(L,\cO_L,\cE_L)}(\IC(\cO,\cE))=0$.
\end{enumerate}
\end{prop}

\begin{proof}
By definition, we have an isomorphism
\begin{equation} \label{eqn:functor-value}
\bS_{G}^{(L,\cO_L,\cE_L)}(\IC(\cO,\cE))\cong\Hom\bigl(\IC(Y_{(L,\cO_L')},(\varpi_{(L,\cO_L')})_*\widetilde{\cE_L'}),\bT_{\fg}(\IC(\cO,\cE))\bigr).
\end{equation}
In case (1), if $(\cO,\cE)=(\cO_V,\cE_V)$ for $V\in\Irr(\bk[N_G(L)/L])$, then the right-hand side of~\eqref{eqn:functor-value} is
\[
\begin{split}
\Hom\bigl(\IC(Y_{(L,\cO_L')},&(\varpi_{(L,\cO_L')})_*\widetilde{\cE_L'}),\IC(Y_{(L,\cO_L')},\overline{\cE_L'}\otimes\cL_{V})\bigr)\\
&\cong\Hom_{\Loc(Y_{(L,\cO_L')},\bk)}\bigl(\overline{\cE_L'}\otimes\cL_{\bk[N_G(L)/L]},\overline{\cE_L'}\otimes\cL_{V}\bigr)\\
&\cong\Hom_{\Rep(N_G(L)/L,\bk)}(\bk[N_G(L)/L],V)\cong V.
\end{split}
\]
(Here, the second isomorphism uses Lemma~\ref{lem:local-systems-abs-irr}.)
In case (2), the right-hand side of~\eqref{eqn:functor-value} vanishes because $\bT_{\fg}(\IC(\cO,\cE))$ is a simple perverse sheaf supported in $X_{(L_1,\cO_{L_1}')}$, which by~\cite[Proposition 6.5]{lusztig-cusp2} either does not intersect $X_{(L,\cO_L')}$ or intersects it only in the boundary of $Y_{(L,\cO_L')}$.
\end{proof}

The following result says that generalized Springer functors are compatible with restriction to Levi subgroups. In the $L=T$ case (i.e., the case of Springer functors), this was shown in~\cite[Section 7]{ahr}.

\begin{thm} \label{thm:restriction-compatibility}
We have an isomorphism of functors
\[
\mathrm{Res}_{N_M(L)/L}^{N_G(L)/L}\circ\bS_{G}^{(L,\cO_L,\cE_L)}\cong
\bS_{M}^{(L,\cO_L,\cE_L)}\circ\Res_{M\subset Q}^G.
\]
In particular, if $(\cO,\cE)\in\fN_{G,\bk}^{(L,\cO_L,\cE_L)}$ corresponds to $V\in\Irr(\bk[N_G(L)/L])$ under the bijection~\eqref{eqn:bijection-g}, then
\[
\mathrm{Res}_{N_M(L)/L}^{N_G(L)/L}(V)\cong\bS_{M}^{(L,\cO_L,\cE_L)} \bigl( \Res_{M\subset Q}^G(\IC(\cO,\cE)) \bigr).
\]
\end{thm}

\begin{proof}
Applying Fourier transform to the isomorphism~\eqref{eqn:uind}, and using~\eqref{eqn:fourier} for $G$ and for $M$, we obtain exactly the isomorphism
\begin{equation} \label{eqn:transitivity}
\Ind_{M\subset Q}^G\bigl(\Ind_{L\subset R}^M(\IC(\cO_L,\cE_L))\bigr)\cong 
\Ind_{L\subset P}^G(\IC(\cO_L,\cE_L))
\end{equation}
derived from the transitivity isomorphism $\Ind_{M\subset Q}^G\circ\Ind_{L\subset R}^M\cong\Ind_{L\subset P}^G$ (see~\cite[Lemma 2.6]{genspring1}). Theorem~\ref{thm:uind} implies that the isomorphism~\eqref{eqn:transitivity} is $(N_M(L)/L)$-equivariant, where the $(N_M(L)/L)$-action on the left-hand side is obtained from that on $\Ind_{L\subset R}^M(\IC(\cO_L,\cE_L))$ and the $(N_M(L)/L)$-action on the right-hand side is obtained by restricting the $(N_G(L)/L)$-action. So for $\cF\in\Perv_G(\cN_G,\bk)$ we have
\[
\begin{split}
\mathrm{Res}_{N_M(L)/L}^{N_G(L)/L}&(\bS_{G}^{(L,\cO_L,\cE_L)}(\cF))
\cong \Hom\bigl(\Ind_{M\subset Q}^G\bigl(\Ind_{L\subset R}^M(\IC(\cO_L,\cE_L))\bigr),\cF\bigr)\\
&\cong \Hom\bigl(\Ind_{L\subset R}^M(\IC(\cO_L,\cE_L)),\Res_{M\subset Q}^G(\cF))\\
&\cong\bS_{M}^{(L,\cO_L,\cE_L)}(\Res_{M\subset Q}^G(\cF)),
\end{split}
\]
which proves the claim. (Here the first isomorphism follows from~\eqref{eqn:transitivity}, and the second one from adjunction.)
\end{proof}

\section{Strategy of the proof of Theorem~\ref{thm:main}}
\label{sect:strategy}

Continue to let $G$ be a connected reductive group over $\C$. In this section we will introduce some hypotheses on $G$ and $\bk$ which imply the modular generalized Springer correspondence for $G$, and explain how to reduce Theorem~\ref{thm:main} to case-by-case checking.

\subsection{Central characters}

Let $L$ be a Levi subgroup of $G$, and $(\cO_L,\cE_L)\in\fN_{L,\bk}$. Recall that $\cE_L$ corresponds to an irreducible representation $V$ over $\bk$ of the finite group $A_L(x):=L_x/L_x^\circ$, where $x\in\cO_L$. The inclusion $Z(L)\subset L_x$ induces a homomorphism $Z(L)/Z(L)^\circ\to A_L(x)$ whose image is a central subgroup. If $\chi:Z(L)/Z(L)^\circ\to\bk^\times$ is a homomorphism, we say that $\cE_L$ has \emph{central character} $\chi$ if $Z(L)/Z(L)^\circ$ acts on $V$ via $\chi$. If $\cE_L$ is absolutely irreducible, it is guaranteed to have some central character by Schur's Lemma.

\begin{lem} \label{lem:central-consistency}
Let $(\cO_L,\cE_L)\in\fN_{L,\bk}^{\cusp}$, and suppose $\cE_L$ has central character $\chi$. 
For any pair $(\cO,\cE)\in\fN_{G,\bk}^{(L,\cO_L,\cE_L)}$, $\cE$ has central character $\chi\circ\varrho$ where $\varrho$ denotes the natural homomorphism $Z(G)/Z(G)^\circ\to Z(L)/Z(L)^\circ$.
\end{lem}

\begin{proof}
This follows easily from the definition of $\Ind_{L\subset P}^G(\IC(\cO_L,\cE_L))$, as in the setting of $\Qlb$-sheaves~\cite[\S 3.2]{lusztig}.
\end{proof}

We will sometimes refer to $\chi\circ\varrho$, rather than to $\chi$, as the central character of $\cE_L$. This does no harm because the homomorphism $\varrho$ is well known to be surjective. 

\begin{lem} \label{lem:fourier-consistency}
Let $(\cO_L,\cE_L)\in\fN_{L,\bk}^{\cusp}$, and define $(\cO_L',\cE_L')\in\fN_{L,\bk}^{\cusp}$ by Fourier transform as in \S{\rm \ref{ss:series}}. If $\cE_L$ has central character $\chi$, then so does $\cE_L'$.
\end{lem}

\begin{proof}
This follows easily from the definition of the Fourier transform, as in the setting of $\Qlb$-sheaves~\cite[Section 9]{lusztig-fourier}.
\end{proof}

\begin{lem}
\label{lem:central-conjugation}
Let $(\cO_L,\cE_L)\in\fN_{L,\bk}$, and assume $\cE_L$ has central character $\chi$. Then for any $n \in N_G(L, \cO_L)$, the local system $n^* \cE_L$ on $\cO_L$ has central character $\chi$.
\end{lem}

\begin{proof}
This follows from the observation that the action of $N_G(L)$ on $Z(L)/Z(L)^\circ$ is trivial, since the surjective morphism  $\varrho$ of Lemma~\ref{lem:central-consistency} is $N_G(L)$-equivariant, and $N_G(L)$ acts trivially on $Z(G)$.
\end{proof}

\subsection{The two key statements}

Consider the following statements about a connected reductive group $H$ and the field $\bk$.

\begin{stmt} \label{stmt:abs-irr}
For any $(\cO,\cE)\in\fN_{H,\bk}^{\cusp}$, the local system $\cE$ is absolutely irreducible.
\end{stmt}

\begin{stmt} \label{stmt:dist-char}
For a fixed nilpotent orbit $\cO\subset\cN_H$, if one considers all the $H$-equivariant local systems $\cE$ on $\cO$ such that $(\cO,\cE)\in\fN_{H,\bk}^{\cusp}$, then these local systems all have a central character and these central characters are distinct.
\end{stmt}

Notice that, for a fixed $H$, one can always enlarge $\bk$ to ensure that Statement~\ref{stmt:abs-irr} holds, and then the existence of the central characters referred to in Statement~\ref{stmt:dist-char} follows automatically, but not the distinctness.

\begin{rmk}
A remarkable feature of Lusztig's generalized Springer correspondence for $\Qlb$-sheaves is that the distinctness of central characters of cuspidal pairs holds even without fixing the orbit $\cO$ (see~\cite[Introduction]{lusztig}). In Section~\ref{sec:sp2n} we will see that when $\bk$ has characteristic $2$, there are several cuspidal pairs for $\Sp(2n)$ with the same (trivial) central character, supported on different distinguished orbits. Thus, Statement~\ref{stmt:dist-char} appears to be the best we can hope for in the modular case.  
\end{rmk}

Let $\mathfrak{L}$ denote a set of representatives of $G$-conjugacy classes of Levi subgroups of $G$. We can now prove the following conditional version of Theorem~\ref{thm:main} (without any assumption that $G$ is classical).

\begin{thm} \label{thm:conditional}
Suppose that Statements~{\rm \ref{stmt:abs-irr}} and~{\rm \ref{stmt:dist-char}} hold with $H=L$ for all proper Levi subgroups $L\subsetneq G$. Then  
\begin{equation} \label{eqn:disjointness-conditional}
\fN_{G,\bk} = \bigsqcup_{L \in \mathfrak{L}} \bigsqcup_{(\cO_L,\cE_L)\in\fN_{L,\bk}^\cusp}\fN_{G,\bk}^{(L,\cO_L,\cE_L)},
\end{equation}
and for any $L\in\mathfrak{L}$ and $(\cO_L,\cE_L)\in\fN_{L,\bk}^\cusp$, we have a canonical bijection
\begin{equation} \label{eqn:bijection-conditional}
\fN_{G,\bk}^{(L,\cO_L,\cE_L)}\longleftrightarrow\Irr(\bk[N_G(L)/L]).
\end{equation}
In particular, we have
\begin{equation} \label{eqn:cuspidal-count}
|\fN_{G,\bk}|=\sum_{L \in \mathfrak{L}} \sum_{(\cO_L,\cE_L)\in\fN_{L,\bk}^\cusp} |\Irr(\bk[N_G(L)/L])|.
\end{equation}
\end{thm}

\begin{proof}
It is clear from what was said in \S\ref{ss:series} that
\begin{equation*}
\fN_{G,\bk} = \bigcup_{L \in \mathfrak{L}} \bigcup_{(\cO_L,\cE_L)\in\fN_{L,\bk}^\cusp}\fN_{G,\bk}^{(L,\cO_L,\cE_L)}.
\end{equation*}
By Corollary~\ref{cor:partial-disjointness}, the union over $\mathfrak{L}$ is disjoint; so to prove~\eqref{eqn:disjointness-conditional} we need only prove the disjointness of the union over $\fN_{L,\bk}^\cusp$ for each $L\in\mathfrak{L}$. When $L=G$, the union over $\fN_{G,\bk}^\cusp$ is disjoint by definition. Suppose for a contradiction that for some $L\in\mathfrak{L}$ with $L\subsetneq G$, there are distinct cuspidal pairs $(\cO_L^{(1)},\cE_L^{(1)}),(\cO_L^{(2)},\cE_L^{(2)})\in\fN_{L,\bk}^\cusp$ such that $\fN_{G,\bk}^{(L,\cO_L^{(1)},\cE_L^{(1)})}$ and $\fN_{G,\bk}^{(L,\cO_L^{(2)},\cE_L^{(2)})}$ are not disjoint. Let $((\cO_L^{(i)})',(\cE_L^{(i)})')$ denote the cuspidal pair obtained from $(\cO_L^{(i)},\cE_L^{(i)})$ by Fourier transform as in~\eqref{eqn:l-fourier}, for $i\in\{1,2\}$. Then by Corollary~\ref{cor:partial-disjointness} again, $(L,(\cO_L^{(1)})')$ and $(L,(\cO_L^{(2)})')$ must be $G$-conjugate, or in other words $(\cO_L^{(1)})'$ and $(\cO_L^{(2)})'$ are $N_G(L)$-conjugate. However, $(\cO_L^{(1)})'$ and $(\cO_L^{(2)})'$ are distinguished $L$-orbits by Proposition~\ref{prop:distinguished}, so Lemma~\ref{lem:dist-norm} forces $(\cO_L^{(1)})'=(\cO_L^{(2)})'$. Since Fourier transform is invertible, we must have $(\cE_L^{(1)})'\neq(\cE_L^{(2)})'$. Statement~\ref{stmt:dist-char} for $H=L$ implies that $(\cE_L^{(1)})'$ and $(\cE_L^{(2)})'$ have different central characters. By Lemma~\ref{lem:fourier-consistency}, we conclude that $\cE_L^{(1)}$ and $\cE_L^{(2)}$ have different central characters. But then Lemma~\ref{lem:central-consistency} implies that $\fN_{G,\bk}^{(L,\cO_L^{(1)},\cE_L^{(1)})}$ and $\fN_{G,\bk}^{(L,\cO_L^{(2)},\cE_L^{(2)})}$ are disjoint, contradicting our assumption. So~\eqref{eqn:disjointness-conditional} is proved.

The bijection~\eqref{eqn:bijection-conditional} is trivial if $L=G$ (both sides have one element). If $L\neq G$, the canonical bijection~\eqref{eqn:bijection-conditional} is provided by the combination of Lemma~\ref{lem:fourier} and Theorem~\ref{thm:cocycle} applied with $(\cO,\cE)=(\cO_L',\cE_L')$. (Here we use Lemma~\ref{lem:dist-norm} to replace $N_G(L,\cO_L')/L$ with $N_G(L)/L$.) So we need only verify the three assumptions of Theorem~\ref{thm:cocycle}. Assumption~\eqref{eqn:distinguished-substitute} follows from Lemma~\ref{lem:cgu}; assumption~\eqref{eqn:abs-irr} is guaranteed by Statement~\ref{stmt:abs-irr} for $H=L$; and assumption~\eqref{eqn:fixed} follows from Statement~\ref{stmt:dist-char} for $H=L$, since local systems with different central characters cannot be in the same $N_G(L)$-orbit, by Lemma~\ref{lem:central-conjugation}.
\end{proof}

Theorem~\ref{thm:conditional} suggests an inductive approach that, if successful, proves the modular generalized Springer correspondence for $G$ at the same time as determining its cuspidal pairs. Assuming by induction that the cuspidal pairs for every proper Levi subgroup $L$ of $G$ have been determined, one may hope that Statements~\ref{stmt:abs-irr} and~\ref{stmt:dist-char} hold for all such $L$ (in the case of Statement~\ref{stmt:abs-irr}, this is not so much a hope as a specification of how large $\bk$ needs to be to allow this approach). If so, then Theorem~\ref{thm:conditional} applies and one has the modular generalized Springer correspondence for $G$. Moreover, from~\eqref{eqn:cuspidal-count} one can work out the number of cuspidal pairs for $G$. In combination with other information such as Proposition~\ref{prop:distinguished} and~\cite[Proposition 2.22]{genspring1}, this may be enough to determine the cuspidal pairs for $G$, completing the inductive step.    

In Sections \ref{sec:sln}--\ref{sec:son} we will see that this approach succeeds when $G$ is classical, and thus prove Theorem \ref{thm:main}. We will consider the various Lie types in turn, taking $G=\SL(n)$ in Section~\ref{sec:sln} (type $\mathbf{A}$), $G=\Sp(2n)$ in Section~\ref{sec:sp2n} (type $\mathbf{C}$), and $G=\Spin(n)$ in Section~\ref{sec:son} (types $\mathbf{B}$ and $\mathbf{D}$). 

\subsection{Some reductions}
\label{ss:reductions}

We need some general reduction principles, to explain why proving Theorem~\ref{thm:main} for the cases where $G$ is simply connected and quasi-simple is enough to prove it for general classical groups. These principles are also required in the inductive proof for each simply connected quasi-simple $G$, because of course Levi subgroups of such $G$ are not themselves quasi-simple. 

First, consider the relationship between $G$ and its maximal semisimple quotient $G/Z(G)^\circ$. The nilpotent cone of $G/Z(G)^\circ$ is the same as $\cN_G$ (on which $Z(G)^\circ$ acts trivially). The forgetful functor
\[
\Perv_{G/Z(G)^\circ}(\cN_G,\bk) \to \Perv_G(\cN_G,\bk)
\]
associated with the quotient morphism $G \twoheadrightarrow G/Z(G)^\circ$ is exact and fully faithful, since the forgetful functors from both categories to the category of all perverse sheaves on $\cN_G$ are fully faithful. Since the natural homomorphism $A_G(x)\to A_{G/Z(G)^\circ}(x)$ is an isomorphism for all $x\in\cN_G$, the sets $\fN_{G,\bk}$ and $\fN_{G/Z(G)^\circ,\bk}$ can be identified. The notion of cuspidal pair is the same whether one considers $G$ or $G/Z(G)^\circ$, so $\fN_{G,\bk}^\cusp$ and $\fN_{G/Z(G)^\circ,\bk}^\cusp$ can be identified also. The Levi subgroups of $G/Z(G)^\circ$ are the subgroups of the form $L/Z(G)^\circ$ where $L$ is a Levi subgroup of $G$. Note that $N_G(L)/L\cong N_{G/Z(G)^\circ}(L/Z(G)^\circ)/(L/Z(G)^\circ)$. The above comments apply to the relationship between $L$ and $L/Z(G)^\circ$ also, and in particular $\fN_{L,\bk}^\cusp=\fN_{L/Z(G)^\circ,\bk}^\cusp$. Finally, the induction series associated to a given element of $\fN_{L,\bk}^\cusp$ is the same for $G/Z(G)^\circ$ as for $G$. We deduce that, in proving Theorem~\ref{thm:main}, we can replace $G$ by the semisimple group $G/Z(G)^\circ$.

Now assume $G$ is semisimple, and let $\widetilde{G}$ be a simply connected cover of $G$. Then $\widetilde{G}$ is a direct product of simply connected quasi-simple groups. We claim that, in proving Theorem \ref{thm:main}, we can replace $G$ by $\widetilde{G}$; granting this, it is then clear that we can reduce to the simply connected quasi-simple case, because all the relevant concepts behave well with respect to direct products. To show the claim, we need to relate cuspidal pairs and induction series for $G$ and for $\widetilde{G}$. 

Let $K=\ker(\widetilde{G}\to G)$, a subgroup of the finite centre $Z(\widetilde{G})$. The nilpotent cone of $\widetilde{G}$ is the same as $\cN_G$ (and $K$ acts trivially on it). Again the forgetful functor
\[
\Perv_{G}(\cN_G,\bk) \to \Perv_{\widetilde{G}}(\cN_G,\bk)
\]
is exact and fully faithful. Since the natural homomorphism $A_{\widetilde{G}}(x)\to A_G(x)$ is surjective for all $x\in\cN_G$, we can identify $\fN_{G,\bk}$ with a subset of $\fN_{\widetilde{G},\bk}$, characterized by the condition that $K$ acts trivially on the local system. A pair in $\fN_{G,\bk}$ is cuspidal for $G$ if and only if it is cuspidal for $\widetilde{G}$, so $\fN_{G,\bk}^\cusp=\fN_{\widetilde{G},\bk}^\cusp\cap\fN_{G,\bk}$. The Levi subgroups of $\widetilde{G}$ are all of the form $\widetilde{L}$ where $\widetilde{L}$ denotes the inverse image of a Levi subgroup $L$ of $G$. Note that $N_{\widetilde{G}}(\widetilde{L})/\widetilde{L}\cong N_G(L)/L$. The above comments apply to the relationship between $L$ and $\widetilde{L}$ also, and in particular $\fN_{L,\bk}^\cusp=\fN_{\widetilde{L},\bk}^\cusp\cap\fN_{L,\bk}$. Clearly,
the induction series associated to a given element of $\fN_{L,\bk}^\cusp$ is the same for $\widetilde{G}$ as for $G$. Since disjointness of induction series is the only point at issue in~\eqref{eqn:disjointness} (see the proof of Theorem~\ref{thm:conditional}), knowing Theorem~\ref{thm:main} for $\widetilde{G}$ implies it for $G$.

\subsection{Combinatorial notation}

As in \cite{genspring1}, we let $\Comp$ denote the set of sequences of nonnegative integers with finitely many nonzero terms. (Here the sequences will be parametrized by positive integers).  Elements of $\Comp$ are sometimes called \emph{compositions}.  For $\sa = (\sa_1, \sa_2, \ldots) \in \Comp$, let $\|\sa\| = \sum_{i=1}^\infty \sa_i$.  Given $\sa, \sfb \in \Comp$ and $k \in \N$, we can form the sum $\sa + \sfb$ and the product $k\sa$.

For $m \in \N$, let $\Part(m)$ denote the set of \emph{partitions} of $m$.  We identify $\Part(m)$ with the subset of $\Comp$ consisting of decreasing sequences $\lambda$ with $\|\lambda\| = m$.  For $\lambda \in \Part(m)$, $\mu \in \Part(m')$ and $k \in \N$, the sum $\lambda + \mu$ and the product $k\lambda$ are defined as above, via this identification.  For $\lambda \in \Part(m)$, let $\sm(\lambda) = (\sm_1(\lambda), \sm_2(\lambda), \ldots)$ be the composition in which $\sm_i(\lambda)$ is the multiplicity of $i$ in $\lambda$. We write $\lambda^\tr$ for the transpose partition, defined by the property that $\lambda_i^\tr-\lambda_{i+1}^\tr=\sm_i(\lambda)$ for all $i$. For $\lambda \in \Part(m)$ and $\mu \in \Part(m')$, we define $\lambda\cup\mu\in\Part(m+m')$ to be the partition whose parts are the union of those of $\lambda$ and those of $\mu$; thus, $(\lambda\cup\mu)^\tr=\lambda^\tr+\mu^\tr$.

Let $\Part_\ell(m) \subset \Part(m)$ be the set of \emph{$\ell$-regular partitions}, i.e., partitions in which $\sm_i(\lambda) < \ell$ for all $i$.  On the other hand, let $\Part(m,\ell) \subset \Part(m)$ be the set of partitions all of whose parts are powers of $\ell$: that is, $\sm_i(\lambda) = 0$ unless $i = \ell^j$ for some $j \ge 0$. For $\sa \in \Comp$, we define
\[
\uPart(\sa) = \prod_{i \ge 1} \Part(\sa_i)
\qquad\text{and}\qquad
\uPart_\ell(\sa) = \prod_{i \ge 1} \Part_\ell(\sa_i).
\]
We write an element of $\uPart(\sa)$ as $\blambda=(\lambda^{(1)},\lambda^{(2)},\cdots)$ where $\lambda^{(i)}\in\Part(\sa_i)$.
Recall that $\Part_\ell(m)$ is in bijection with $\Irr(\bk[\fS_m])$, where $\bk$ has characteristic $\ell$ and $\fS_m$ denotes the symmetric group; hence $\uPart_\ell(\sa)$ is in bijection with $\Irr(\bk[\fS_{\sa}])$, where $\fS_{\sa}=\prod_{i\geq 1} \fS_{\sa_i}$.

For $m \in \N$, we let $\Bipart(m)$ denote the set of bipartitions of $m$. For $\sa \in \Comp$ we define $\uBipart(\sa)$ in the obvious way.  When $\ell \ne 2$, we also let $\Bipart_\ell(m) \subset \Bipart(m)$ denote the subset consisting of $\ell$-regular bipartitions (i.e.~pairs $(\lambda^1,\lambda^2)$ where both $\lambda^1$ and $\lambda^2$ are $\ell$-regular), and for $\sa \in \Comp$, we define $\uBipart_\ell(\sa)$ correspondingly. Recall that when $\bk$ has characteristic $\ell\neq 2$, $\Bipart_\ell(m)$ is in bijection with $\Irr(\bk[(\Z/2\Z)\wr\fS_m])$, where $\wr$ denotes the wreath product; hence $\uBipart_\ell(\sa)$ is in bijection with $\Irr(\bk[(\Z/2\Z)\wr\fS_{\sa}])$.

If $n \geq 1$ and $\ell$ is a prime number, $n_{\ell'}$ denotes $n/\ell^a$ where $\ell^a$ is the largest power of $\ell$ that divides $n$.

\section{The special linear group}
\label{sec:sln}

In this section we fix $n \geq 1$ and a prime number $\ell$, and we consider the case where $G=\SL(n)$ and $\bk$ is a field of characteristic $\ell$ containing all the $n$-th roots of unity (or equivalently all the $n_{\ell'}$-th roots of unity). The main result appears in Theorem~\ref{thm:sln-cuspidal}.

\subsection{Preliminaries}
\label{ss:sln-prelim}

We identify the centre $Z(G)$ with the group $\mu_n$ of complex $n$-th roots of unity. Let $\widehat{\mu_n}$ be the set of group homomorphisms $\chi:\mu_n\to\bk^\times$. Note that $\widehat{\mu_n}$ is a cyclic group of order $n_{\ell'}$ under pointwise multiplication. For $\chi\in\widehat{\mu_n}$, let $e(\chi)$ denote the order of $\chi$. We now explain how to use the elements of $\widehat{\mu_n}$ to parametrize the local systems of interest to us.

Recall that the $G$-orbits in $\cN_G$ are in bijection with $\Part(n)$: for $\lambda\in\Part(n)$, the corresponding orbit $\cO_\lambda$ consists of nilpotent matrices with Jordan blocks of sizes $\lambda_1,\lambda_2,\ldots$. For any $x\in\cO_{\lambda}$, the natural homomorphism $Z(G)\to A_G(x)$ is surjective with kernel $\mu_{n/\gcd(\lambda)}$ where $\gcd(\lambda)$ denotes $\gcd(\lambda_1,\lambda_2,\ldots)$. Hence the irreducible $G$-equivariant $\bk$-local systems on $\cO_\lambda$ all have rank one, and they are distinguished by their central characters, which range over those $\chi\in\widehat{\mu_n}$ such that $e(\chi)\mid\gcd(\lambda)_{\ell'}$. We will write these local systems as $\cE_{\lambda,\chi}$ accordingly. Thus
\begin{equation*} \label{eqn:sln-pairs}
\fN_{G,\bk}=\{(\cO_\lambda,\cE_{\lambda,\chi}), \, (\lambda,\chi)\in\Part(n)'\},
\end{equation*}
where
\[
\Part(n)':=\{(\lambda,\chi)\in\Part(n)\times\widehat{\mu_n}, \, e(\chi)\mid\gcd(\lambda)_{\ell'}\}.
\]
The unique distinguished orbit in $\cN_G$ is the regular orbit $\cO_{(n)}$, consisting of nilpotent matrices with a single Jordan block. The irreducible $G$-equivariant $\bk$-local systems on $\cO_{(n)}$ are the $\cE_{(n),\chi}$ where $\chi$ runs over $\widehat{\mu_n}$.

The set of $G$-conjugacy classes of Levi subgroups of $G$ is also in bijection with $\Part(n)$: for $\nu=(\nu_1,\nu_2,\cdots,\nu_{s})\in\Part(n)$ (where $s=\ell(\nu)$), one can set
\begin{equation*}\label{eqn:sln-levi}
L_\nu = \mathrm{S}(\GL(\nu_1)\times\GL(\nu_2)\times\cdots\times\GL(\nu_{s})),
\end{equation*}
and choose $\mathfrak{L}:=\{L_\nu, \, \nu \in \Part(n)\}$.
The relative Weyl group $N_G(L_\nu)/L_\nu$ is isomorphic to $\fS_{\sm(\nu)}$, so $\Irr(\bk[N_G(L_\nu)/L_\nu])$ is in bijection with $\uPart_\ell(\sm(\nu))$.

Note that we have isomorphisms
\begin{equation} \label{eqn:isogeny}
\begin{split}
Z(L_\nu)&\cong \{(z_1,\cdots,z_{s})\in(\C^\times)^{s}\,|\,\prod_{i=1}^{s}z_i^{\nu_i}=1\},\\
Z(L_\nu)^\circ&\cong \{(z_1,\cdots,z_{s})\in(\C^\times)^{s}\,|\,\prod_{i=1}^{s}z_i^{\nu_i/\gcd(\nu)}=1\},\\
L_\nu/Z(L_\nu)^\circ&\cong \frac{\SL(\nu_1)\times\cdots\times\SL(\nu_{s})}{\{(\zeta_1,\cdots,\zeta_{s})\in\mu_{\nu_1}\times\cdots\times\mu_{\nu_{s}}\,|\,\prod_{i=1}^{s}\zeta_i^{\nu_i/\gcd(\nu)}=1\}}.
\end{split}
\end{equation}
The natural surjective homomorphism $Z(G)\to Z(L_\nu)/Z(L_\nu)^\circ$ has kernel $\mu_{n/\gcd(\nu)}$. Hence the group homomorphisms $Z(L_\nu)/Z(L_\nu)^\circ\to\bk^\times$ are in bijection with those $\chi\in\widehat{\mu_n}$ such that $e(\chi)\mid\gcd(\nu)_{\ell'}$.

We can identify the nilpotent cone $\cN_{L_\nu}$ with the product $\cN_{\SL(\nu_1)}\times\cdots\times\cN_{\SL(\nu_s)}$. We let $\cO^{L_\nu}_{[\nu]}$ denote the regular $L_\nu$-orbit in $\cN_{L_\nu}$, i.e.\ $\cO^{L_\nu}_{[\nu]}=\cO_{(\nu_1)}\times\cdots\times\cO_{(\nu_{s})}$. For $x\in\cO^{L_\nu}_{[\nu]}$, the natural homomorphism $Z(L_\nu)/Z(L_\nu)^\circ\to A_{L_\nu}(x)$ is an isomorphism. Hence the irreducible $L_\nu$-equivariant $\bk$-local systems on $\cO^{L_\nu}_{[\nu]}$ all have rank one, and they are distinguished by their central characters; we will write them as $\cE^{L_\nu}_\chi$ where $\chi\in\widehat{\mu_n}$ is such that $e(\chi)\mid\gcd(\nu)_{\ell'}$. We can use the third isomorphism in~\eqref{eqn:isogeny} to regard $\cE^{L_\nu}_\chi$ as an $(\SL(\nu_1)\times\cdots\times\SL(\nu_{s}))$-equivariant local system; we then have
\begin{equation} \label{eqn:isogeny2}
\cE^{L_\nu}_\chi=\cE^{\SL(\nu_1)}_{(\nu_1),\chi_1}\boxtimes\cdots\boxtimes\cE^{\SL(\nu_s)}_{(\nu_s),\chi_s},
\end{equation}
where $\chi_i\in\widehat{\mu_{\nu_i}}$ is defined uniquely by the rule that $\chi_i(\zeta_i)=\chi(\zeta)$ whenever $\zeta\in\mu_n$ and $\zeta_i\in\mu_{\nu_i}$ satisfy $\zeta^{n/\gcd(\nu)}=\zeta_i^{\nu_i/\gcd(\nu)}$. Here, the notation $\cE^{\SL(\nu_i)}_{(\nu_i),\chi_i}$ is the analogue for $\SL(\nu_i)$ of the notation $\cE_{(n),\chi}$ for $G=\SL(n)$, i.e., it denotes the $\SL(\nu_i)$-equivariant local system on $\cO_{(\nu_i)}$ associated to $\chi_i\in\widehat{\mu_{\nu_i}}$.

For the purposes of modular reduction arguments, we let $\K$ be the extension of $\Ql$ obtained by adjoining all $n$-th roots of unity, $\O$ be its ring of integers, and $\F$ be the residue field of $\O$. Then $\F$ is isomorphic to the extension of $\Fl$ obtained by adjoining all $n$-th roots of unity, and $\bk$ is an extension of $\F$.

\subsection{The two statements and the classification of cuspidal pairs}

By Proposition~\ref{prop:distinguished}, every cuspidal pair for $G$ must be supported on the regular orbit $\cO_{(n)}$. 

\begin{lem}
\label{lem:sln-cuspidal}
If $\chi\in\widehat{\mu_n}$ satisfies $e(\chi)=n_{\ell'}$, then the pair $(\cO_{(n)},\cE_{(n),\chi})$ is cuspidal.
\end{lem}

\begin{proof}
This proof is similar to the proof of~\cite[Proposition 2.25]{genspring1}. In fact, it is enough to prove the proposition in case $\bk=\F$. In this case,
there exists an $\O$-free local system $\cE^\O$ on $\cO_{(n)}$ such that $\cE^\K:= \K \otimes_\O \cE^\O$ is a rank-one local system associated with a generator of the group of homomorphisms $\mu_n \to \K^\times$, and such that $\F \otimes_\O \cE^\O \cong \cE_{(n),\chi}$. Then $(\cO_{(n)},\cE^\K)$ is a cuspidal pair by \cite[(10.3.2)]{lusztig}, and $\IC(\cO_{(n)},\cE_{(n),\chi})$ occurs in the modular reduction of $\IC(\cO_{(n)}, \cE^\K)$. By~\cite[Proposition 2.22]{genspring1} this implies that $(\cO_{(n)},\cE_{(n),\chi})$ is cuspidal.
\end{proof}

\begin{rmk}
The preceding proof involved the observation that $\IC(\cO_{(n)},\cE_{(n),\chi})$ occurs in the modular reduction of $\IC(\cO_{(n)}, \cE^\K)$.  In fact, since the only distinguished orbit in $\cN_{\SL(n)}$ is $\cO_{(n)}$,
the modular reduction of $\IC(\cO_{(n)}, \cE^\K)$ is equal to $\IC(\cO_{(n)},\cE_{(n),\chi})$.
\end{rmk}

\begin{thm}
\label{thm:sln-cuspidal}
Let $\bk$ be a field containing all $n$-th roots of unity.
Then Theorem~{\rm \ref{thm:main}} and Statements~{\rm \ref{stmt:abs-irr}} and~{\rm \ref{stmt:dist-char}} hold for $G=\SL(n)$. The only cuspidal pairs are those described in Lemma~{\rm \ref{lem:sln-cuspidal}}, so the number of cuspidal pairs is $\phi(n_{\ell'})$.
\end{thm}

\begin{proof}
We prove this by induction on $n$, the $n=1$ case being trivial. 

Let $\nu\in\Part(n)$ with $\nu\neq(n)$, so that the corresponding Levi subgroup $L_\nu$ is a proper subgroup of $G$. As explained in \S\ref{ss:reductions}, the cuspidal pairs for $L_\nu$ can be identified with those for $L_\nu/Z(L_\nu)^\circ$, which can in turn be identified with a subset of the cuspidal pairs for the simply-connected cover $\SL(\nu_1)\times\cdots\times\SL(\nu_{s})$; see~\eqref{eqn:isogeny}.

After possibly replacing $\bk$ by a larger field $\bk'$ containing additional roots of unity, the inductive hypothesis applies to each factor $\SL(\nu_i)$, and tells us the classification of cuspidal pairs for $\SL(\nu_1)\times\cdots\times\SL(\nu_{s})$ over $\bk'$: they have the form
\[
(\cO_{(\nu_1)}\times\cdots\times\cO_{(\nu_{s})},\cE^{\SL(\nu_1)}_{(\nu_1),\chi_1}\boxtimes\cdots\boxtimes\cE^{\SL(\nu_s)}_{(\nu_s),\chi_s}),
\]
where each $\chi_i\in\widehat{\mu_{\nu_i}}$ satisfies $e(\chi_i)=(\nu_i)_{\ell'}$.  However, the discussion preceding~\eqref{eqn:isogeny2} shows that among these, the $L_\nu$-equivariant local systems are already defined over $\bk$.  Thus every cuspidal pair for $L_\nu$ (over $\bk$) is supported on the regular orbit $\cO^{L_\nu}_{[\nu]}$; moreover, applying~\eqref{eqn:isogeny2}, we see that $(\cO^{L_\nu}_{[\nu]},\cE^{L_\nu}_\chi)$ is cuspidal if and only if $e(\chi)=(\nu_i)_{\ell'}$ for every $i$.

We conclude that $L_\nu$ has cuspidal pairs if and only if $\nu$ has the form $d\rho$ where $d\mid n_{\ell'}$ and $\rho\in\Part(n/d,\ell)$, and in this case the cuspidal pairs are
$(\cO^{L_{d\rho}}_{[d\rho]},\cE^{L_{d\rho}}_\chi)$ where $\chi\in\widehat{\mu_n}$ satisfies $e(\chi)=d$. In particular, the number of cuspidal pairs for $L_{d\rho}$ is $\phi(d)$, and they are distinguished by their central characters.

We have established Statements~\ref{stmt:abs-irr} and~\ref{stmt:dist-char} for all proper Levi subgroups of $G$, so we can invoke Theorem~\ref{thm:conditional} and conclude that Theorem~\ref{thm:main} holds for $G$.

For the cuspidal pairs described in Lemma~\ref{lem:sln-cuspidal}, it is clear that Statements~\ref{stmt:abs-irr} and~\ref{stmt:dist-char} hold.  Thus, to complete the inductive step, it remains to show that those $\phi(n_{\ell'})$ pairs are the only cuspidal pairs for $G = L_{(n)}$. It suffices to show that this number of cuspidal pairs makes the equality~\eqref{eqn:cuspidal-count} hold, which follows immediately from Lemma~\ref{lem:sln-comb} below (using the obvious bijection between $\uPart_\ell(\sm(d\rho))$ and $\uPart_\ell(\sm(\rho))$).
\end{proof}

Recall the combinatorial bijection that was used in~\cite{genspring1} to describe the modular generalized Springer correspondence for $\GL(n)$:
\begin{lem}[{\cite[Lemma 3.9]{genspring1}}] \label{lem:gln-comb} 
The following map is a bijection:
\[
\Psico = \bigsqcup_{\nu \in \Part(n,\ell)} \psico_\nu: \bigsqcup_{\nu \in \Part(n,\ell)} \uPart_\ell(\sm(\nu)) \to \Part(n),
\]
where
\[
\psico_\nu: \uPart_\ell(\sm(\nu)) \to \Part(n):\blambda\mapsto \sum_{i \ge 0} \ell^i(\lambda^{(\ell^i)})^\tr.
\]
\end{lem}

The analogous fact needed to complete the proof of Theorem~\ref{thm:sln-cuspidal} is:

\begin{lem} \label{lem:sln-comb}
The following map is a bijection:
\[
\Xico = \bigsqcup_{\substack{\chi\in\widehat{\mu_n}\\\rho\in \Part(n/e(\chi),\ell)}} \xico_{\chi,\rho}: \bigsqcup_{\substack{\chi\in\widehat{\mu_n}\\\rho\in \Part(n/e(\chi),\ell)}} \uPart_\ell(\sm(\rho)) \to \Part(n)',
\]
where
\[
\xico_{\chi,\rho}: \uPart_\ell(\sm(\rho)) \to \Part(n)':\blambda\mapsto (\sum_{i \ge 0} e(\chi)\ell^i(\lambda^{(\ell^i)})^\tr,\chi).
\]
\end{lem}

\begin{proof}
For fixed $\chi\in\widehat{\mu_n}$, the partitions $\sigma\in\Part(n)$ for which $(\sigma,\chi)\in\Part(n)'$ are exactly those of the form $e(\chi)\tau$ for $\tau\in\Part(n/e(\chi))$. So the result follows from Lemma~\ref{lem:gln-comb} applied with $n/e(\chi)$ in place of $n$.
\end{proof}

Note the following consequence of Theorem~\ref{thm:sln-cuspidal} and its proof:

\begin{cor} \label{cor:sln-fourier-invariance}
For $\SL(n)$,
cuspidal perverse sheaves are invariant under Fourier transform. In other words, for any Levi $L \subset \SL(n)$ and any $(\cO_L, \cE_L) \in \fN_{L,\bk}^{\cusp}$, we have $(\cO_L', \cE_L')=(\cO_L,\cE_L)$.
\end{cor}

\begin{proof}
This follows from Lemma~\ref{lem:fourier-consistency}, since the cuspidal pairs for $\SL(n)$, classified in Theorem~\ref{thm:sln-cuspidal}, have distinct central characters.
\end{proof}

\begin{rmk} \label{rmk:pgln}
Consider the group $\mathrm{PGL}(n)=\SL(n)/\mu_n$. As explained in \S\ref{ss:reductions}, the cuspidal pairs for $\mathrm{PGL}(n)$ can be identified with the cuspidal pairs for $\SL(n)$ that have trivial central character. By Theorem~\ref{thm:sln-cuspidal}, there is a unique such cuspidal pair (namely $(\cO_{(n)},\ubk)$) if $n$ is a power of $\ell$, and none otherwise; thus we recover the classification of cuspidal pairs for $\GL(n)$ given in~\cite[Theorem 3.1]{genspring1}.
\end{rmk}

\section{The symplectic group}
\label{sec:sp2n}

In this section we fix $n \geq 1$ and take $G=\Sp(2n)$. Since all the groups $A_G(x)$ for $x\in\cN_G$ are $2$-groups, the behaviour is markedly different in the $\ell=2$ and $\ell\neq 2$ cases.

Recall that the $G$-orbits in $\cN_G$ are classified by Jordan form: they are in bijection with the set
\[
\Part_{\Sp}(2n)=\{\lambda\in\Part(2n)\,|\,\sm_{2i+1}(\lambda)\text{ is even for all }i\}.
\]
By~\cite[Theorem 8.2.14]{cm}, the distinguished orbits are the orbits $\cO_\lambda$ where $\lambda$ belongs to the set $\Part_{2,\Sp}(2n)$ of partitions of $2n$ into distinct even parts.

\subsection{The $\ell=2$ case} Since the only irreducible representation of a $2$-group in characteristic $2$ is the trivial representation, the only simple $G$-equivariant local system on a nilpotent orbit in the $\ell=2$ case is the constant sheaf.

\begin{thm} \label{thm:sp2n2-cuspidal}
Let $\bk$ be any field of characteristic~$2$. Then Theorem~{\rm \ref{thm:main}} and
Statements~{\rm \ref{stmt:abs-irr}} and~{\rm \ref{stmt:dist-char}} hold for $G=\Sp(2n)$. Every pair $(\cO_\lambda,\ubk)$ for $\lambda\in\Part_{2,\Sp}(2n)$ is cuspidal, so the number of cuspidal pairs is $|\Part_{2,\Sp}(2n)|=|\Part_2(n)|$.
\end{thm}

\begin{proof}
We prove this by induction on $n$, the $n=1$ case being part of Theorem~\ref{thm:sln-cuspidal}. The set of $G$-conjugacy classes of Levi subgroups of $G$ is in bijection with $\bigsqcup_{0\leq m\leq n}\Part(m)$: for $\nu=(\nu_1,\nu_2,\cdots,\nu_{s})\in\Part(m)$ (where $s=\ell(\nu)$), a corresponding Levi subgroup has the form
\[ L_\nu\cong \GL(\nu_1)\times\GL(\nu_2)\times\cdots\times\GL(\nu_{s})\times\Sp(2(n-m)), \]
where we omit the last factor if $m=n$.  The only irreducible $L_\nu$-equivariant local system on any orbit in $\cN_{L_\nu}$ is the trivial one (as we noted above in the case $L_\nu = G$), so it is clear that Statements~\ref{stmt:abs-irr} and~\ref{stmt:dist-char} hold for all Levi subgroups of $G$, including $G$ itself.  By Theorem~\ref{thm:conditional}, Theorem~\ref{thm:main} holds for $G$.

It remains to classify the cuspidal pairs for $G$.  For a proper Levi subgroup $L_\nu \subset G$, the inductive hypothesis tells us the classification of cuspidal pairs for the $\Sp(2(n-m))$ factor, and~\cite[Theorem 3.1]{genspring1} tells us the classification of cuspidal pairs for the $\GL(\nu_i)$ factors. We conclude that $L_\nu$ has cuspidal pairs if and only if $\nu\in\Part(m,2)$, and in that case the cuspidal pairs have the form $(\cO_{(\nu_1)}\times\cdots\times\cO_{(\nu_s)}\times\cO_{\mu},\ubk)$ where $\mu$ runs over $\Part_{2,\Sp}(2(n-m))$.

We have $N_G(L_\nu)/L_\nu \cong (\Z/2\Z) \wr \fS_{\sm(\nu)}$.  Note that $|\Irr(\bk[(\Z/2\Z)\wr\fS_{\sm(\nu)}])|=|\Irr(\bk[\fS_{\sm(\nu)}])|=|\uPart_2(\sm(\nu))|$.  We wish to show that every pair $(\cO_\lambda,\ubk)$ for $\lambda\in\Part_{2,\Sp}(2n)$ is cuspidal.  By Proposition~\ref{prop:distinguished}, \eqref{eqn:cuspidal-count} and the preceding observations, it suffices to show that
\begin{equation} \label{eqn:sp2n2-count}
|\Part_{\Sp}(2n)|=\sum_{0\leq m\leq n}|\Part_{2,\Sp}(2(n-m))|\sum_{\nu\in\Part(m,2)}|\uPart_2(\sm(\nu))|.
\end{equation}
By Lemma~\ref{lem:gln-comb}, 
the right-hand side of~\eqref{eqn:sp2n2-count} is equal to
$\sum_{0\leq m\leq n}|\Part_{2,\Sp}(2(n-m))|\times|\Part(m)|$. There is an obvious bijection
\begin{equation} \label{eqn:sp2n2-bijection}
f:\bigsqcup_{0\leq m\leq n}\Part_{2,\Sp}(2(n-m))\times\Part(m)\simto \Part_{\Sp}(2n)
\end{equation}
defined by $f(\mu,\lambda')=\mu\cup\lambda'\cup\lambda'$, so~\eqref{eqn:sp2n2-count} holds.
\end{proof}

\subsection{The $\ell\neq 2$ case} 
\label{ss:sp-not2}
Here we do have nontrivial local systems on our nilpotent orbits: for $x\in\cO_\lambda$ with $\lambda\in\Part_{\Sp}(2n)$, the group $A_G(x)$ is isomorphic to $(\Z/2\Z)^{|\{i\,|\,\sm_{2i}(\lambda)\neq 0\}|}$, so the number of isomorphism classes of simple $G$-equivariant local systems on $\cO_\lambda$ is $2^{|\{i\,|\,\sm_{2i}(\lambda)\neq 0\}|}$, and all these local systems have rank one.

If $n=\binom{k+1}{2}$ for some positive integer $k$, then by~\cite[Corollary 12.4(b)]{lusztig} there is a unique rank-one $G$-equivariant $\Ql$-local system $\cD_k^{\Ql}$ on the orbit $\cO_{(2k,2(k-1),\cdots,4,2)}$ such that $(\cO_{(2k,2(k-1),\cdots,4,2)},\cD_k^{\Ql})$ is a cuspidal pair in Lusztig's sense. This local system has an obvious $\Zl$-form $\cD_k^{\Zl}$, and we set $\cE_k:=\bk \otimes_{\Zl} \cD_k^{\Zl}$. Then $(\cO_{(2k,2(k-1),\cdots,4,2)},\cE_k)$ is a cuspidal pair in our sense by \cite[Proposition~2.22]{genspring1}.

\begin{thm} \label{thm:sp2n-cuspidal}
Let $\bk$ be a field of characteristic different from $2$. Then
Theorem~{\rm \ref{thm:main}} and Statements~{\rm \ref{stmt:abs-irr}} and~{\rm \ref{stmt:dist-char}} hold for $G=\Sp(2n)$. If $n$ is not of the form $\binom{k+1}{2}$, there is no cuspidal pair; if $n=\binom{k+1}{2}$, the unique cuspidal pair is $(\cO_{(2k,2(k-1),\cdots,4,2)},\cE_k)$.
\end{thm}

\begin{proof}
Again we prove this by induction on $n$, the $n=1$ case being part of Theorem~\ref{thm:sln-cuspidal}. Recall the description of Levi subgroups from the proof of Theorem~\ref{thm:sp2n2-cuspidal}. By the inductive hypothesis, we know that for $\nu\in\Part(m)$, $1\leq m\leq n$, the corresponding Levi $L_\nu$ has a cuspidal pair if and only if $\nu\in\Part(m,\ell)$ and $n-m=\binom{k+1}{2}$ for some positive integer $k$, and in that case the unique cuspidal pair is $(\cO_{(\nu_1)}\times\cdots\times\cO_{(\nu_s)}\times\cO_{(2k,\cdots,4,2)},\ubk\boxtimes\cdots\boxtimes\ubk\boxtimes\cE_k)$. It is clear that Statement~\ref{stmt:abs-irr} holds for all Levi subgroups (including $G$), and that Statement~\ref{stmt:dist-char} holds at least for proper Levi subgroups.  By Theorem~\ref{thm:conditional}, Theorem~\ref{thm:main} holds for $G$.

Statement~\ref{stmt:dist-char} for $G$ itself will be immediate once we show that the number of cuspidal pairs for $G$ is $1$ if $n=\binom{k+1}{2}$ and $0$ otherwise. By~\eqref{eqn:cuspidal-count} and the fact that $|\Irr(\bk[(\Z/2\Z)\wr\fS_{\sm(\nu)}])|=|\uBipart_\ell(\sm(\nu))|$ it suffices to show that
\begin{equation} \label{eqn:sp2n-count}
\sum_{\lambda\in\Part_{\Sp}(2n)}2^{|\{i\,|\,\sm_{2i}(\lambda)\neq 0\}|}
=\sum_{\substack{m\in\N,k\in\Z_{>0}\\m+\binom{k+1}{2}=n}}\sum_{\nu\in\Part(m,\ell)}|\uBipart_\ell(\sm(\nu))|.
\end{equation}
But a trivial modification of the bijection of Lemma~\ref{lem:gln-comb} shows that 
\begin{equation}\label{eqn:bipart-ladic}
\sum_{\nu\in\Part(m,\ell)}|\uBipart_\ell(\sm(\nu))|=|\Bipart(m)|,
\end{equation}
so~\eqref{eqn:sp2n-count} reduces to the identity~\cite[(10.4.1)]{lusztig} that Lusztig used to classify cuspidal pairs for $\Sp(2n)$ in the characteristic-$0$ case.
\end{proof}

\begin{rmk} \label{rmk:psp2n}
Consider the group $\mathrm{PSp}(2n)=\Sp(2n)/\{\pm I\}$. The cuspidal pairs for $\mathrm{PSp}(2n)$ can be identified with the cuspidal pairs for $\Sp(2n)$ on which the nontrivial central element $-I$ acts trivially. If $\ell=2$, all the cuspidal pairs described in Theorem~\ref{thm:sp2n2-cuspidal} have this property. If $\ell\neq 2$, then by the construction of $\cE_k$, we have the same rule as in the characteristic-$0$ case (see \cite[Introduction]{lusztig}): the cuspidal pair $(\cO_{(2k,2(k-1),\cdots,4,2)},\cE_k)$ of Theorem~\ref{thm:sp2n-cuspidal} descends to $\mathrm{PSp}(2n)$ if and only if $n$ is even, i.e.\ $k\equiv 0$ (mod $4$) or $k\equiv 3$ (mod $4$).  
\end{rmk}

\section{The special orthogonal and spin groups}
\label{sec:son}

In this section we fix $N\geq 3$, and set $G=\Spin(N)$ and $\overline{G}=\SO(N)$. As usual, we consider the cases $N=2n+1$ (type $\mathbf{B}_n$) and $N=2n$ (type $\mathbf{D}_n$) separately. As in the symplectic group case, all the groups $A_G(x)$ for $x\in\cN_G$ are $2$-groups, so we also have a natural dichotomy according to whether $\ell=2$ or $\ell\neq 2$.

\subsection{The $N=2n+1$, $\ell=2$ case}

Recall that the $G$-orbits in $\cN_G$, which are the same as the $\overline{G}$-orbits, are classified by Jordan form: they are in bijection with
\[
\Part_{\SO}(2n+1)=\{\lambda\in\Part(2n+1)\,|\,\sm_{2i}(\lambda)\text{ is even for all }i\}.
\]
By~\cite[Theorem 8.2.14]{cm}, the distinguished orbits are the orbits $\cO_\lambda$ where $\lambda$ belongs to the set $\Part_{2,\SO}(2n+1)$ of partitions of $2n+1$ into distinct odd parts. 

In the $\ell=2$ case, there are no non-constant simple $G$-equivariant local systems on nilpotent orbits. Hence there is essentially no difference between the story for $G$ and that for $\overline{G}$.

\begin{thm} \label{thm:spinodd2-cuspidal}
Let $\bk$ be any field of characteristic $2$. Then
Theorem~{\rm \ref{thm:main}} and Statements~{\rm \ref{stmt:abs-irr}} and~{\rm \ref{stmt:dist-char}} hold for $G=\Spin(2n+1)$. Every pair $(\cO_\lambda,\ubk)$ for $\lambda\in\Part_{2,\SO}(2n+1)$ is cuspidal, so the number of cuspidal pairs is $|\Part_{2,\SO}(2n+1)|$.
\end{thm}

\begin{proof}
We prove this by induction on $n$, the $n=1$ case being part of Theorem~\ref{thm:sln-cuspidal}. It suffices to consider $\overline{G}$, for which the proof proceeds much like that of Theorem~\ref{thm:sp2n2-cuspidal}. The set of $\overline{G}$-conjugacy classes of Levi subgroups of $\overline{G}$ is again in bijection with $\bigsqcup_{0\leq m\leq n}\Part(m)$: for $\nu=(\nu_1,\nu_2,\cdots,\nu_{s})\in\Part(m)$, a corresponding Levi subgroup has the form
\[ \overline{L}_\nu\cong \GL(\nu_1)\times\GL(\nu_2)\times\cdots\times\GL(\nu_{s})\times\SO(2(n-m)+1), \]
where we omit the last factor if $m=n$. As in the setting of Theorem~\ref{thm:sp2n2-cuspidal}, the only irreducible $\overline{L}_\nu$-equivariant local system on any orbit in $\cN_{L_\nu}$ is the trivial one.  Once again, it follows that Statements~\ref{stmt:abs-irr} and~\ref{stmt:dist-char} hold for all Levi subgroups of $\overline{G}$ (including $\overline{G}$ itself), and that Theorem~\ref{thm:main} holds for $\overline{G}$.

To finish the proof, it remains to classify the cuspidal pairs for $\overline{G}$. The inductive hypothesis and~\cite[Theorem~3.1]{genspring1} tell us that a proper Levi subgroup $\overline{L}_\nu \subset \overline{G}$ has cuspidal pairs if and only if $\nu\in\Part(m,2)$, and that in that case the cuspidal pairs have the form $(\cO_{(\nu_1)}\times\cdots\times\cO_{(\nu_s)}\times\cO_{\mu},\ubk)$ where $\mu$ runs over $\Part_{2,\SO}(2(n-m)+1)$.  To show that every pair $(\cO_\lambda,\ubk)$ for $\lambda\in\Part_{2,\SO}(2n+1)$ is cuspidal, we must show that
\begin{equation} \label{eqn:spinodd2-count}
|\Part_{\SO}(2n+1)|=\sum_{0\leq m\leq n}|\Part_{2,\SO}(2(n-m)+1)|\sum_{\nu\in\Part(m,2)}|\uPart_2(\sm(\nu))|.
\end{equation}
But this can be proved in exactly the same way as~\eqref{eqn:sp2n2-count}.
\end{proof}

\subsection{The $N=2n$, $\ell=2$ case}

Here the classification of $G$-orbits ($=\overline{G}$-orbits) in $\cN_G$ is slightly different from the $N=2n+1$ case: for every partition $\lambda\in\Part_{\SO}(2n)$
we have a single $\mathrm{O}(2n)$-orbit $\cO_\lambda$, which on restriction to $\overline{G}$ either remains a single $G$-orbit or splits into two, the latter occurring precisely when $\lambda$ belongs to the set $\Part_{\ve}(2n)$ of partitions $\lambda\in\Part_{\SO}(2n)$ satisfying $\sm_{2i+1}(\lambda)=0$ for all $i$. (Of course, $\Part_{\ve}(2n)$ is empty if $n$ is odd.) By~\cite[Theorem 8.2.14]{cm}, the distinguished orbits are the orbits $\cO_\lambda$ where $\lambda$ belongs to the set $\Part_{2,\SO}(2n)$ of partitions of $2n$ into distinct odd parts. (Recall that we are assuming that $N=2n\geq 4$.)

\begin{thm} \label{thm:spineven2-cuspidal}
Let $\bk$ be any field of characteristic $2$. Then
Theorem~{\rm \ref{thm:main}} and Statements~{\rm \ref{stmt:abs-irr}} and~{\rm \ref{stmt:dist-char}} hold for $G=\Spin(2n)$. Every pair $(\cO_\lambda,\ubk)$ for $\lambda\in\Part_{2,\SO}(2n)$ is cuspidal, so the number of cuspidal pairs is $|\Part_{2,\SO}(2n)|$.
\end{thm}

\begin{proof}
The proof is much the same as that of Theorem~\ref{thm:spinodd2-cuspidal}. The base case of the induction is now the $n=2$ case, which follows from Theorem~\ref{thm:sln-cuspidal}. If $n$ is odd, the set of $\overline{G}$-conjugacy classes of Levi subgroups of $\overline{G}$ is in bijection not with $\bigsqcup_{0\leq m\leq n}\Part(m)$ but rather with $\bigsqcup_{\substack{0\leq m\leq n\\m\neq n-1}}\Part(m)$. (The exclusion of the $m=n-1$ case is because a Levi subgroup of the form $\GL(\nu_1)\times\cdots\times\GL(\nu_s)\times\SO(2)$ is conjugate to one of the form $\GL(\nu_1)\times\cdots\times\GL(\nu_s)\times\GL(1)$). So the equality required in place of~\eqref{eqn:spinodd2-count} is
\begin{equation} \label{eqn:spineven2-count}
|\Part_{\SO}(2n)|=\sum_{\substack{0\leq m\leq n\\m\neq n-1}}|\Part_{2,\SO}(2(n-m))|\sum_{\nu\in\Part(m,2)}|\uPart_2(\sm(\nu))|,
\end{equation}
which can also be proved in the same way as~\eqref{eqn:sp2n2-count}. (In fact, the exclusion of the $m=n-1$ case makes no difference to~\eqref{eqn:spineven2-count}, since $\Part_{2,\SO}(2)$ is empty.) Now suppose that $n$ is even. As seen above, $|\fN_{G,\bk}|$ is not $|\Part_{\SO}(2n)|$ but rather $|\Part_{\SO}(2n)|+|\Part_{\ve}(2n)|$. Correspondingly, the set of $\overline{G}$-conjugacy classes of Levi subgroups of $\overline{G}$ is not quite in bijection with $\bigsqcup_{\substack{0\leq m\leq n\\m\neq n-1}}\Part(m)$: if $m=n$ and all parts of $\nu\in\Part(n)$ are even, there are two $\overline{G}$-conjugacy classes of Levi subgroups of the form $\GL(\nu_1)\times\cdots\times\GL(\nu_{s})$. The equality we need to prove, therefore, is the sum of~\eqref{eqn:spineven2-count} and
\begin{equation}\label{eqn:spineven2-extra}
|\Part_{\ve}(2n)|=\sum_{\substack{\nu\in\Part(n,2)\\\sm_1(\nu)=0}}|\uPart_2(\sm(\nu))|.
\end{equation}
Note that the left-hand side of~\eqref{eqn:spineven2-extra} is the same as the number of partitions of $n$ into even parts; under the bijection of Lemma~\ref{lem:gln-comb} (for $\ell=2$), these correspond exactly to the terms of the disjoint union labelled by $\nu\in\Part(n,2)$ where $\sm_1(\nu)=0$, so~\eqref{eqn:spineven2-extra} is proved.
\end{proof}

\subsection{Special orthogonal groups with $\ell \ne 2$}

Now, as a preliminary step, we treat the case of  $\overline{G} = \SO(N)$ when $\ell \neq 2$.
The situation in this case is parallel to that of \S\ref{ss:sp-not2}.
For $x \in \cO_\lambda$ with $\lambda \in \Part_{\SO}(N)$, we have $A_{\overline{G}}(x) \cong (\Z/2\Z)^{a(\lambda)}$ where
\[
a(\lambda) =
\begin{cases}
0 & \text{if $N$ is even and $\lambda \in \Part_{\ve}(N)$,} \\
|\{i \mid \sm_{2i-1}(\lambda) \ne 0 \}|-1 & \text{otherwise.}
\end{cases}
\]
Thus, the number of isomorphism classes of simple $\overline{G}$-equivariant local systems on $\cO_\lambda$ is $2^{a(\lambda)}$, and all these local systems have rank one.

If $N=k^2$ for some positive integer $k$, then by~\cite[Corollary 13.4(b)]{lusztig} there is a unique rank-one $\overline{G}$-equivariant $\Ql$-local system $\cD_k^{\Ql}$ on the orbit $\cO_{(2k-1,2k-3,\cdots,3,1)}$ such that $(\cO_{(2k-1,2k-3,\cdots,3,1)},\cD_k^{\Ql})$ is a cuspidal pair in Lusztig's sense. This local system has an obvious $\Zl$-form $\cD_k^{\Zl}$, and we set $\cE_k:=\bk \otimes_{\Zl} \cD_k^{\Zl}$. Then $(\cO_{(2k-1,2k-3,\cdots,3,1)},\cE_k)$ is a cuspidal pair in our sense by \cite[Proposition~2.22]{genspring1}.

\begin{thm} \label{thm:sonot2-cuspidal}
Let $\bk$ be a field of characteristic different from $2$. Then
Theorem~{\rm \ref{thm:main}} and Statements~{\rm \ref{stmt:abs-irr}} and~{\rm \ref{stmt:dist-char}} hold for $\overline{G}=\SO(N)$. If $N$ is not of the form $k^2$, there is no cuspidal pair; if $N = k^2$, the unique cuspidal pair is $(\cO_{(2k-1,2k-3,\cdots,3,1)},\cE_k)$.
\end{thm}
\begin{proof}
Since $\SO(3)$ is a quotient of $\SL(2)$, and $\SO(4)$ a quotient of $\SL(2) \times \SL(2)$, the theorem above in these two cases is implied by Theorem~\ref{thm:sln-cuspidal}, using the reductions explained in~\S\ref{ss:reductions}.

For $N \ge 5$, we proceed by induction.  Recall the description of Levi subgroups from the proofs of Theorems~\ref{thm:spinodd2-cuspidal} and~\ref{thm:spineven2-cuspidal}.  From the description of $A_{\overline{G}}(x)$ above, it is clear that Statement~\ref{stmt:abs-irr} holds for all Levi subgroups of $\overline{G}$, including $\overline{G}$ itself.  By the inductive hypothesis, we know that for $\nu \in \Part(m)$, $1 \le m \le \lfloor N/2\rfloor$ (and $m\neq N/2-1$ when $N$ is even), the corresponding Levi $\overline{L}_\nu$ has a cuspidal pair if and only if $\nu \in \Part(m,\ell)$ and $N - 2m = k^2$ for some positive integer $k$, and in that case the unique cuspidal pair is $(\cO_{(\nu_1)} \times \cdots \times \cO_{(\nu_s)} \times \cO_{(2k-1,\cdots,3,1)}, \ubk \boxtimes \cdots \boxtimes \ubk \boxtimes \cE_k)$. Thus, Statement~\ref{stmt:dist-char} holds for proper Levi subgroups of $\overline{G}$.  By Theorem~\ref{thm:conditional}, Theorem~\ref{thm:main} holds for $\overline{G}$.

Statement~\ref{stmt:dist-char} for $\overline{G}$ will be immediate once we show that the number of cuspidal pairs for $\overline{G}$ is $1$ if $N = k^2$ and $0$ otherwise.  In order to treat the even and odd cases simultaneously, let us adopt the convention that $\Part_{\ve}(N) = \emptyset$ if $N$ is odd.  Then 
\begin{equation}\label{eqn:sonot2-fn-count}
|\fN_{\overline{G},\bk}| = \sum_{\lambda\in\Part_{\SO}(N)}2^{a(\lambda)} + |\Part_{\ve}(N)|.
\end{equation}
Next, we must count the representations of the various $N_{\overline{G}}(\overline{L}_\nu)/\overline{L}_\nu$.  We have
\begin{equation}\label{eqn:son-rel-weyl}
N_{\overline{G}}(\overline{L}_\nu)/\overline{L}_\nu =
\begin{cases}
(\Z/2\Z)\wr\fS_{\sm(\nu)} & \text{if $2m < N$, or} \\
& \text{ \quad if $2m = N$ and $2\mid \gcd(\nu)$,} \\
\text{an index-$2$ sub-} \\
\text{ \quad group of $(\Z/2\Z)\wr\fS_{\sm(\nu)}$} & \text{if $2m = N$ and $2 \nmid \gcd(\nu)$.} 
\end{cases}
\end{equation}
In the first case, $N_{\overline{G}}(\overline{L}_\nu)/\overline{L}_\nu$ is a product of Coxeter groups of type $\mathbf{B}$, and its irreducible representations are parametrized by $\uBipart_\ell(\sm(\nu))$, just as in the proof of Theorem~\ref{thm:sp2n-cuspidal}.

When $2m = N$, on the other hand, the situation is analogous to the relationship between irreducible representations of Coxeter groups of types $\mathbf{B}$ and $\mathbf{D}$, via Clifford theory.  Let $\sigma$ denote the  action of $\Z/2\Z$ on $\Bipart(k)$ which exchanges the two partitions making up a bipartition.  Then $\sigma$ induces in an obvious way actions on $\Bipart_\ell(k)$ and on $\uBipart_\ell(\sm(\nu))$.  Let $\blambda \in \uBipart_\ell(\sm(\nu))$.  If $\sigma(\blambda) \ne \blambda$, then the corresponding irreducible representations $D^{\blambda}$ and $D^{\sigma(\blambda)}$ become isomorphic when restricted to $N_{\overline{G}}(\overline{L}_\nu)/\overline{L}_\nu$.  But if $\blambda = \sigma(\blambda)$, then the restriction of $D^{\blambda}$ to $N_{\overline{G}}(\overline{L}_\nu)/\overline{L}_\nu$ breaks up as the sum of two nonisomorphic irreducible representations.  Thus,
\[
|\Irr(\bk[N_{\overline{G}}(\overline{L}_\nu)/\overline{L}_\nu])| = 
\frac{|\uBipart_\ell(\sm(\nu))| - |\uBipart_\ell(\sm(\nu))^\sigma|}{2} + 2|\uBipart_\ell(\sm(\nu))^\sigma|,
\]
where $\uBipart_\ell(\sm(\nu))^\sigma = \{ \blambda \in \uBipart_\ell(\sm(\nu)) \mid \sigma(\blambda) = \blambda \}$.  The set $\uBipart_\ell(\sm(\nu))^\sigma$ is empty unless all components of $\sm(\nu)$ are even.  In that case, it makes sense to form the composition $\frac{1}{2} \sm(\nu)$, and there is an obvious bijection
\[
\uBipart_\ell(\sm(\nu))^\sigma \cong \uPart_\ell(\textstyle\frac{1}{2} \sm(\nu)).
\]
By interpreting $\uPart_\ell(\frac{1}{2}\sm(\nu))$ as the empty set when $\frac{1}{2} \sm(\nu)$ is not defined, we obtain the following formula, valid whenever $2\|\nu\| = m$:
\[
|\Irr(\bk[N_{\overline{G}}(\overline{L}_\nu)/\overline{L}_\nu])| = 
\textstyle\frac{1}{2}|\uBipart_\ell(\sm(\nu))| + \frac{3}{2}|\uPart_\ell(\textstyle\frac{1}{2}\sm(\nu))|.
\]

We are now ready to count the total number of irreducible representations of all $N_{\overline{G}}(\overline{L}_\nu)/\overline{L}_\nu$, as $\overline{L}_\nu$ ranges over Levi subgroups admitting a cuspidal pair. (In this computation, when $N=k^2$ we also count the cuspidal pair for the Levi $\overline{G}$ constructed before the statement of the theorem.) Note that for such groups, we have $\nu \in \Part(m,\ell)$ for some $m$.  Since $\ell \ne 2$, $\gcd(\nu)$ will never be divisible by $2$.  Thus, for our purposes, the cases in~\eqref{eqn:son-rel-weyl} are distinguished simply by whether $2m < N$ or $2m = N$.  In the following computation, all sums involving the condition $2m = N$ should be regarded as $0$ if $N$ is odd.
\begin{align*}
&\sum_{\substack{m,k \in \N \\2m+k^2=N\\ \nu\in\Part(m,\ell)}} \!\!\!\!|\Irr(\bk[N_{\overline{G}}(\overline{L}_\nu)/\overline{L}_\nu])| \\
&=
\!\!\!\!\sum_{\substack{m\in\N,k\in\Z_{>0}\\2m+k^2=N\\ \nu\in\Part(m,\ell)}} \!\!\!\!|\Irr(\bk[N_{\overline{G}}(\overline{L}_\nu)/\overline{L}_\nu])|
+
\!\!\!\!\sum_{\substack{m\in\N\\2m=N\\ \nu\in\Part(m,\ell)}} \!\!\!\!|\Irr(\bk[N_{\overline{G}}(\overline{L}_\nu)/\overline{L}_\nu])| \\
&=
\!\!\!\!\sum_{\substack{m\in\N,k\in\Z_{>0}\\2m+k^2=N\\ \nu\in\Part(m,\ell)}} \!\!\!\!|\uBipart_\ell(\sm(\nu))|
+
\!\!\!\!\!\!\!\!\sum_{\substack{m\in\N\\2m=N\\ \nu\in\Part(m,\ell)}} \!\!\!\!\left(\frac{1}{2}|\uBipart_\ell(\sm(\nu))| + \frac{3}{2}|\uPart_\ell(\textstyle\frac{1}{2}\sm(\nu))|\right)  \\
&=
\!\!\!\!\sum_{\substack{m\in\N,k\in\Z_{>0}\\2m+k^2=N\\ \nu\in\Part(m,\ell)}} \!\!\!\!|\uBipart_\ell(\sm(\nu))|
+ 
\!\!\!\!\!\!\!\!\sum_{\substack{m\in\N\\2m=N\\ \nu\in\Part(m,\ell)}\hss} \!\!\!\!\frac{1}{2}|\uBipart_\ell(\sm(\nu))| 
+
\!\!\!\!\!\!\!\!\sum_{\substack{m\in\N\\2m=N\\ \nu\in\Part(m,\ell)\\ \text{all $\sm_i(\nu)$ even}}} \!\!\!\!\frac{3}{2}|\uPart_\ell(\textstyle\frac{1}{2}\sm(\nu))|  \\
&=
\!\!\!\!\sum_{\substack{m\in\N,k\in\Z_{>0}\\2m+k^2=N}} \!\!\!\!|\Bipart(m)|
+ 
\!\!\sum_{\substack{m\in\N\\2m=N}} \!\!\frac{1}{2}|\Bipart(m)| 
+
\!\!\sum_{\substack{m\in\N\\2m=N}} \!\!\frac{3}{2}|\Part(m/2)| 
\end{align*}
The last step in this computation is justified by~\eqref{eqn:bipart-ladic} and by the following identity, which is an easy consequence of Lemma~\ref{lem:gln-comb}:
\[
\sum_{\substack{\nu \in \Part(m,\ell)\\ \text{all $\sm_i(\nu)$ even}}} |\uPart_\ell(\textstyle\frac{1}{2}\sm(\nu))| = |\Part(m/2)|.
\]
By~\eqref{eqn:cuspidal-count}, to complete the proof of the theorem, we must show that the quantity above is equal to the right-hand side of~\eqref{eqn:sonot2-fn-count}.  That equality is none other than the identity~\cite[(10.6.3)]{lusztig} used by Lusztig to classify characteristic-$0$ cuspidal pairs.
\end{proof}

\subsection{Spin groups with $\ell \ne 2$}  

We now turn our attention to $G = \Spin(N)$ when $\ell \neq 2$.  Let $\varepsilon$ denote the nontrivial element of the kernel of the map $G \to \overline{G}$.  The groups $A_G(x)$ are no longer necessarily just products of copies of $\Z/2\Z$, although they are still $2$-groups.  An explicit description of these groups can be found in~\cite[\S14.3]{lusztig} (see also the remarks following~\cite[Corollary~6.1.7]{cm}).  As explained further below, it follows from this description that if $\bk$ contains all fourth roots of unity of its algebraic closure, then every irreducible representation of $A_G(x)$ over $\bk$ is absolutely irreducible.  Therefore, we assume in this subsection that $\bk$ contains the fourth roots of unity. 
Let $\K$ be the extension of $\Ql$ obtained by adjoining the fourth roots of unity, $\O$ be its ring of integers, and $\F$ be the residue field of $\O$. Then $\bk$ is an extension of $\F$.

Given a pair $(\cO,\cE) \in \fN_{G,\bk}$, one can consider the character $\chi: Z(G) \to \bk^\times$ by which $Z(G)$ acts on $\cE$. If this character descends to a character of the quotient $Z(\overline{G})$---that is, if $\chi(\varepsilon) = 1$---then the pair $(\cO,\cE)$ is actually $\overline{G}$-equivariant, and its cuspidality has been studied in Theorem~\ref{thm:sonot2-cuspidal}.  Thus, it now suffices to study the remaining characters.

Fix a character $\chi: Z(G) \to \bk^\times$ such that $\chi(\varepsilon) = -1$. There is one such character if $N$ is odd, and two if $N$ is even. 
Let $\tilde{\chi}:Z(G) \to \O^\times\subset\K^\times$ be the unique lift of $\chi$. 
Let $\fN_{G,\bk,\chi} \subset \fN_{G,\bk}$ be the set of pairs $(\cO,\cE)$ such that $Z(G)$ acts on $\cE$ by $\chi$, and let $\fN_{G,\K,\tilde{\chi}}$ be the corresponding set of characteristic-$0$ pairs.

By~\cite[Proposition~14.4]{lusztig}, $\fN_{G,\K,\tilde{\chi}}$ is in bijection with the set
\[
\Part_{\Spin,\varepsilon}(N) = \{ \lambda \in \Part_{\SO}(N) \mid \text{$\sm_{2i+1}(\lambda) \le 1$ for all $i \ge 0$} \}.
\]
Explicitly, the bijection works as follows. If $\lambda \in \Part_{\Spin,\varepsilon}(N)\setminus\Part_{\ve}(N)$, then the $G$-orbit $\cO_\lambda$ supports a unique irreducible $G$-equivariant $\K$-local system $\cD_{\lambda,\tilde{\chi}}^{\K}$ of central character $\tilde{\chi}$. Lusztig explains in~\cite[\S 14.3]{lusztig} that the quotient $\K[A_G(x_\lambda)] / (1 + \varepsilon)$ (for any chosen $x_\lambda \in\cO_\lambda$) is isomorphic to the even part of a Clifford algebra, and that $\cD_{\lambda,\tilde{\chi}}^{\K}$ corresponds to the unique simple module for this algebra on which $Z(G)$ acts by $\tilde{\chi}$.
(Lusztig works over $\Qlb$, but the quadratic form defining the Clifford algebra is defined and split over $\O$, so the results from~\cite[\S 9, no.~4]{bourbaki:algebre} that he cites apply equally well over $\K$ as over $\Qlb$.) 
The same arguments show that $\O[A_G(x_\lambda)] / (1 + \varepsilon)$ and $\bk[A_G(x_\lambda)] / (1 + \varepsilon)$ are also isomorphic to the even part of a similarly defined Clifford algebra, and the proofs in~\cite[\S 9, no.~4]{bourbaki:algebre} show that the relevant simple module over $\K$ is defined over $\O$, and that its modular reduction to $\bk$ is the unique simple module for $\bk[A_G(x_\lambda)] / (1 + \varepsilon)$ on which $Z(G)$ acts by $\chi$.
Hence $\cD_{\lambda,\tilde{\chi}}^{\K}$ has a natural $\O$-form $\cD_{\lambda,\tilde{\chi}}^{\O}$, and the modular reduction $\cE_{\lambda,\chi} := \bk \otimes_{\O} \cD_{\lambda,\tilde{\chi}}^{\O}$ is 
the unique irreducible $G$-equivariant $\bk$-local system on $\cO_\lambda$ of central character $\chi$, and is absolutely irreducible.

If $\lambda\in\Part_{\ve}(N)$ (forcing $N\equiv 0 \pmod 4$), then (again by~\cite[\S 14.3]{lusztig}) one of the two $G$-orbits in $\cO_\lambda$ supports a rank-one $\K$-local system $\cD_{\lambda,\tilde{\chi}}^{\K}$ of central character $\tilde{\chi}$, and this is the unique irreducible $G$-equivariant $\K$-local system of central character $\tilde{\chi}$ on either of the two orbits. In this case $|A_G(x)|=2$, so it is clear that $\cD_{\lambda,\tilde{\chi}}^{\K}$ has a natural $\O$-form $\cD_{\lambda,\tilde{\chi}}^{\O}$, and that the modular reduction $\cE_{\lambda,\chi} := \bk \otimes_{\O} \cD_{\lambda,\tilde{\chi}}^{\O}$ is the unique irreducible $G$-equivariant $\bk$-local system of central character $\chi$ on either of the two orbits in $\cO_\lambda$. This local system is of rank one, hence absolutely irreducible.

To summarize, every element of $\fN_{G,\bk,\chi}$ is of the form $(\cO_\lambda,\cE_{\lambda,\chi})$ for some $\lambda \in \Part_{\Spin,\varepsilon}(N)$ (where, if $\lambda\in\Part_{\ve}(N)$, $\cO_\lambda$ should be replaced  by the appropriate one of the two $G$-orbits it contains), and hence $\fN_{G,\bk,\chi}$ is in bijection with $\Part_{\Spin,\varepsilon}(N)$.  In particular, we have
\begin{equation}\label{eqn:spinnot2-fn-count}
|\fN_{G,\bk,\chi}| = |\Part_{\Spin,\varepsilon}(N)|.
\end{equation}

If $N = \binom{k+1}{2}$ with $k$ odd, then $(\cO_{(2k-1, 2k-5, \cdots,5,1)}, \cE_{(2k-1, 2k-5, \cdots,5,1),\chi})$ is a cuspidal pair, by~\cite[Corollary~4.9]{lus-spalt} and~\cite[Proposition~2.22]{genspring1}.  Similarly, if $N = \binom{k+1}{2}$ with $k$ even, $(\cO_{(2k-1,2k-5, \cdots, 7,3)}, \cE_{(2k-1,2k-5,\cdots,7,3),\chi})$ is a cuspidal pair.

\begin{thm} \label{thm:spinnot2-cuspidal}
Let $\bk$ be a field of characteristic different from $2$ and containing all fourth roots of unity.
Then Theorem~{\rm \ref{thm:main}} and 
Statements~{\rm \ref{stmt:abs-irr}} and~{\rm \ref{stmt:dist-char}} hold for $G=\Spin(N)$. If $N$ is not of the form $\binom{k+1}{2}$, there is no cuspidal pair in $\fN_{G,\bk,\chi}$.  If $N=\binom{k+1}{2}$, the unique cuspidal pair in $\fN_{G,\bk,\chi}$ is
\[
\begin{cases}
(\cO_{(2k-1,2k-5,\cdots,5,1)},\cE_{(2k-1,2k-5,\cdots,5,1),\chi}) & \text{if $k$ is odd,} \\
(\cO_{(2k-1,2k-5,\cdots,7,3)},\cE_{(2k-1,2k-5,\cdots,7,3),\chi}) & \text{if $k$ is even.}
\end{cases}
\]
\end{thm}

\begin{proof}
The groups $\Spin(3) \cong \SL(2)$ and $\Spin(4) \cong \SL(2) \times \SL(2)$ fall under Theorem~\ref{thm:sln-cuspidal}.  For $N \ge 5$, we proceed by induction.

We begin with a review of the Levi subgroups of $G$.  Suppose $0 \le m \le \lfloor N/2\rfloor$, and $m\neq N/2-1$ if $N$ is even.  Let $\nu \in \Part(m)$.  As in the preceding subsections, we denote by $\overline{L}_\nu$ a Levi subgroup of $\overline{G} = \SO(N)$ that is isomorphic to $GL(\nu_1) \times \cdots \times \GL(\nu_s) \times \SO(N - 2m)$.  Consider the group
\[
M_\nu = \{(z, g_1,\cdots,g_s) \in \C^\times \times \GL(\nu_1) \times \cdots \GL(\nu_s) \mid z^2 = \det(g_1)\cdots \det(g_s) \}.
\]
Note that $Z(M_\nu)/Z(M_\nu)^\circ$ has order~$2$ if $2 \mid \gcd(\nu)$, and $Z(M_\nu)$ is connected otherwise.

Let $\delta_\nu$ denote the element $(-1,1,\cdots,1) \in M_\nu$.  Let $\varepsilon_{N-2m} \in \Spin(N-2m)$ be the nontrivial element of the kernel of $\Spin(N-2m) \to \SO(N-2m)$.  Finally, let
\[
L_\nu = M_\nu \times \Spin(N - 2m)/\langle (\delta_\nu, \varepsilon_{N-2m}) \rangle.
\]
This is a Levi subgroup of $\Spin(N)$, isomorphic to the preimage of $\overline{L}_\nu$ under $\Spin(N) \to \SO(N)$.  The element $\varepsilon$ can be identified with $(\delta_\nu,1) = (1,\varepsilon_{N-2m}) \in L_\nu$.

Assume for now that $m > 0$, so that $L_\nu$ is a proper Levi subgroup of $G$. 
Consider two pairs $(\cO,\cE) \in \fN_{M_\nu,\bk}$ and $(\cO',\cE') \in \fN_{\Spin(N-2m),\bk}$.  The pair $(\cO \times \cO', \cE \boxtimes \cE')$ is $L_\nu$-equivariant if and only if the scalar by which $\delta_\nu$ acts on $\cE$ coincides with the scalar by which $\varepsilon_{N-2m}$ acts on $\cE'$.

We are interested in the subset $\fN_{L_\nu,\bk,\chi} \subset \fN_{L_\nu,\bk}$ consisting of pairs on which $Z(G)$ acts by $\chi$.  There is a natural identification
\[
\fN_{L_\nu,\bk,\chi} \leftrightarrow 
\left\{\small
\begin{array}{c}
(\cO,\cE), (\cO',\cE') \in{} \\
 \fN_{M_\nu,\bk} \times \fN_{\Spin(N-2m),\bk}
\end{array}
\;\Bigg|\;
\begin{array}{c}
\text{$\delta_\nu$ acts on $\cE$ by $-1$, and} \\
\text{$Z(\Spin(N-2m))/Z(\Spin(N-2m))^\circ$} \\
\text{acts on $\cE'$ by $\chi$}
\end{array}
\right\}.
\]

Let us now consider the conditions under which $\fN_{L_\nu,\bk,\chi}$ contains a cuspidal element.  An obvious restriction is that $2 \mid \gcd(\nu)$: otherwise, $Z(M_\nu)$ is connected, so $\delta_\nu$ cannot act nontrivially, and $\fN_{L_\nu,\bk,\chi}$ is empty.  Suppose now that $2 \mid \gcd(\nu)$.  Under the bijection above, an element of $\fN_{L_\nu,\bk,\chi}$ is cuspidal if and only if both $(\cO,\cE)$ and $(\cO',\cE')$ are cuspidal.  By induction, the latter can happen only if $N - 2m = \binom{k+1}{2}$ for some $k$, and in that case, there is a unique possibility for $(\cO',\cE')$.  On the other hand, for $(\cO,\cE)$, the reasoning is similar to that carried out in the proof of Theorem~\ref{thm:sln-cuspidal}.  If $(\cO,\cE)$ is cuspidal, then $\cO$ must be the regular nilpotent orbit for $M_\nu$.  Write $(\cO,\cE)$ as
\[
(\cO_{(\nu_1)} \times \cdots \times \cO_{(\nu_s)}, \cE_1 \boxtimes \cdots \boxtimes \cE_s),
\]
where each pair $(\cO_{(\nu_i)},\cE_i)$ lies in $\fN_{\SL(\nu_i),\bk}$.   Each such pair must be cuspidal for $\SL(\nu_i)$, so by the classification of cuspidal pairs in Theorem~\ref{thm:sln-cuspidal}, $Z(M_\nu)/Z(M_\nu)^\circ$ must act on $\cE_i$ by a character of order $(\nu_i)_{\ell'}$.  Since we have already required $\delta_\nu$ to act by $-1$, such an $\cE_i$ exists only when every $\nu_i$ is of the form $2\ell^n$.  

To summarize, $\fN_{L_\nu,\bk,\chi}$ contains a unique cuspidal element if both of the following conditions hold: $N - 2m = \binom{k+1}{2}$ for some $k$, and $\nu = 2\nu'$ for some $\nu' \in \Part(m/2,\ell)$.  Otherwise, $\fN_{L_\nu,\bk,\chi}$ contains no cuspidal element.  The cuspidal elements of $\fN_{L_\nu,\bk}$ with trivial central character were discussed in the proof of Theorem~\ref{thm:sonot2-cuspidal}.  Together, these observations show that Statement~\ref{stmt:dist-char} holds for every proper Levi subgroup of $G$.

As explained above, our assumption on $\bk$ implies that Statement~\ref{stmt:abs-irr} holds for all the $M_\nu$, and hence for all Levi subgroups of $G$, including $G$ itself.  By Theorem~\ref{thm:conditional}, Theorem~\ref{thm:main} holds for $G$.  It remains only to classify the cuspidal pairs for $G$, since that classification will imply Statement~\ref{stmt:dist-char} for $G$.

Note that $N_G(L_\nu)/L_\nu \cong N_{\overline{G}}(\overline{L}_\nu)/\overline{L}_\nu$.  From~\eqref{eqn:son-rel-weyl}, we see that this group is always a product of Coxeter groups of type $\mathbf{B}$, and that its irreducible $\bk$-representations are parametrized by $\uBipart_\ell(\sm(\nu)) \cong \uBipart_\ell(\sm(\nu'))$.  (Here, we have used the observation that the nonzero entries of $\sm(\nu)$ are the same as those of $\sm(\nu')$.)  We can now compute the total number of irreducible representations of all the $N_G(L_\nu)/L_\nu$ as $L_\nu$ ranges over Levi subgroups admitting a cuspidal pair. (In the formulas below, when $N = \binom{k+1}{2}$ we also count the cuspidal pair for the Levi $G$ constructed before the statement of the theorem.) In the following calculation, the quantity $m'$ corresponds to $m/2$ in the preceding discussion.
\begin{multline}\label{eqn:spinnot2-rep-count}
\sum_{\substack{m' \in \N, k \in \Z_{>0}\\ 4m' + \binom{k+1}{2} = N}} \sum_{\nu' \in \Part(m',\ell)} |\Irr(\bk[N_G(L_\nu)/L_\nu])| \\
=
\sum_{\substack{m' \in \N, k \in \Z_{>0}\\ 4m' + \binom{k+1}{2} = N}} \sum_{\nu' \in \Part(m',\ell)} |\uBipart_\ell(\sm(\nu'))|
=
\sum_{\substack{m' \in \N, k \in \Z_{>0}\\ 4m' + \binom{k+1}{2} = N}} |\Bipart(m')|.
\end{multline}

We wish to show that the number of cuspidal pairs in $\fN_{G,\bk,\chi}$ is $1$ if $N = \binom{k+1}{2}$ and $0$ otherwise.  By~\eqref{eqn:cuspidal-count}, it suffices to show that the quantities in~\eqref{eqn:spinnot2-fn-count} and~\eqref{eqn:spinnot2-rep-count} are equal.  That is the content of~\cite[Corollary~14.5]{lusztig}, used by Lusztig in the characteristic-$0$ version of the problem.
\end{proof}

\section{Computations in some cases}
\label{sec:det2}

With Theorem~\ref{thm:main} established for the classical groups, we may consider the question of computing the bijection~\eqref{eqn:mgsc} combinatorially.  In this section, we carry out this computation for $\SL(n)$ in arbitrary characteristic, and for $\SO(n)$ and $\Sp(2n)$ when $\ell = 2$. (Recall from Section~\ref{sec:son} that when $\ell=2$, the bijection~\eqref{eqn:mgsc} for $\Spin(n)$ is essentially the same as for $\SO(n)$.)

In these constructions, if $(\K,\O,\F)$ is an $\ell$-modular system, for $\blambda \in \uPart(\sa)$ we denote by $S^{\blambda}_{\K}$ the irreducible $\K$-representation of $\fS_{\sa}$ associated with $\blambda$, and by $S^{\blambda}_{\O}$ its standard $\O$-form. If $\blambda \in \uPart_\ell(\sa)$, we denote by $D^{\blambda}_{\F}$ (or simply $D^{\blambda}$) the irreducible $\F$-representation of $\fS_{\sa}$ associated with $\blambda$.

\subsection{The special linear group}
\label{ss:sln-det}

Let $G = \SL(n)$. The notation and conventions of Section~\ref{sec:sln} will be in force, especially those from~\S\ref{ss:sln-prelim} involving Levi subgroups of $G$.  As a consequence of Theorem~\ref{thm:sln-cuspidal} and its proof, the set $\fN_{L_\nu,\bk}^{\cusp}$ is empty unless $\nu \in \Part(n)$ has the form $d\rho$ for $d\mid n_{\ell'}$ and $\rho\in\Part(n/d,\ell)$, in which case it is in bijection with $\{\chi\in\widehat{\mu_n}\,|\,e(\chi)=d\}$. So the modular generalized Springer correspondence~\eqref{eqn:mgsc} for $G=\SL(n)$ is a bijection
\begin{equation*}\label{eqn:sln-genspring}
\bigsqcup_{\substack{\chi\in\widehat{\mu_n}\\\rho\in \Part(n/e(\chi),\ell)}}
\Irr(\bk[N_G(L_{e(\chi)\rho})/L_{e(\chi)\rho}]) \longleftrightarrow 
\fN_{G,\bk}.
\end{equation*}
Using our combinatorial parametrizations of $\Irr(\bk[N_G(L_{e(\chi)\rho})/L_{e(\chi)\rho}])$ and
$\fN_{G,\bk}$, we can reinterpret this as a bijection
\begin{equation}\label{eqn:sln-genspring2}
\Xi: \bigsqcup_{\substack{\chi\in\widehat{\mu_n}\\\rho\in \Part(n/e(\chi),\ell)}} \uPart_\ell(\sm(\rho)) \simto \Part(n)'.
\end{equation}
It remains to determine the bijection $\Xi$ explicitly, which is achieved by the following result.

\begin{thm} \label{thm:sln-det}
The modular generalized Springer correspondence $\Xi$ for $G=\SL(n)$, when interpreted as in~\eqref{eqn:sln-genspring2}, coincides with the combinatorial bijection $\Xico$ defined in Lemma~\textup{\ref{lem:sln-comb}}. 
\end{thm}

Write $\xi_{\chi,\rho}$ for the restriction of $\Xi$ to the subset of the domain indexed by $\chi$ and $\rho$. We need to show that $\xi_{\chi,\rho}=\xico_{\chi,\rho}$, where $\xico_{\chi,\rho}$ is as in  Lemma~\textup{\ref{lem:sln-comb}}. Now by definition, the image of $\xi_{\chi,\rho}$ consists of the combinatorial parameters for the pairs in the subset $\fN_{G,\bk}^{(L_{e(\chi)\rho},\cO^{L_{e(\chi)\rho}}_{[e(\chi)\rho]},\cE^{L_{e(\chi)\rho}}_\chi)}$ of $\fN_{G,\bk}$. By Lemma~\ref{lem:central-consistency}, all such pairs have as their second component a local system with central character $\chi$. So every element of the image of $\xi_{\chi,\rho}$ has the form $(e(\chi)\tau,\chi)$ for some $\tau\in\Part(n/e(\chi))$. 

Therefore, for each $\chi\in\widehat{\mu_n}$ there is some bijection
\begin{equation*}  \label{eqn:sln-genspring-chi}
\Psi_\chi=\bigsqcup_{\rho \in \Part(n/e(\chi),\ell)} \psi_{\chi,\rho}: \bigsqcup_{\rho \in \Part(n/e(\chi),\ell)}\uPart_\ell(\sm(\rho)) \simto \Part(n/e(\chi))
\end{equation*}
such that 
\begin{equation*}
\xi_{\chi,\rho}(\blambda)=(e(\chi)\psi_{\chi,\rho}(\blambda),\chi)\quad\text{ for all $\blambda\in\uPart_\ell(\sm(\rho))$.}
\end{equation*} 

Theorem~\ref{thm:sln-det} reduces to the following result:

\begin{thm} \label{thm:sln-det-chi}
For every $\chi\in\widehat{\mu_n}$ and $\rho\in\Part(n/e(\chi),\ell)$ we have $\psi_{\chi,\rho}=\psico_{\rho}$, where $\psico_{\rho}$ is defined as in Lemma~\textup{\ref{lem:gln-comb}}.
\end{thm}

The proof of Theorem~\ref{thm:sln-det-chi} is similar to that of~\cite[Theorem~3.4]{genspring1}; in fact, when $\chi$ is trivial, Theorem~\ref{thm:sln-det-chi} is essentially equivalent to~\cite[Theorem~3.4]{genspring1}, by the principles of \S\ref{ss:reductions}. However, the general case presents some additional complications.

Let $e$ denote $e(\chi)$. We proceed by induction on $n$, the base case $n=1$ being trivial. Since we know that $\Psi_\chi$ and $\bigsqcup_{\rho \in \Part(n/e,\ell)} \psico_{\rho}$ are bijections with the same (finite) domain and codomain, it suffices to prove that
\begin{equation}\label{eqn:sln-det-ineq}
\psi_{\chi,\rho}(\blambda) \le \psico_\rho(\blambda)\quad\text{ for all $\blambda\in\uPart_\ell(\sm(\rho))$,}
\end{equation}
where $\le$ denotes the usual dominance partial order on $\Part(n/e)$. (However, the induction hypothesis still has equality rather than $\le$.) 

The first step corresponds to~\cite[Lemma~3.10]{genspring1}.

\begin{lem} \label{lem:first-case}
Assume that $n=me\ell^i$ for some $m \geq 1$ and $i\geq 0$. Then~\eqref{eqn:sln-det-ineq} holds for the partition
\[
\rho = (\underbrace{\ell^i, \ell^i, \ldots, \ell^i}_{\text{$m$ entries}}).
\]
\end{lem}

\begin{proof}
Note that the composition $\sm(\rho)$ contains a single nonzero entry, equal to $m$, so that $\uPart(\sm(\rho))$ can be identified with $\Part(m)$ and $\uPart_\ell(\sm(\rho))$ with $\Part_\ell(m)$. The inequality we need to prove is that
\begin{equation}\label{eqn:sln-det-ineq-special}
\psi_{\chi,\rho}(\lambda) \le \ell^i\lambda^\tr \quad
\text{ for all $\lambda \in \Part_\ell(m)$.}
\end{equation}
Recall the $\ell$-modular system $(\K,\O,\F)$ defined in \S\ref{ss:sln-prelim}.
We can and will assume that $\bk=\F$.

As with~\cite[Lemma~3.10]{genspring1}, we make use of the fact that the Levi subgroup $L_{e\rho} = \mathrm{S}(\GL(e\ell^i) \times \cdots \times\GL(e\ell^i))$ has cuspidal pairs in characteristic $0$, in which setting the generalized Springer correspondence was determined in~\cite{lus-spalt}. 
Explicitly, let $\tilde\chi:\mu_n\to\O^\times$ denote a character of order $e\ell^i$ whose modular reduction is $\chi$; since $\gcd(e\rho)=e\ell^i$, this character $\tilde\chi$ factors through a faithful character of $Z(L_{e\rho})/Z(L_{e\rho})^\circ$. Form the corresponding $\O$-local system $\cE^{L_{e\rho}}_{\tilde\chi,\O}$ on $\cO^{L_{e\rho}}_{[e\rho]}$, and let $\cE^{L_{e\rho}}_{\tilde\chi,\K} = \K \otimes_\O \cE^{L_{e\rho}}_{\tilde\chi,\O}$.  Then $(\cO^{L_{e\rho}}_{[e\rho]}, \cE^{L_{e\rho}}_{\tilde\chi,\K}) \in \fN_{L_{e\rho},\K}$ is a characteristic-$0$ cuspidal pair by~\cite[(10.3.2)]{lusztig}. 

As part of the characteristic-$0$ generalized Springer correspondence (see~\cite[Proposition~5.2]{lus-spalt}), we have an injection
\[
\psi_{\tilde\chi,\rho,\K}:\Part(m) \to \Part(n/e)
\]
such that for any $\lambda\in\Part(m)$, the simple summand of $\Ind_{L_{e\rho} \subset P_{e\rho}}^G \IC(\cO^{L_{e\rho}}_{[e\rho]}, \cE^{L_{e\rho}}_{\tilde\chi,\K})$ corresponding to the irreducible $\K$-representation $S^{\lambda}_{\K}$ of $\fS_m\cong N_G(L_{e\rho})/L_{e\rho}$ is $\IC(\cO_{e\psi_{\tilde\chi,\rho,\K}(\lambda)},\cE_{e\psi_{\tilde\chi,\rho,\K}(\lambda),\tilde\chi,\K})$, where $\cE_{e\psi_{\tilde\chi,\rho,\K}(\lambda),\tilde\chi,\K}$ denotes the unique irreducible $G$-equivariant $\K$-local system of central character $\tilde\chi$ on the orbit $\cO_{e\psi_{\tilde\chi,\rho,\K}(\lambda)}$. (Here, $P_{e\rho}$ is any parabolic subgroup of $G$ having $L_{e\rho}$ as Levi factor.) For consistency with our definitions in the modular case, we define the generalized Springer correspondence using Fourier transform, i.e.\ via the characteristic-$0$ analogues of Lemma~\ref{lem:fourier} and Theorem~\ref{thm:cocycle} (for the latter, see~\cite[Proposition~2.20]{genspring1}). That is, $\psi_{\tilde\chi,\rho,\K}$ is specified by the rule
\begin{equation} \label{eqn:sln-fourier1}
\bT_{\fg}\bigl(\IC(\cO_{e\psi_{\tilde\chi,\rho,\K}(\lambda)},\cE_{e\psi_{\tilde\chi,\rho,\K}(\lambda),\tilde\chi,\K}) \bigr)\cong
\IC\bigl(Y_{(L_{e\rho},\cO^{L_{e\rho}}_{[e\rho]})},\cS_{\K}^\lambda\otimes\overline{\cE}^{L_{e\rho}}_{\tilde\chi,\K}\bigr),
\end{equation}
where $\cS_{\K}^\lambda$ is the $\K$-local system on $Y_{(L_{e\rho},\cO^{L_{e\rho}}_{[e\rho]})}$ corresponding to the representation $S^{\lambda}_{\K}$ of $\fS_m$ via the usual Galois covering $\varpi_{(L_{e\rho},\cO^{L_{e\rho}}_{[e\rho]})}$, and $\overline{\cE}^{L_{e\rho}}_{\tilde\chi,\K}$ is the unique irreducible summand of $(\varpi_{(L_{e\rho},\cO^{L_{e\rho}}_{[e\rho]})})_*\widetilde{\cE}^{L_{e\rho}}_{\tilde\chi,\K}$ whose $\IC$-extension has a nonzero restriction to $\mathrm{Ind}_{L_{e\rho}}^G(\cO^{L_{e\rho}}_{[e\rho]})=\cO_{(n)}$. 
Equivalently (see~\cite[Section 7]{lusztig-cusp2}), $\overline{\cE}^{L_{e\rho}}_{\tilde\chi,\K}$ corresponds to the $\fS_m$-equivariant structure on $\widetilde{\cE}^{L_{e\rho}}_{\tilde\chi,\K}$ defined by restricting an $\fS_m$-equivariant structure on the local system $\widehat{\cE}^{L_{e\rho}}_{\tilde\chi,\K}$ on $\widetilde{T}_{(L_{e\rho},\cO^{L_{e\rho}}_{[e\rho]})}$, the latter $\fS_m$-equivariant structure being normalized by the characteristic-$0$ version of Proposition~\ref{prop:equivariance-hatE}. Since $(\varpi_{(L_{e\rho},\cO^{L_{e\rho}}_{[e\rho]})})^*\overline{\cE}^{L_{e\rho}}_{\tilde\chi,\K}\cong\cE^{L_{e\rho}}_{\tilde\chi,\K}$, the local system $\overline{\cE}^{L_{e\rho}}_{\tilde\chi,\K}$ has rank one.

As shown in~\cite[\S 3.7 and Theorem~3.8(c)]{evens-mirk}, the generalized Springer correspondence defined as above using Fourier transform differs by a sign twist from the correspondence computed in~\cite[Proposition~5.2]{lus-spalt}, so $\psi_{\tilde\chi,\rho,\K}$ is given explicitly by
\begin{equation*}
\psi_{\tilde\chi,\rho,\K}(\lambda)=\ell^i\lambda^\tr.
\end{equation*} 

On the other hand, in our modular setting, the injection
\[
\psi_{\chi,\rho}:\Part_\ell(m)\to \Part(n/e)
\]
is defined via Lemma~\ref{lem:fourier} and Theorem~\ref{thm:cocycle}, so we have
\begin{equation} \label{eqn:sln-fourier2}
\bT_{\fg}\bigl(\IC(\cO_{e\psi_{\chi,\rho}(\lambda)},\cE_{e\psi_{\chi,\rho}(\lambda),\chi}) \bigr)\cong
\IC\bigl(Y_{(L_{e\rho},\cO^{L_{e\rho}}_{[e\rho]})},\cD^\lambda\otimes\overline{\cE}^{L_{e\rho}}_{\chi}\bigr),
\end{equation}
where $\cD^\lambda$ is the $\bk$-local system on $Y_{(L_{e\rho},\cO^{L_{e\rho}}_{[e\rho]})}$ corresponding to the irreducible $\bk$-representation $D^{\lambda}$ of $\fS_m$, and $\overline{\cE}^{L_{e\rho}}_{\chi}$ 
corresponds to the $\fS_m$-equivariant structure on $\widetilde{\cE}^{L_{e\rho}}_{\chi}$ defined by restricting the $\fS_m$-equivariant structure on $\widehat{\cE}^{L_{e\rho}}_{\chi}$ given by Proposition~\ref{prop:equivariance-hatE}. Here, we have used Corollary~\ref{cor:sln-fourier-invariance} to omit the primes on $\cO^{L_{e\rho}}_{[e\rho]}$ and $\cE^{L_{e\rho}}_\chi$.

For any $\lambda \in \Part_\ell(m)$, the representation $D^{\lambda}$ occurs in the modular reduction $S^{\lambda}_{\O}\otimes_{\O}\bk$ of $S^\lambda_{\K}$; likewise, $\cD^\lambda$ occurs in the modular reduction of $\cS^{\lambda}_{\K}$. Since the modular reduction of $\tilde\chi$ is $\chi$, the modular reduction of the rank-one local system $\cE^{L_{e\rho}}_{\tilde\chi,\K}$ is $\cE^{L_{e\rho}}_\chi$. Hence the local systems $\widetilde{\cE}^{L_{e\rho}}_{\tilde\chi,\K}$ and $\widehat{\cE}^{L_{e\rho}}_{\tilde\chi,\K}$ (also of rank one) have modular reductions $\widetilde{\cE}^{L_{e\rho}}_\chi$ and $\widehat{\cE}^{L_{e\rho}}_\chi$, respectively. The $\fS_m$-equivariant structure on $\widehat{\cE}^{L_{e\rho}}_{\tilde\chi,\K}$ defined as in~\cite[Lemma 7.10(b)(c)]{lusztig-cusp2} induces an $\fS_m$-equivariant structure on $\widehat{\cE}^{L_{e\rho}}_\chi$ that satisfies the condition of Proposition~\ref{prop:equivariance-hatE}, and hence coincides with the $\fS_m$-equivariant structure defined by Proposition~\ref{prop:equivariance-hatE}. It follows that the modular reduction of $\overline{\cE}^{L_{e\rho}}_{\tilde\chi,\K}$ is $\overline{\cE}^{L_{e\rho}}_{\chi}$. 

Hence $\cD^\lambda\otimes\overline{\cE}^{L_{e\rho}}_{\chi}$ occurs in the modular reduction of $\cS_{\K}^\lambda\otimes\overline{\cE}^{L_{e\rho}}_{\tilde\chi,\K}$. By the argument following~\cite[(3.14)]{genspring1}, we can conclude from this and from the equations~\eqref{eqn:sln-fourier1} and~\eqref{eqn:sln-fourier2} that $\IC(\cO_{e\psi_{\chi,\rho}(\lambda)},\cE_{e\psi_{\chi,\rho}(\lambda),\chi})$ occurs in the modular reduction of $\IC(\cO_{e\psi_{\tilde\chi,\rho,\K}(\lambda)},\cE_{e\psi_{\tilde\chi,\rho,\K}(\lambda),\tilde\chi,\K})$, and in particular is supported in $\overline{\cO_{e\ell^i\lambda^\tr}}$. So we have shown that $\cO_{e\psi_{\chi,\rho}(\lambda)}\subset\overline{\cO_{e\ell^i\lambda^\tr}}$, which gives~\eqref{eqn:sln-det-ineq-special}. 
\end{proof}

The remaining step corresponds to~\cite[Lemma~3.11]{genspring1}.

\begin{lem} \label{lem:hard-part}
If $\rho\in\Part(n/e,\ell)$ is not of the form $(\ell^i,\ell^i,\ldots,\ell^i)$, then~\eqref{eqn:sln-det-ineq} holds.
\end{lem}

\begin{proof}
Let $m_i = \sm_{\ell^i}(\rho)$, and form the Levi subgroup $M = \mathrm{S}(\GL(m_0e) \times \GL(m_1e\ell) \times \cdots \times \GL(m_ie\ell^i) \times \cdots)$.  This group contains $L_{e\rho}$; their relationship can be pictured as follows:
{\small\[
\begin{array}{@{}c@{}c@{}c@{}c@{}c@{}c@{}c@{}c@{}c@{}}
L_{e\rho} &{}= \mathrm{S}(&      \underbrace{\GL(e) \times \cdots \times \GL(e)}_{\text{$m_0$ copies}}
  &{}\times{}& \underbrace{\GL(e\ell) \times \cdots \times \GL(e\ell)}_{\text{$m_1$ copies}}
  &{}\times{}& \underbrace{\GL(e\ell^2) \times \cdots \times \GL(e\ell^2)}_{\text{$m_2$ copies}}
  &{}\times{}& \cdots) \\
\downarrow && \downarrow && \downarrow && \downarrow \\
M &{}= \mathrm{S}(&      \GL(m_0e)
  &{}\times{}& \GL(m_1e\ell)
  &{}\times{}& \GL(m_2e\ell^2)
  &{}\times{}& \cdots).
\end{array}
\]}%
Let $Q$ be the parabolic subgroup of $G$ containing $P_{e\rho}$ that has $M$ as its Levi factor. We are going to apply the results of Section~\ref{sect:transitivity}, specifically Theorem~\ref{thm:uind-indec}, to the triple $L_{e\rho}\subset M\subset G$ (the assumptions $(3)$ and $(4)$ are trivially true for our local systems). Note that $N_{M}(L_{e\rho})/L_{e\rho}=N_G(L_{e\rho})/L_{e\rho}\cong\fS_{\sm(\rho)}$. 

By~\eqref{eqn:isogeny}, $M/Z(M)^\circ$ is a central quotient of 
\[ \SL(m_0e)\times\SL(m_1e\ell)\times\cdots\times\SL(m_ie\ell^i)\times\cdots. \]
By assumption, the induction hypothesis applies to each factor of the latter product, so we know that the modular generalized Springer correspondence for $M$ is given by Theorem~\ref{thm:sln-det}. In particular, we know that for all $\blambda\in\uPart_\ell(\sm(\rho))$,
\begin{equation*}
\begin{split}
\bT_{\mathfrak{m}}\bigl(
\IC&(\cO_{e\ell^0(\lambda^{(\ell^0)})^\tr}\times\cO_{e\ell^1(\lambda^{(\ell^1)})^\tr}\times\cdots,
\cE_{e\ell^0(\lambda^{(\ell^0)})^\tr,e\ell^1(\lambda^{(\ell^1)})^\tr,\cdots;\chi})\bigr)\\
&\cong
\IC\bigl(Y_{(L_{e\rho},\cO^{L_{e\rho}}_{[e\rho]})}^M,\cD^{\blambda,M}\otimes(\overline{\cE}^{L_{e\rho}}_{\chi})^M\bigr),
\end{split}
\end{equation*}
where $\cE_{e\ell^0(\lambda^{(\ell^0)})^\tr,e\ell^1(\lambda^{(\ell^1)})^\tr,\cdots;\chi}$ denotes the unique $M$-equivariant irreducible local system on the orbit $\cO_{e\ell^0(\lambda^{(\ell^0)})^\tr}\times\cO_{e\ell^1(\lambda^{(\ell^1)})^\tr}\times\cdots$ with central character $\chi$, and $\cD^{\blambda,M}$ is the irreducible local system on $Y_{(L_{e\rho},\cO^{L_{e\rho}}_{[e\rho]})}^M$ corresponding to the irreducible representation $D^\blambda$ via the Galois covering $\varpi_{(L_{e\rho},\cO^{L_{e\rho}}_{[e\rho]})}^M$. (Here, as in the proof of Lemma~\ref{lem:first-case}, we have used Corollary~\ref{cor:sln-fourier-invariance}.)

By comparison, for $G$ we have the as-yet-uncomputed map
\[
\psi_{\chi,\rho}:\uPart_\ell(\sm(\rho))\to\Part(n/e)
\] 
defined by the rule that for all $\blambda\in\uPart_\ell(\sm(\rho))$,
\begin{equation*}
\bT_{\fg}\bigl(\IC(\cO_{e\psi_{\chi,\rho}(\blambda)},\cE_{e\psi_{\chi,\rho}(\blambda),\chi}) \bigr)\cong
\IC\bigl(Y_{(L_{e\rho},\cO^{L_{e\rho}}_{[e\rho]})},\cD^\blambda\otimes\overline{\cE}^{L_{e\rho}}_{\chi}\bigr),
\end{equation*}
where $\cD^{\blambda}$ is the irreducible local system on $Y_{(L_{e\rho},\cO^{L_{e\rho}}_{[e\rho]})}$ corresponding to the irreducible representation $D^\blambda$ via the Galois covering $\varpi_{(L_{e\rho},\cO^{L_{e\rho}}_{[e\rho]})}$. (Here again we have used Corollary~\ref{cor:sln-fourier-invariance}.)

Let $G^\blambda$ be the projective cover of $D^\blambda$ as a $\bk[\fS_{\sm(\rho)}]$-module, and let $\cG^\blambda$ and $\cG^{\blambda,M}$ denote the corresponding local systems on $Y_{(L_{e\rho},\cO^{L_{e\rho}}_{[e\rho]})}$ and $Y_{(L_{e\rho},\cO^{L_{e\rho}}_{[e\rho]})}^M$ respectively. Then $\IC(Y_{(L_{e\rho},\cO^{L_{e\rho}}_{[e\rho]})},\cG^{\blambda} \otimes\overline{\cE}^{L_{e\rho}}_{\chi})$ is the indecomposable direct summand of $\IC(Y_{(L_{e\rho},\cO^{L_{e\rho}}_{[e\rho]})},(\varpi_{(L_{e\rho},\cO^{L_{e\rho}}_{[e\rho]})})_* \widetilde{\cE}^{L_{e\rho}}_{\chi})$ with head $\IC(Y_{(L_{e\rho},\cO^{L_{e\rho}}_{[e\rho]})},\cD^{\blambda} \otimes\overline{\cE}^{L_{e\rho}}_{\chi})$. Since $\bT_{\fg}$ is an equivalence, there is an indecomposable direct summand $\mathcal{Q}^{\blambda}$ of $\Ind_{L_{e\rho}\subset P_{e\rho}}^G(\IC(\cO^{L_{e\rho}}_{[e\rho]},\cE^{L_{e\rho}}_\chi))$ such that 
\begin{equation} \label{eqn:fourier-q-first}
\bT_{\fg}(\mathcal{Q}^{\blambda})\cong\IC(Y_{(L_{e\rho},\cO^{L_{e\rho}}_{[e\rho]})},\cG^{\blambda} \otimes\overline{\cE}^{L_{e\rho}}_{\chi}),
\end{equation}
and the head of $\mathcal{Q}^{\blambda}$ is $\IC(\cO_{e\psi_{\chi,\rho}(\blambda)},\cE_{e\psi_{\chi,\rho}(\blambda),\chi})$. Similarly, there is an indecomposable direct summand $\mathcal{P}^{\blambda}$ of $\Ind_{L_{e\rho}\subset P_{e\rho}\cap M}^M(\IC(\cO^{L_{e\rho}}_{[e\rho]},\cE^{L_{e\rho}}_\chi))$ such that 
\begin{equation} \label{eqn:fourier-p-first}
\bT_{\fg}(\mathcal{P}^{\blambda})\cong\IC \bigl( Y_{(L_{e\rho},\cO^{L_{e\rho}}_{[e\rho]})}^M,\cG^{\blambda,M} \otimes (\overline{\cE}^{L_{e\rho}}_{\chi})^M \bigr),
\end{equation}
and the head of $\mathcal{P}^{\blambda}$ is $\IC(\cO_{e\ell^0(\lambda^{(\ell^0)})^\tr}\times\cO_{e\ell^1(\lambda^{(\ell^1)})^\tr}\times\cdots,
\cE_{e\ell^0(\lambda^{(\ell^0)})^\tr,e\ell^1(\lambda^{(\ell^1)})^\tr,\cdots;\chi})$.

In this case where $N_M(L_{e\rho})/L_{e\rho}=N_G(L_{e\rho})/L_{e\rho}$, Theorem~\ref{thm:uind-indec} says simply that
\begin{equation*}
\uInd_{M\subset Q}^G \bigl( \IC(Y_{(L_{e\rho},\cO^{L_{e\rho}}_{[e\rho]})}^M,\cG^{\blambda,M} \otimes (\overline{\cE}^{L_{e\rho}}_{\chi})^M) \bigr)
\cong
\IC(Y_{(L_{e\rho},\cO^{L_{e\rho}}_{[e\rho]})},\cG^{\blambda} \otimes\overline{\cE}^{L_{e\rho}}_{\chi}).
\end{equation*} 
Using~\eqref{eqn:fourier-q-first},~\eqref{eqn:fourier-p-first}, and~\cite[Corollary~2.10]{genspring1} we deduce that
\begin{equation} \label{eqn:proj-ind-first}
\Ind_{M\subset Q}^G(\mathcal{P}^\blambda)\cong\mathcal{Q}^\blambda.
\end{equation}
Since $\Ind_{M\subset Q}^G$ is exact,~\eqref{eqn:proj-ind-first} implies that 
\begin{equation} \label{eqn:long-ind}
\Ind_{M\subset Q}^G(\IC(\cO_{e\ell^0(\lambda^{(\ell^0)})^\tr}\times\cO_{e\ell^1(\lambda^{(\ell^1)})^\tr}\times\cdots,
\cE_{e\ell^0(\lambda^{(\ell^0)})^\tr,e\ell^1(\lambda^{(\ell^1)})^\tr,\cdots;\chi}))
\end{equation}
is a quotient of $\mathcal{Q}^\blambda$. But $\mathcal{Q}^\blambda$ has the simple head $\IC(\cO_{e\psi_{\chi,\rho}(\blambda)},\cE_{e\psi_{\chi,\rho}(\blambda),\chi})$, so we deduce that the induced perverse sheaf~\eqref{eqn:long-ind} surjects to $\IC(\cO_{e\psi_{\chi,\rho}(\blambda)},\cE_{e\psi_{\chi,\rho}(\blambda),\chi})$. The desired inequality~\eqref{eqn:sln-det-ineq} now follows from Lemma~\ref{lem:bounds},
since
\begin{equation*}
\mathrm{Ind}_M^G(\cO_{e\ell^0(\lambda^{(\ell^0)})^\tr}\times\cO_{e\ell^1(\lambda^{(\ell^1)})^\tr}\times\cdots)=\cO_{e\psico_\rho(\blambda)}
\end{equation*}
by~\cite[Lemma 7.2.5]{cm}.
\end{proof}

\subsection{Special orthogonal and symplectic groups in characteristic $2$}

In this subsection, we take $\ell = 2$.  Let $G=\G(N)$ where $\G$ stands for either $\SO$ or $\Sp$, and $N\geq 3$. In the $\Sp(N)$ case we assume, of course, that $N$ is even; in the $\SO(N)$ case we assume first that $N\not\equiv 0 \pmod 4$, and we will treat the case where $N\equiv 0 \pmod 4$ later.
Recall that there are no nontrivial $L$-equivariant irreducible local systems on nilpotent orbits for $L$ where $L$ is any Levi subgroup of $G$. For brevity, we will omit the trivial local system from the notation where possible.

We saw in Sections~\ref{sec:sp2n} and~\ref{sec:son} that the Levi subgroups admitting a cuspidal pair are those of the form
\[
L_\nu = \GL(\nu_1) \times \cdots \times \GL(\nu_m) \times \G(N - 2k),
\qquad 0 \le k \le \lfloor N/2\rfloor, \ \nu \in \Part(k,2),
\]
excluding the case $k=N/2-1$ when $\G=\SO$ and $N$ is even.
Let $P_\nu \subset \G(N)$ be a parabolic subgroup with $L_\nu$ as its Levi factor.  Let $W_\nu = N_{\G(N)}(L_\nu)/(L_\nu)$.  Recall that this is isomorphic either to $(\Z/2\Z) \wr \fS_{\sm(\nu)}$, or else (in certain cases in type $\mathbf{D}$) to an index-$2$ subgroup thereof.  In either case,
its irreducible representations in characteristic $2$ are parametrized by $\uPart_2(\sm(\nu))$: the irreducible $\bk$-representation of $W_\nu$ labelled by $\blambda\in\uPart_2(\sm(\nu))$ is obtained by pulling back the irreducible representation $D^{\blambda}$ of $\fS_{\sm(\nu)}$ through the projection $W_\nu\twoheadrightarrow \fS_{\sm(\nu)}$.

The orbits in $\cN_{L_\nu}$ supporting cuspidal pairs are those of the form
\[
\cO_{[\nu];\mu} := \cO_{(\nu_1)} \times \cdots \times \cO_{(\nu_m)} \times \cO_\mu, \qquad
\mu \in \Part_{2,\G}(N - 2k).
\]
Thus, the modular generalized Springer correspondence for $G$ can be regarded as a bijection
\begin{equation*} \label{eqn:sp2n2-genspring2}
\Omega: \bigsqcup_{0 \le k \le \lfloor N/2\rfloor} \  \bigsqcup_{\nu \in \Part(k,2)} \Part_{2,\G}(N-2k)\times\uPart_2(\sm(\nu)) \to \Part_\G(N).
\end{equation*}
(Here, we do not need to exclude the $k=N/2-1$ case when $\G=\SO$ and $N$ is even, because $\Part_{2,\SO}(2)$ is empty anyway.)

\begin{thm} \label{thm:sp2n2-det}
Let $G = \SO(N)$ with $N \not\equiv 0 \pmod 4$, or $G = \Sp(N)$ with $N$ even.
The modular generalized Springer correspondence for $G$ is given by
\[
\Omega=\bigsqcup_{0 \le k \le \lfloor N/2\rfloor} \  \bigsqcup_{\nu \in \Part(k,2)} \omegaco_{k,\nu},
\]
where $\omegaco_{k,\nu}:\Part_{2,\G}(N-2k)\times\uPart_2(\sm(\nu)) \to \Part_\G(N)$ is defined by
\[
\omegaco_{k,\nu}(\mu,\blambda)=\mu\cup\psico_{k,\nu}(\blambda)\cup\psico_{k,\nu}(\blambda).
\]
Here $\psico_{k,\nu}$ denotes the map $\psico_\nu$ of Lemma~{\rm \ref{lem:gln-comb}} with $k$ in place of $n$ \textup{(}and with $\ell=2$\textup{)}.
\end{thm}

\begin{proof}
Let $\Omegaco$ denote $\bigsqcup_{0 \le k \le \lfloor N/2\rfloor} \  \bigsqcup_{\nu \in \Part(k,2)} \omegaco_{k,\nu}$, a map with the same domain and codomain as $\Omega$. Combining the bijection of Lemma~\ref{lem:gln-comb} with that of~\eqref{eqn:sp2n2-bijection} (or its $\SO$ analogue), we see that $\Omegaco$ is a bijection. Hence 
it suffices to show that
\begin{equation} \label{eqn:sp2n2-inequality}
\omegaco_{k,\nu}(\mu,\blambda)\leq\omega_{k,\nu}(\mu,\blambda)
\end{equation} 
for all $k,\nu,\mu,\blambda$ as above, where $\omega_{k,\nu}$ denotes the restriction of $\Omega$ to the subset of the domain indexed by $k$ and $\nu$, and $\leq$ is the dominance order (corresponding to the closure order on nilpotent orbits for $G$).

We need to consider some Fourier transforms. Let $\mu\mapsto\mu^\dagger$ denote the involution of $\Part_{2,\G}(N-2k)$ defined by
\begin{equation*}
\bT_{\fg(N-2k)}(\IC(\cO_\mu))\cong\IC(\cO_{\mu^\dagger}).
\end{equation*}
(As mentioned in \S\ref{ss:series}, it is quite possible that $\mu^\dagger=\mu$ always, but we are not able to prove this.) 
Then using~\cite[(3.6)]{genspring1} we obtain
\begin{equation*}
\bT_{\fl_\nu}(\IC(\cO_{[\nu];\mu}))\cong\IC(\cO_{[\nu];\mu^\dagger}+\fz_{L_\nu}).
\end{equation*}
By the definition of the bijection $\Omega$, the simple perverse sheaf $\IC(\cO_{\omega_{k,\nu}(\mu,\blambda)})$ is a quotient of $\Ind_{L_\nu\subset P_\nu}^G(\IC(\cO_{[\nu];\mu}))$, namely the one with Fourier transform
\begin{equation*} \label{eqn:fourier-omega}
\bT_{\fg}(\IC(\cO_{\omega_{k,\nu}(\mu,\blambda)}))\cong\IC(Y_{(L_\nu,\cO_{[\nu];\mu^\dagger})},\cD^{\blambda}),
\end{equation*}
where $\cD^{\blambda}$ is the local system on $Y_{(L_\nu,\cO_{[\nu];\mu^\dagger})}$ corresponding, via the Galois covering $\varpi_{(L_\nu,\cO_{[\nu];\mu^\dagger})}$, to the irreducible representation $D^{\blambda}$ of $W_\nu$. (Here we are using Lemma~\ref{lem:trivial-case}.) 

Let $G^{\blambda}$ be the projective cover of $D^{\blambda}$ as a $\bk[W_\nu]$-module, and let $\cG^{\blambda}$ be the local system on $Y_{(L_\nu,\cO_{[\nu];\mu^\dagger})}$ corresponding to $G^{\blambda}$. 
As in the proof of Lemma~\ref{lem:hard-part},
there is an indecomposable direct summand $\mathcal{Q}^{\blambda}$ of $\Ind_{L_\nu\subset P_\nu}^G(\IC(\cO_{[\nu];\mu}))$ such that 
\begin{equation} \label{eqn:fourier-q}
\bT_{\fg}(\mathcal{Q}^{\blambda})\cong\IC(Y_{(L_\nu,\cO_{[\nu];\mu^\dagger})},\cG^{\blambda}),
\end{equation}
and the head of $\mathcal{Q}^{\blambda}$ is $\IC(\cO_{\omega_{k,\nu}(\mu,\blambda)})$.

Let $M_k$ denote the Levi subgroup of $G$ containing $L_\nu$ that has the form $M_k = \GL(k) \times \G(N - 2k)$, and let $Q_k$ be the parabolic subgroup of $G$ containing $P_{\nu}$ that has $M_k$ as Levi factor. We are going to apply the results of Section~\ref{sect:transitivity}, specifically Theorem~\ref{thm:uind-indec}, to the triple $L_\nu\subset M_k\subset G$ (the assumptions $(3)$ and $(4)$ are trivially true for the constant local system). Note that $N_{M_k}(L_\nu)/L_\nu \cong \fS_{\sm(\nu)}$. We identify $\cN_{M_k}$ with $\cN_{\GL(k)}\times\cN_{\G(N-2k)}$.

Since the $\G(N-2k)$ factor of $M_k$ plays no role in the induction $\Ind_{L_\nu\subset P_\nu \cap M_k}^{M_k}$, we know from~\cite[Theorem 3.4]{genspring1} that $\IC(\cO_{\psico_{k,\nu}(\blambda)}\times\cO_\mu)$ is a quotient of $\Ind_{L_\nu\subset P_\nu \cap M_k}^{M_k}(\IC(\cO_{[\nu];\mu}))$, namely the one with Fourier transform
\begin{equation*} \label{eqn:fourier-psi}
\bT_{\mathfrak{m}_k}(\IC(\cO_{\psico_{k,\nu}(\blambda)}\times\cO_\mu))\cong\IC(Y_{(L_\nu,\cO_{[\nu];\mu^\dagger})}^{M_k},\cD^{\blambda,M_k}),
\end{equation*}
where $\cD^{\blambda,M_k}$ is the local system on $Y_{(L_\nu,\cO_{[\nu];\mu^\dagger})}^{M_k}$ corresponding, via the Galois covering $\varpi_{(L_\nu,\cO_{[\nu];\mu^\dagger})}^{M_k}$, to the irreducible representation $D^{\blambda}$ of $\fS_{\sm(\nu)}$. 

Let $F^{\blambda}$ be the projective cover of $D^{\blambda}$ as a $\bk[\fS_{\sm(\nu)}]$-module, and let $\cF^{\blambda}$ be the local system on $Y_{(L_\nu,\cO_{[\nu];\mu^\dagger})}^{M_k}$ corresponding to $F^{\blambda}$. There is an indecomposable direct summand $\mathcal{P}^{\blambda}$ of $\Ind_{L_\nu\subset P_\nu \cap M_k}^{M_k}(\IC(\cO_{[\nu];\mu}))$ such that
\begin{equation} \label{eqn:fourier-p}
\bT_{\mathfrak{m}_k}(\mathcal{P}^\blambda)\cong\IC(Y_{(L_\nu,\cO_{[\nu];\mu^\dagger})}^{M_k},\cF^{\blambda}),
\end{equation}
and the head of $\mathcal{P}^{\blambda}$ is $\IC(\cO_{\psico_{k,\nu}(\blambda)}\times\cO_\mu)$.

Now since $W_\nu$ and $\fS_{\sm(\nu)}$ have the same irreducible $\bk$-representations, the induced representation $\mathrm{Ind}_{\fS_{\sm(\nu)}}^{W_\nu}(F^\blambda)$ is isomorphic to $G^\blambda$. So in this case Theorem~\ref{thm:uind-indec} says that
\begin{equation*} 
\uInd_{M_k\subset Q_k}^G\bigl(\IC(Y_{(L_\nu,\cO_{[\nu];\mu^\dagger})}^{M_k},\cF^{\blambda})\bigr)\cong
\IC(Y_{(L_\nu,\cO_{[\nu];\mu^\dagger})},\cG^{\blambda}).
\end{equation*}
Using~\eqref{eqn:fourier-q},~\eqref{eqn:fourier-p}, and~\cite[Corollary~2.10]{genspring1} we deduce that
\begin{equation} \label{eqn:proj-ind}
\Ind_{M_k\subset Q_k}^G(\mathcal{P}^\blambda)\cong\mathcal{Q}^\blambda.
\end{equation}
Since $\Ind_{M_k\subset Q_k}^G$ is exact,~\eqref{eqn:proj-ind} implies that $\Ind_{M_k\subset Q_k}^G(\IC(\cO_{\psico_{k,\nu}(\blambda)}\times\cO_\mu))$ is a quotient of $\mathcal{Q}^\blambda$; on the other hand, $\mathcal{Q}^\blambda$ has the simple head $\IC(\cO_{\omega_{k,\nu}(\mu,\blambda)})$. We conclude that  $\Ind_{M_k\subset Q_k}^G(\IC(\cO_{\psico_{k,\nu}(\blambda)}\times\cO_\mu))$ surjects to $\IC(\cO_{\omega_{k,\nu}(\mu,\blambda)})$.
The desired inequality~\eqref{eqn:sp2n2-inequality} now follows from Lemma~\ref{lem:bounds}, since
\begin{equation} \label{eqn:saturation}
G\cdot(\cO_{\psico_{k,\nu}(\blambda)}\times\cO_\mu)=\cO_{\omegaco_{k,\nu}(\blambda)}
\end{equation}
by definition of $\omegaco_{k,\nu}$.
\end{proof}

We now turn to the case where $G = \SO(N)$ with $N \equiv 0 \pmod 4$.  In this case, as we noted in Section~\ref{sec:son}, certain partitions correspond to more than one nilpotent orbit or conjugacy class of Levi subgroups.  To label these objects combinatorially, we will use partitions that may be decorated with a superscript Roman numeral $I$ or $II$, as in~\cite{cm}.  Let
\[
\Part(k, 2)' := \Part(k,2) \qquad\text{if $k < N/2$,}
\]
and
\begin{multline*}
\Part(N/2,2)' := \{ \nu \mid \nu \in \Part(N/2,2), \sm_1(\nu) \ne 0 \}\, \cup \\
\{\nu^I \mid \nu \in \Part(N/2,2), \sm_1(\nu) = 0\} \cup
\{\nu^{II} \mid \nu \in \Part(N/2,2), \sm_1(\nu) = 0\}.
\end{multline*}
Then the conjugacy classes of Levi subgroups of $G$ admitting a cuspidal pair are in bijection with the set $\bigsqcup_{\substack{0 \le k \le N/2\\k\neq N/2-1}} \Part(k,2)'$: in particular, for $\nu\in\Part(N/2,2)$ with $\sm_1(\nu)=0$, we have two representative Levi subgroups $L_{\nu^I}$ and $L_{\nu^{II}}$, both isomorphic to $\GL(\nu_1)\times\GL(\nu_2)\times\cdots$ but not $G$-conjugate to each other. (They are $\mathrm{O}(N)$-conjugate.)

To make the labelling consistent, we choose representatives $M_{N/2}^I$ and $M_{N/2}^{II}$ of the two $G$-conjugacy classes of Levi subgroups isomorphic to $\GL(N/2)$. Then for $\nu\in\Part(N/2,2)$ with $\sm_1(\nu)=0$, we specify that $L_{\nu^I}$ is contained in $M_{N/2}^I$ and $L_{\nu^{II}}$ is contained in $M_{N/2}^{II}$. Note that if $\nu\in\Part(N/2,2)$ and $\sm_1(\nu) \ne 0$, then a Levi subgroup $L_\nu$ of the corresponding conjugacy class is contained both in a conjugate of $M_{N/2}^I$ and in a conjugate of $M_{N/2}^{II}$.
 
Similarly, let
\begin{multline*}
\Part_\SO(N)' := (\Part_\SO(N) \setminus \Part_\ve(N))\, \cup \\
\{ \lambda^I \mid \lambda \in \Part_\ve(N) \} \cup
\{ \lambda^{II} \mid \lambda \in \Part_\ve(N) \}.
\end{multline*}
Then $\fN_{G,\bk}$ (or equivalently, the set of nilpotent orbits in $\cN_G$) is in bijection with $\Part_\SO(N)'$. For $\lambda\in\Part_\ve(N)$, we have two orbits $\cO_{\lambda^I}$ and $\cO_{\lambda^{II}}$ with the same Jordan type; we specify that $\cO_{\lambda^I}$ is the orbit that meets the Lie algebra of $M_{N/2}^I$ and $\cO_{\lambda^{II}}$ is the orbit that meets the Lie algebra of $M_{N/2}^{II}$. 

\begin{rmk}
To match the use of Roman numerals above with that in~\cite{cm}, choose $M_{N/2}^I$ and $M_{N/2}^{II}$ as in~\cite[Lem\-ma~7.3.2(ii)]{cm}.  With that choice, our labelling of orbits by Roman numerals is consistent with that in~\cite[Theorem~5.1.4 and Lemma~5.3.5]{cm}.  This claim can be worked out using the explicit description of orbit representatives in~\cite[Recipe~5.2.6]{cm}.  Alternatively, it follows from~\cite[Corollary~6.3.5, Theorem~7.3.3(iii), and Theorem~8.3.1]{cm}.
\end{rmk}

For convenience, in the formulas below we continue to use the notation $\nu$ for an arbitrary element of $\Part(N/2,2)'$ even though it may be a decorated partition; in that case, notation such as $\sm(\nu)$ should be interpreted using the underlying partition.

The modular generalized Springer correspondence for $G$ can thus be regarded as a bijection
\[
\Omega': \bigsqcup_{0 \le k \le N/2} \  \bigsqcup_{\nu \in \Part(k,2)'} \Part_{2,\SO}(N-2k)\times\uPart_2(\sm(\nu)) \to \Part_\SO(N)'.
\]
For $0\le k \le N/2$ and $\nu \in \Part(k,2)'$ we define a map
\[
\omegacop_{k,\nu}: \Part_{2,\SO}(N - 2k) \times \uPart_2(\sm(\nu)) \to \Part_\SO(N)'
\]
by the same formula as for $\omegaco_{k,\nu}$ in Theorem~\ref{thm:sp2n2-det}, with the following addendum: when $k = N/2$ and $\nu$ is decorated with a Roman numeral (forcing $\sm_1(\nu)=0$), the same Roman numeral should be used to decorate the output of this map (which necessarily belongs to $\Part_\ve(N)$).

\begin{thm}
\label{thm:so4n2-det}
Let $G = \SO(N)$ with $N \equiv 0 \pmod 4$. The modular generalized Springer correspondence for $G$ is given by
\[
\Omega'=\bigsqcup_{0 \le k \le  N/2} \  \bigsqcup_{\nu \in \Part(k,2)'} \omegacop_{k,\nu}.
\]
\end{thm}

\begin{proof}
The proof is essentially identical to that of Theorem~\ref{thm:sp2n2-det}, 
with the obvious proviso that when we define the subgroup $M_{N/2}$, in the case that $k=N/2$ and $\sm_1(\nu)=0$, we choose whichever of $M_{N/2}^{I}$ or $M_{N/2}^{II}$ matches the Roman numeral decoration on our given $\nu\in\Part(N/2,2)'$. This ensures that, when $\cO_{\psico_{N/2,\nu}(\blambda)}$ is interpreted as a nilpotent orbit in this $M_{N/2}$, its $G$-saturation is $\cO_{\omegacop_{N/2,\nu}(\blambda)}$, proving the analogue of~\eqref{eqn:saturation}.
\end{proof}

We deduce a description of the (un-generalized) modular Springer correspondence for $G=\SO(N)$ and $G=\Sp(N)$ in characteristic $\ell=2$, complementing the results of~\cite{jls} in the $\ell\neq 2$ case. Notice that the proofs of Theorems~\ref{thm:sp2n2-det} and~\ref{thm:so4n2-det} relied on the fact that we were dealing with bijections, so we needed to work with the full generalized correspondence in order to obtain this description.

\begin{cor} \label{cor:ordinary-springer}
The modular Springer correspondence for $G=\SO(N)$ or $G=\Sp(N)$ is the map 
\[ \Irr(\bk[N_G(T)/T])\to\fN_{G,\bk} \] 
described combinatorially by
\[
\Part_2(\lfloor N/2\rfloor)\to\Part_{G}(N):\lambda\mapsto\begin{cases}
\lambda^\tr\cup\lambda^\tr,&\text{ if $N$ is even,}\\
(1)\cup\lambda^\tr\cup\lambda^\tr,&\text{ if $N$ is odd.}
\end{cases}
\]
\end{cor}

\begin{proof}
This is obtained from Theorem~\ref{thm:sp2n2-det} or~\ref{thm:so4n2-det} by taking $k=\lfloor N/2\rfloor$ and $\nu=(1^k)$, so that $L_\nu$ is a maximal torus $T$. 
\end{proof}



\begin{thebibliography}{AHJR2}

\bibitem[AHR]{ahr}
P.~Achar, A.~Henderson and S.~Riche, \emph{Geometric Satake, Springer correspondence, and small representations II}, preprint, arXiv:1205.5089.

\bibitem[AHJR1]{ahjr}
P.~Achar, A.~Henderson, D.~Juteau and S.~Riche, \emph{Weyl group actions on the Springer sheaf}, Proc.\ Lond.\ Math.\ Soc.\ \textbf{108} (2014), no.~6, 1501--1528.

\bibitem[AHJR2]{genspring1}
P.~Achar, A.~Henderson, D.~Juteau and S.~Riche, \emph{Modular generalized Springer correspondence I: the general linear group}, preprint, arXiv:1307.2702, J.\ Eur.\ Math.\ Soc., to appear.

\bibitem[Bo1]{bonnafe-cgu}
C.~Bonnaf\'e, \emph{A note on centralizers of unipotent elements},
Ital.\ J.\ Pure Appl.\ Math.\ \textbf{16} (2004), 163--171. 

\bibitem[Bo2]{bonnafe1}
C.~Bonnaf\'e, \emph{Actions of relative Weyl groups. I}, J.\ Group Theory \textbf{7} (2004), no.~1, 1--37.

\bibitem[Bo3]{bonnafe2}
C.~Bonnaf\'e, \emph{Actions of relative Weyl groups. II}, J.\ Group Theory \textbf{8} (2005), no.~3, 351--387.

\bibitem[Bou]{bourbaki:algebre}
N.~Bourbaki, \emph{\'El\'ements de math\'ematique. Premi\`ere partie: Les structures fondamentales de l'analyse. Livre II: Alg\`ebre. Chapitre 9: Formes sesquilin\'eaires et formes quadratiques},
Actualit\'es Sci.\ Ind.\ no.~1272, Hermann, Paris, 1959.

\bibitem[CM]{cm}
D.~H. Collingwood and W.~M. McGovern, {\em Nilpotent orbits in semisimple {L}ie
  algebras}, Van Nostrand Reinhold Co., New York, 1993.
  
\bibitem[EM]{evens-mirk}
S.~Evens and I.~Mirkovi\'{c}, {\em Fourier transform and the Iwahori--Matsumoto
  involution}, Duke Math. J. {\bf 86} (1997), 435--464.
  
\bibitem[GH]{gh}
M.~Geck and G.~Hiss, {\em Modular representations of finite groups of Lie type in non-defining characteristic}, Finite reductive groups (Luminy, 1994), Progr.\ Math., vol.~141, Birkh\"auser Boston, Boston, MA, 1997, 195--249. 

\bibitem[GHM]{ghm}
M.~Geck, G.~Hiss and G.~Malle, {\em Towards a classification of the irreducible representations in non-describing characteristic of a finite group of Lie type}, Math.\ Z.\ \textbf{221} (1996), no.~3, 353--386.  
  
\bibitem[H]{howlett}
R.~B.~Howlett, {\em Normalizers of parabolic subgroups of reflection groups}, J.\ Lond.\ Math.\ Soc.\ (2) {\bf 21} (1980), no.~1, 62--80.  
  
\bibitem[I]{isaacs}
I.~M.~Isaacs, {\em Character theory of finite groups},
Pure and Applied Mathematics, no.~69, Academic Press [Harcourt Brace Jovanovich, Publishers], New York-London, 1976.

\bibitem[J]{juteau-aif}
D.~Juteau, {\em Decomposition numbers for perverse sheaves},
Ann. Inst. Fourier {\bf 59} (2009), 1177--1229.

\bibitem[JLS]{jls}
D.~Juteau, C.~Lecouvey, and K.~Sorlin, {\em Springer basic sets and
modular Springer correspondence for classical types}, preprint, 
arXiv:1410.1477.

\bibitem[Le]{letellier}
E.~Letellier, {\em Fourier transforms of invariant functions on finite
  reductive Lie algebras}, Lecture Notes in Math., vol.~1859,
  Springer-Verlag, Berlin, 2005.

\bibitem[Lu1]{lusztig}
G.~Lusztig, \emph{Intersection cohomology complexes on a reductive group},
Invent.~Math.~\textbf{75} (1984), 205--272.

\bibitem[Lu2]{lusztig-fourier}
G.~Lusztig, \emph{Fourier transforms on a semisimple Lie algebra over $\mathbb{F}_q$}, Algebraic groups (Utrecht 1986), Lecture Notes in Math., vol.~1271, Springer-Verlag, Berlin, 1987, 177--188.

\bibitem[Lu3]{lusztig-cusp2}
G.~Lusztig, \emph{Cuspidal local systems and graded Hecke algebras. II}, Representations of groups (Banff, AB, 1994), CMS Conf.~Proc., vol.~16, Amer.~Math.~Soc., Providence, RI, 1995, 217--275.

\bibitem[LS]{lus-spalt}
G.~Lusztig and N.~Spaltenstein, {\em On the generalized Springer correspondence
  for classical groups}, Algebraic groups and related topics (Kyoto/Nagoya,
  1983), Adv.\ Stud.\ Pure Math., vol.~6, North-Holland Publishing Co.,
  Amsterdam, 1985, 289--316.

\bibitem[S]{spaltenstein}
N.~Spaltenstein, {\em On the generalized Springer correspondence for exceptional groups},
Algebraic groups and related topics (Kyoto/Nagoya,
  1983), Adv. Stud. Pure Math., vol.~6, North-Holland Publishing Co.,
  Amsterdam, 1985, 317--338.

\end{thebibliography}
\end{document}